\documentclass[11pt]{amsart}
\usepackage{amsmath,amsfonts,mathrsfs}

\newif\ifdraft
\draftfalse
\usepackage{amsmath,amsfonts,amsthm,mathrsfs}
\usepackage{amssymb}

\usepackage[unicode,bookmarks]{hyperref}
\usepackage[usenames,dvipsnames]{xcolor}
\hypersetup{colorlinks=true,citecolor=NavyBlue,linkcolor=Brown,urlcolor=Orange}

\usepackage[alphabetic,initials]{amsrefs}

\usepackage{enumitem}

\usepackage{chngcntr}

\ifdraft
\usepackage[notcite,notref,color]{showkeys}

\definecolor{labelkey}{gray}{0.5}
\fi

\usepackage{tikz}

\usepackage{tikz-cd}

\usetikzlibrary{matrix,arrows}
\newlength{\myarrowsize} 

\pgfarrowsdeclare{cmto}{cmto}{
	\pgfsetdash{}{0pt} 
	\pgfsetbeveljoin 
	\pgfsetroundcap 
	\setlength{\myarrowsize}{0.6pt}
	\addtolength{\myarrowsize}{.5\pgflinewidth}
	\pgfarrowsleftextend{-4\myarrowsize-.5\pgflinewidth} 
	\pgfarrowsrightextend{.8\pgflinewidth}
}{
	\setlength{\myarrowsize}{0.6pt} 
  	\addtolength{\myarrowsize}{.5\pgflinewidth}  
	\pgfsetlinewidth{0.5\pgflinewidth}
	\pgfsetroundjoin
	\pgfpathmoveto{\pgfpoint{1.5\pgflinewidth}{0}}
	\pgfpatharc{-109}{-170}{4\myarrowsize}
	\pgfpatharc{10}{189}{0.58\pgflinewidth and 0.2\pgflinewidth}
	\pgfpatharc{-170}{-115}{4\myarrowsize+\pgflinewidth}
	\pgfpathclose
	\pgfusepathqfillstroke
	\pgfpathmoveto{\pgfpoint{1.5\pgflinewidth}{0}}
	\pgfpatharc{109}{170}{4\myarrowsize}
	\pgfpatharc{-10}{-189}{0.58\pgflinewidth and 0.2\pgflinewidth}
	\pgfpatharc{170}{115}{4\myarrowsize+\pgflinewidth}
	\pgfpathclose
	\pgfusepathqfillstroke
	\pgfsetlinewidth{2\pgflinewidth}
}

\pgfarrowsdeclare{cmonto}{cmonto}{
	\pgfsetdash{}{0pt} 
	\pgfsetbeveljoin 
	\pgfsetroundcap 
	\setlength{\myarrowsize}{0.6pt}
	\addtolength{\myarrowsize}{.5\pgflinewidth}
	\pgfarrowsleftextend{-4\myarrowsize-.5\pgflinewidth} 
	\pgfarrowsrightextend{.8\pgflinewidth}
}{
	\setlength{\myarrowsize}{0.6pt} 
  	\addtolength{\myarrowsize}{.5\pgflinewidth}  
	\pgfsetlinewidth{0.5\pgflinewidth}
	\pgfsetroundjoin
	\pgfpathmoveto{\pgfpoint{1.5\pgflinewidth}{0}}
	\pgfpatharc{-109}{-170}{4\myarrowsize}
	\pgfpatharc{10}{189}{0.58\pgflinewidth and 0.2\pgflinewidth}
	\pgfpatharc{-170}{-115}{4\myarrowsize+\pgflinewidth}
	\pgfpathclose
	\pgfusepathqfillstroke
	\pgfpathmoveto{\pgfpoint{1.5\pgflinewidth}{0}}
	\pgfpatharc{109}{170}{4\myarrowsize}
	\pgfpatharc{-10}{-189}{0.58\pgflinewidth and 0.2\pgflinewidth}
	\pgfpatharc{170}{115}{4\myarrowsize+\pgflinewidth}
	\pgfpathclose
	\pgfusepathqfillstroke
	\pgfpathmoveto{\pgfpoint{1.5\pgflinewidth-0.3em}{0}}
	\pgfpatharc{-109}{-170}{4\myarrowsize}
	\pgfpatharc{10}{189}{0.58\pgflinewidth and 0.2\pgflinewidth}
	\pgfpatharc{-170}{-115}{4\myarrowsize+\pgflinewidth}
	\pgfpathclose
	\pgfusepathqfillstroke
	\pgfpathmoveto{\pgfpoint{1.5\pgflinewidth-0.3em}{0}}
	\pgfpatharc{109}{170}{4\myarrowsize}
	\pgfpatharc{-10}{-189}{0.58\pgflinewidth and 0.2\pgflinewidth}
	\pgfpatharc{170}{115}{4\myarrowsize+\pgflinewidth}
	\pgfpathclose
	\pgfusepathqfillstroke
	\pgfsetlinewidth{2\pgflinewidth}
}

\pgfarrowsdeclare{cmhook}{cmhook}{
	\pgfsetdash{}{0pt} 
	\pgfsetbeveljoin 
	\pgfsetroundcap 
	\setlength{\myarrowsize}{0.6pt}
	\addtolength{\myarrowsize}{.5\pgflinewidth}
	\pgfarrowsleftextend{-4\myarrowsize-.5\pgflinewidth} 
	\pgfarrowsrightextend{.8\pgflinewidth}
}{
	\setlength{\myarrowsize}{0.6pt} 
  	\addtolength{\myarrowsize}{.5\pgflinewidth}  
 	\pgfsetdash{}{0pt}
	\pgfsetroundcap
	\pgfpathmoveto{\pgfqpoint{0pt}{-4.667\pgflinewidth}}
	\pgfpathcurveto
    {\pgfqpoint{4\pgflinewidth}{-4.667\pgflinewidth}}
    {\pgfqpoint{4\pgflinewidth}{0pt}}
    {\pgfpointorigin}
	\pgfusepathqstroke
}

\newenvironment{diagram*}[2]{%
\[%
\begin{tikzpicture}[>=cmto,baseline=(current bounding box.center),%
	to/.style={->,font=\scriptsize,cap=round},%
	into/.style={cmhook->,font=\scriptsize,cap=round},%
	onto/.style={-cmonto,font=\scriptsize,cap=round},%
	math/.style={matrix of math nodes, row sep=#2, column sep=#1,%
		text height=1.5ex, text depth=0.25ex}]%
}{%
\end{tikzpicture}%
\]%
\ignorespacesafterend%
}

\newcommand{\Dmod}{\mathscr{D}}

\newcommand{\Mmod}{\mathcal{M}}
\newcommand{\Nmod}{\mathcal{N}}

\newcommand{\derR}{\mathbf{R}}
\newcommand{\derL}{\mathbf{L}}

\newcommand{\shH}{\mathcal{H}}

\newcommand{\shExt}{\mathcal{E}\hspace{-1.5pt}\mathit{xt}}

\newcommand{\ZZ}{\mathbb{Z}}
\newcommand{\QQ}{\mathbb{Q}}

\newcommand{\CC}{\mathbb{C}}

\newcommand{\PP}{\mathbb{P}}

\DeclareMathOperator{\depth}{depth}

\newcommand*{\sheafhom}{\mathscr{H}\kern -.5pt om}

\newcommand{\shf}[1]{\mathscr{#1}}

\def\overbar#1#2#3{{%
	\setbox0=\hbox{\displaystyle{#1}}%
	\dimen0=\wd0
	\advance\dimen0 by -#2 
	\vbox {\nointerlineskip \moveright #3 \vbox{\hrule height 0.3pt width \dimen0}%
		\nointerlineskip \vskip 1.5pt \box0}%
}}

\newcommand{\onto}{\to\hspace{-0.7em}\to}

\newcommand{\shF}{\shf{F}}
\newcommand{\shG}{\shf{G}}
\newcommand{\shE}{\shf{E}}
\newcommand{\shN}{\shf{N}}

\newcommand{\shO}{\shf{O}}

\makeatletter
\let\@@seccntformat\@seccntformat
\renewcommand*{\@seccntformat}[1]{%
  \expandafter\ifx\csname @seccntformat@#1\endcsname\relax
    \expandafter\@@seccntformat
  \else
    \expandafter
      \csname @seccntformat@#1\expandafter\endcsname
  \fi
    {#1}%
}
\newcommand*{\@seccntformat@subsection}[1]{%
  \textbf{\csname the#1\endcsname.}
}
\makeatother

\makeatletter
\let\@paragraph\paragraph
\renewcommand*{\paragraph}[1]{%
	\vspace{0.3\baselineskip}%
	\@paragraph{\textit{#1}}%
}
\makeatother

\counterwithin{equation}{subsection}
\counterwithout{subsection}{section}
\counterwithin{figure}{subsection}

\newtheorem{theorem}[equation]{Theorem}
\newtheorem*{theorem*}{Theorem}
\newtheorem{lemma}[equation]{Lemma}
\newtheorem*{lemma*}{Lemma}
\newtheorem{corollary}[equation]{Corollary}
\newtheorem{proposition}[equation]{Proposition}
\newtheorem*{proposition*}{Proposition}
\newtheorem{conjecture}[equation]{Conjecture}

\theoremstyle{definition}
\newtheorem{definition}[equation]{Definition}
\newtheorem*{definition*}{Definition}
\newtheorem{remark}[equation]{Remark}

\newtheorem{question}[equation]{Question}

\newtheorem{example}[equation]{Example}
\newtheorem*{example*}{Example}
\newtheorem*{problem*}{Problem}

\theoremstyle{plain}

\newcommand{\theoremref}[1]{\hyperref[#1]{Theorem~\ref*{#1}}}
\newcommand{\lemmaref}[1]{\hyperref[#1]{Lemma~\ref*{#1}}}
\newcommand{\definitionref}[1]{\hyperref[#1]{Definition~\ref*{#1}}}
\newcommand{\propositionref}[1]{\hyperref[#1]{Proposition~\ref*{#1}}}
\newcommand{\conjectureref}[1]{\hyperref[#1]{Conjecture~\ref*{#1}}}
\newcommand{\corollaryref}[1]{\hyperref[#1]{Corollary~\ref*{#1}}}
\newcommand{\exampleref}[1]{\hyperref[#1]{Example~\ref*{#1}}}

\makeatletter
\let\old@caption\caption
\renewcommand*{\caption}[1]{%
	\setcounter{figure}{\value{equation}}%
	\stepcounter{equation}%
	\old@caption{#1}\relax%
}
\makeatother

\newcounter{intro}

\newtheorem{intro-conjecture}[intro]{Conjecture}
\newtheorem{intro-corollary}[intro]{Corollary}
\newtheorem{intro-theorem}[intro]{Theorem}

\newcommand{\shL}{\mathcal{L}}

\newcommand{\parref}[1]{\hyperref[#1]{\S\ref*{#1}}}

\makeatletter
\newcommand*\if@single[3]{%
  \setbox0\hbox{${\mathaccent"0362{#1}}^H$}%
  \setbox2\hbox{${\mathaccent"0362{\kern0pt#1}}^H$}%
  \ifdim\ht0=\ht2 #3\else #2\fi
  }
\newcommand*\rel@kern[1]{\kern#1\dimexpr\macc@kerna}
\newcommand*\widebar[1]{\@ifnextchar^{{\wide@bar{#1}{0}}}{\wide@bar{#1}{1}}}
\newcommand*\wide@bar[2]{\if@single{#1}{\wide@bar@{#1}{#2}{1}}{\wide@bar@{#1}{#2}{2}}}
\newcommand*\wide@bar@[3]{%
  \begingroup
  \def\mathaccent##1##2{%
    \if#32 \let\macc@nucleus\first@char \fi
    \setbox\z@\hbox{$\macc@style{\macc@nucleus}_{}$}%
    \setbox\tw@\hbox{$\macc@style{\macc@nucleus}{}_{}$}%
    \dimen@\wd\tw@
    \advance\dimen@-\wd\z@
    \divide\dimen@ 3
    \@tempdima\wd\tw@
    \advance\@tempdima-\scriptspace
    \divide\@tempdima 10
    \advance\dimen@-\@tempdima
    \ifdim\dimen@>\z@ \dimen@0pt\fi
    \rel@kern{0.6}\kern-\dimen@
    \if#31
      \overline{\rel@kern{-0.6}\kern\dimen@\macc@nucleus\rel@kern{0.4}\kern\dimen@}%
      \advance\dimen@0.4\dimexpr\macc@kerna
      \let\final@kern#2%
      \ifdim\dimen@<\z@ \let\final@kern1\fi
      \if\final@kern1 \kern-\dimen@\fi
    \else
      \overline{\rel@kern{-0.6}\kern\dimen@#1}%
    \fi
  }%
  \macc@depth\@ne
  \let\math@bgroup\@empty \let\math@egroup\macc@set@skewchar
  \mathsurround\z@ \frozen@everymath{\mathgroup\macc@group\relax}%
  \macc@set@skewchar\relax
  \let\mathaccentV\macc@nested@a
  \if#31
    \macc@nested@a\relax111{#1}%
  \else
    \def\gobble@till@marker##1\endmarker{}%
    \futurelet\first@char\gobble@till@marker#1\endmarker
    \ifcat\noexpand\first@char A\else
      \def\first@char{}%
    \fi
    \macc@nested@a\relax111{\first@char}%
  \fi
  \endgroup
}
\makeatother

\newcommand{\I}{\mathcal{I}}

\def\E{\mathcal{E}}
\def\J{{\mathcal J}}

\def\cH{{\mathcal H}}

\def\ZZ{{\mathbf Z}}

\def\CC{{\mathbf C}}
\def\AAA{{\mathbf A}}

\def\QQ{{\mathbf Q}}
\def\PP{{\mathbf P}}

\newtheorem*{thmA'}{Theorem~A$^\prime$}

\begin{document}

\vspace{\baselineskip}

\title[Hodge filtration on local cohomology]{Hodge filtration on local cohomology, du Bois complex, and local cohomological dimension}

\author[M. Musta\c{t}\u{a}]{Mircea~Musta\c{t}\u{a}}
\address{Department of Mathematics, University of Michigan, 530 Church Street,
Ann Arbor, MI 48109, USA}
\email{{\tt mmustata@umich.edu}}

\author[M.~Popa]{Mihnea~Popa}
\address{Department of Mathematics, Harvard University, 
1 Oxford Street, Cambridge, MA 02138, USA} 
\email{{\tt mpopa@math.harvard.edu}}

\thanks{MM was partially supported by NSF grants DMS-2001132 and DMS-1952399, and MP by NSF grant DMS-2040378.}


\begin{abstract}
We study the Hodge filtration on the local cohomology sheaves of a smooth complex algebraic variety along a closed subscheme $Z$ in terms of log resolutions, and derive applications regarding the local cohomological dimension, the Du Bois complex, local vanishing, and reflexive differentials associated to $Z$.
\end{abstract}

\maketitle

\makeatletter
\newcommand\@dotsep{4.5}
\def\@tocline#1#2#3#4#5#6#7{\relax
  \ifnum #1>\c@tocdepth 
  \else
    \par \addpenalty\@secpenalty\addvspace{#2}%
    \begingroup \hyphenpenalty\@M
    \@ifempty{#4}{%
      \@tempdima\csname r@tocindent\number#1\endcsname\relax
    }{%
      \@tempdima#4\relax
    }%
    \parindent\z@ \leftskip#3\relax
    \advance\leftskip\@tempdima\relax
    \rightskip\@pnumwidth plus1em \parfillskip-\@pnumwidth
    #5\leavevmode\hskip-\@tempdima #6\relax
    \leaders\hbox{$\m@th
      \mkern \@dotsep mu\hbox{.}\mkern \@dotsep mu$}\hfill
    \hbox to\@pnumwidth{\@tocpagenum{#7}}\par
    \nobreak
    \endgroup
  \fi}
\def\l@section{\@tocline{1}{0pt}{1pc}{}{\bfseries}}
\def\l@subsection{\@tocline{2}{0pt}{25pt}{5pc}{}}
\makeatother

\tableofcontents

\section{Introduction}

In this paper we prove several results about basic invariants of a closed subscheme $Z$ of a smooth, irreducible complex $n$-dimensional algebraic variety $X$. We give a characterization of the local cohomological dimension ${\rm lcd} (X, Z)$ in terms of coherent sheaf data associated to a log resolution of $(X, Z)$, complementing the celebrated topological criterion in  \cite{Ogus}. We also obtain local vanishing results for sheaves of forms with log poles associated to such a resolution, generalizing Nakano-type results in \cite{Saito-LOG} and \cite{MP1}. We prove a vanishing result for cohomologies of the graded pieces of the Du Bois complex when $Z$ is a local complete intersection, extending the study of higher Du Bois singularities of hypersurfaces in \cite{MOPW} and \cite{Saito_et_al}. When $Z$ has isolated singularities, we refine a result in \cite{KeS} on the coincidence of $h$-differentials and reflexive differentials, for forms of low degree. As a byproduct we also obtain new proofs of various results in the literature, for instance an injectivity theorem in \cite{KS}, or various cases of results related to local cohomology in \cite{MSS} and \cite{DT}.

The common theme leading to the proof of all these results is the study of the Hodge filtration on local cohomology. 
The local cohomology sheaves of $\shO_X$ along $Z$ are important and subtle invariants of the pair $(X,Z)$; while they are not coherent over $\shO_X$, they are well-behaved modules over the sheaf $\Dmod_X$ of differential operators. Exploiting this fact has been the key to important developments, especially in commutative algebra, starting with the foundational paper of Lyubeznik \cite{Lyubeznik}. Our focus in this paper is the fact that they have an even more refined structure, namely that of mixed Hodge modules; as such, they come endowed with a canonical Hodge filtration. We study this filtration on local cohomology and relate it to various invariants of $Z$ mentioned above.

\noindent
{\bf Local cohomology sheaves as mixed Hodge modules.}
We consider the local cohomology sheaves $\cH^q_Z(\shO_X)$, where $q$ is a positive integer; for a review of these objects, see 
\S\ref{scn:LC}.\footnote{Most of our results are local, hence while we will usually use the geometric language, they can also be seen as results about the modules $H^q_I (A)$, where $A$ is a regular $\CC$-algebra of finite type.} It is well understood that all $\cH^q_Z(\shO_X)$ carry the structure of 
(regular, holonomic) filtered $\Dmod_X$-modules underlying mixed Hodge modules on $X$, with support in $Z$; see \S\ref{scn:HF}. In particular, they come endowed with a good filtration
$F_k \cH^q_Z(\shO_X)$ by coherent subsheaves,  with $k \ge 0$, called the Hodge filtration. This data depends only on the reduced structure of $Z$.

When $Z$ is a hypersurface only $\cH^1_Z(\shO_X)$ is non-trivial, and in fact 
$$\cH^1_Z(\shO_X) \simeq \shO_X (*Z) / \shO_X,$$
where $\shO_X (*Z)$ is the sheaf of rational functions on $X$ with poles along $Z$. Hence the study of the Hodge filtration on  $\cH^1_Z(\shO_X)$ reduces to that of the Hodge filtration on $\shO_X(*Z)$, or equivalently to that of the Hodge ideals treated in \cite{MP1}. As in that paper, for concrete applications one essential point is to provide an alternative description of the Hodge filtration in terms of log resolutions. We discuss this next.

Suppose that $f\colon Y\to X$ is a log resolution of the pair $(X, Z)$, assumed to be an isomorphism over the complement of $Z$ in $X$. We denote $E = f^{-1}(Z)_{\rm red}$, which is a simple normal crossing divisor on $Y$. 
We observe in \S\ref{scn:birational} that there is a filtered complex of right $f^{-1} \Dmod_X$-modules 
$$A^{\bullet}:\quad
0\to f^*\Dmod_X\to \Omega_Y^1(\log E)\otimes_{\shO_Y}f^*\Dmod_X\to\cdots\to\omega_Y(E)\otimes_{\shO_Y}f^*\Dmod_X\to 0,$$
such that, restricting the discussion to $q\geq 2$ for simplicity, we have an isomorphism 
$$R^{q-1}f_*A^{\bullet} \simeq \cH^q_Z(\omega_X)\simeq \cH^q_Z(\shO_X)\otimes \omega_X.$$
Moreover, the Hodge filtration $F_\bullet \cH^q_Z(\omega_X)$ is obtained as the image of the push-forward of a 
natural filtration $F_\bullet A^\bullet$, also described in \S\ref{scn:birational}. This description parallels the birational definition of Hodge ideals of hypersurfaces in \cite{MP1}.

Once this birational description has been established, the main engine towards applications is the \emph{strictness} property of the Hodge filtration on direct images of Hodge modules via projective morphisms. This is a vast generalization of the 
$E_1$-degeneration of the Hodge-to-de Rham spectral sequence, established by Saito \cite{Saito-MHP}, \cite{Saito-MHM}.
Its main consequence to local cohomology is described in Proposition \ref{comput_log_res}; here we start by mentioning the most immediate application, namely an injectivity theorem.

Note first that, with the notation above, there is a natural morphism $\shO_Z \to \derR f_* \shO_E$ in 
${\bf D}^b \big({\rm Coh} (X)\big)$, which 
by duality gives rise to a morphism
$$\alpha \colon \derR f_* \omega_E^\bullet \to \omega_Z^\bullet,$$
where the notation refers to the respective dualizing complexes (for $E$ we of course have $\omega_E^\bullet = \omega_E [n -1]$). On the other hand, there is a morphism 
$$\beta \colon \omega_Z^\bullet \to \derR \underline{\Gamma_Z} (\omega_X) [n]$$
to the total (derived) local cohomology of $\omega_X$,  arising from the natural morphism $\derR \shH om_X (\shO_Z, \omega_X) \to \derR \underline{\Gamma_Z} (\omega_X)$. Here $\underline{\Gamma_Z} (-)$ denotes the sheaf version of the functor of sections with support in $Z$.

\begin{intro-theorem}\label{inclusions}
The morphism obtained as the composition
 $$\beta \circ \alpha \colon \derR f_* \omega_E^\bullet \to \derR \underline{\Gamma_Z} (\omega_X) [n]$$
 is injective on cohomology, i.e. the induced morphisms on cohomology give for each $q$ an injection 
$$R^{q - 1} f_* \omega_E \to \shH^q_Z (\omega_X).$$
\end{intro-theorem}

Since the morphism in the theorem factors through  $\omega_Z^\bullet$ via $\alpha$, this recovers in particular the following very useful result of Kov\'acs and Schwede:

\begin{intro-corollary}[{\cite[Theorem 3.3]{KS}}]\label{KS-injection}
For each $i$, the natural homomorphism
$$\shH^i  (\derR f_* \omega_E^\bullet) \to \shH^i (\omega_Z^\bullet)$$
is injective.
\end{intro-corollary}

As the authors explain in \emph{loc. cit.}, this can be thought of as a Grauert-Riemenschneider type result. To be more explicit, it says that for each $q \ge 1$, the natural morphism
$$\alpha_q \colon R^{q -1} f_* \omega_E \to {\mathcal Ext}^q_{\shO_X}(\shO_Z,\omega_X)$$
is injective. (In particular, if $Z$ is Cohen-Macaulay of pure codimension $r$, then
$R^q  f_* \omega_E= 0$ for $q \neq r-1$.)  We make use of this when studying the local cohomological dimension of $Z$ 
in terms of depth.

\medskip

For a hypersurface $Z$, one of the key tools in the study of the Hodge filtration $F_k \shO_X(*Z)$ is its containment in the pole order filtration $P_k \shO_X(*Z) = \shO_X\big((k+1)Z\big)$, as noted in \cite{Saito-B}. In arbitrary codimension, it is 
still the case that 
$$F_k \cH^q_Z(\shO_X) \subseteq O_k\cH^q_Z(\shO_X):=\{u\in \cH^q_Z(\shO_X)\mid \I_Z^{k+1}u=0\}$$
for all $k$, where $O_k$ is an \emph{order filtration} analogous to $P_k$. Unless $q = {\rm codim}_X(Z)$ however, 
work of Lyubeznik \cite{Lyubeznik} implies that the sheaves $O_k\cH^q_Z(\shO_X)$ are not coherent. A natural replacement seems to be an \emph{Ext filtration} defined as 
$$E_k\cH^q_Z(\shO_X):={\rm Im} ~\big[ \shE xt^q_{\shO_X} \big(\shO_X/ \I_Z^{k+1}, \shO_X\big) \to 
\cH_Z^q(\shO_X) \big],\,\,\,\,\,\, k \ge 0,$$
and satisfying $E_k \subseteq O_k$. Note that both $O_k\cH^q_Z(\shO_X)$ and $E_k\cH^q_Z(\shO_X)$ depend on the scheme-theoretic structure of $Z$
(we get the smallest version by taking $Z$ to be reduced).
When $Z$ is a local complete intersection of pure codimension $r$, then the two filtrations on 
$\cH^r_Z(\shO_X)$ coincide; we show that in this case they also coincide with the Hodge filtration if and only if $Z$ is smooth, see Corollary 
\ref{sm-equiv}.

In general, Theorem \ref{inclusions} and the birational interpretation of the Hodge filtration in \S\ref{scn:birational} imply that 
$$F_0 \cH^q_Z(\shO_X) \subseteq E_ 0\cH^q_Z(\shO_X)$$
for all $q$. Furthermore, as a combination of our results with a well-known characterization of Du Bois singularities (see \cite{Steenbrink}, \cite{Schwede}), we obtain:

\begin{intro-theorem}\label{char_DuBois}
Let $Z \subseteq X$ be a closed reduced subscheme of codimension $r$. If $Z$ is Du Bois, then 
$$F_0 \cH^q_Z(\shO_X)= E_0 \cH^q_Z(\shO_X) \,\,\,\,\, {\rm for ~all ~}q.$$
If we assume that $Z$ is Cohen-Macaulay, of pure dimension, then $F_0\cH^q_Z(\shO_X) = 0$ for all $q \neq r$, and 
$$ Z ~{\rm is~Du~Bois} \iff F_0 \cH^r_Z(\shO_X)= E_0 \cH^r_Z(\shO_X).$$
\end{intro-theorem}

It is a very interesting question whether $F_k \subseteq E_k$ for all $k \ge 1$, as the difference between the two should be a subtle measure of the singularities of $Z$ by analogy with the case of hypersurfaces.  This does happen when $Z$ is a local complete intersection, in which case Theorem \ref{thm-DB-main} provides a vast generalization of the last equivalence in the theorem above. Under this assumption we also expect that the equality $F_1 = E_1$ implies that $Z$ has rational singularities and an even stronger statement regarding the equality  $F_k = E_k$ for higher $k$; see 
Conjectures \ref{F_1=E_1} and \ref{conj-BS}.

\noindent
{\bf Local vanishing and local cohomological dimension.}
Since the ambient space $X$ is smooth, it is well known that the lowest index $q$ for which $\shH^{q}_Z (\shO_X)\neq 0$
is equal to the codimension of $Z$. On the other hand, the highest such index is a more mysterious and much studied invariant, the \emph{local cohomological dimension} of $Z$ in $X$:
\begin{equation}\label{def_local_cohom_dim}
{\rm lcd}(X, Z) = {\rm max}~\{q~|~\cH_Z^q(\shO_X) \neq 0\}.
\end{equation}
(Note that ${\rm lcd}(X, Z)$ is more commonly defined as the minimal integer $q$ such that $\cH_Z^i(\shF) = 0$ for all $i > q$ and all quasi-coherent sheaves $\shF$ on $X$, but the two definitions agree; see e.g. \cite[Proposition 2.1]{Ogus}.) Our study of the Hodge filtration allows us to provide a new perspective on ${\rm lcd}(X, Z)$, by relating it to invariants arising from log resolutions.

Note to begin with that thanks to the Grauert-Riemenschneider theorem, we have 
$$f_* \omega_E \simeq f_* \omega_Y (E) / \omega_X \,\,\,\,\,\,{\rm and} \,\,\,\,\,\, R^q f_* \omega_E \simeq R^q f_* \omega_Y (E) \,\,\,\,\,\,{\rm for}\,\,\,\,\,\, q \ge 1,$$
hence we can alternatively think of Theorem \ref{inclusions} as a result about the higher direct images 
$R^q f_* \omega_Y (E)$. For instance, they must vanish when  $\shH^{q+1}_Z (\shO_X)= 0$ and $q \ge 1$.
Using the same circle of ideas, based on the strictness property and the birational description of the Hodge filtration, this can be extended to a Nakano-type vanishing result for the higher direct images of all bundles of forms with log poles along $E$.  

\begin{intro-theorem}\label{local-vanishing}
Let $Z \subseteq X$ be a closed subscheme of codimension $r$, and let 
$c: = {\rm lcd}(X, Z)$. If $f\colon Y \to X$ 
is a log resolution of $(X, Z)$ which is an isomorphism away from $Z$, and $E = f^{-1}(Z)_{\rm red}$, then
$$R^qf_* \Omega_Y^p (\log E) = 0$$ 
in either of the following two cases:
\begin{enumerate}
\item[i)] $p+q\geq n+1$ and $q\leq r-2$;
\item[ii)] $p+q\geq n+c$.
\end{enumerate}
\end{intro-theorem}

In the case when $Z$ is a hypersurface (so that $c =1$), this is a result of Saito \cite[Corollary~3]{Saito-LOG}; cf. also \cite[Theorem~32.1]{MP1}. We remark that the part of the statement saying that vanishing holds whenever $p + q \ge n + c$  can also be obtained as a consequence of Theorem \ref{thm:lcd} below. Various bounds on the local cohomological dimension that are relevant to this theorem can be found in \S\ref{scn:local-vanishing}.
At least for local complete intersections, the range of vanishing in Theorem \ref{local-vanishing} could perhaps be further improved by analogy with \cite[Theorem~D]{MP3}, though this will rely on connections with the Bernstein-Sato polynomial of $Z$ not known at the moment; cf. Remark \ref{rmk:BS}.

As one of the most important applications of the techniques in this paper, we go in the opposite direction and obtain a 
characterization of the local cohomological dimension ${\rm lcd}(X, Z)$ in terms of the vanishing of sheaves of the form 
$R^q f_* \Omega_Y^p (\log E)$ associated to a log resolution.

\begin{intro-theorem}\label{thm:lcd}
Let $Z$ be a closed subscheme of $X$, and $c$ a positive integer. Then the following are equivalent:
\begin{enumerate}
\item ${\rm lcd}(X, Z) \le c$.
\item For any (some) log resolution $f\colon Y\to X$ of the pair $(X, Z)$, assumed to be an isomorphism over the complement of $Z$ in $X$, if $E = f^{-1}(Z)_{\rm red}$, then
$$R^{j + i} f_* \Omega_Y^{n-i} (\log E) = 0, \,\,\,\,\,\,{\rm for ~all}\,\,\,\,j \ge c, \, i \ge 0.$$
\end{enumerate}
\end{intro-theorem}

Over our base field $\CC$, this provides an alternative algebraic criterion in terms of finitely many coherent sheaves, complementing Ogus' celebrated topological criterion \cite[Theorem 2.13]{Ogus}, and answering in particular a problem raised there. The equivalence between Ogus' criterion and ours seems unclear at the moment, and is an intriguing topic of study; see Remark \ref{rmk:Ogus} for further discussion. In \S\ref{scn:lcd} we also give concrete applications of this characterization; more on this below as well. 

\noindent
{\bf Note:} While throughout this paper we focus on the case of algebraic varieties, it is worth noting that the criterion in 
Theorem \ref{thm:lcd} holds in the \emph{analytic setting} as well, i.e. when $Z$ is an analytic subspace of a complex manifold $X$. The same holds for Theorem \ref{local-vanishing}. This is due to the fact that our constructions and arguments based on the theory of mixed Hodge modules apply equally well in this setting; see Remarks \ref{rmk:analytic1} and 
\ref{rmk:analytic2}.

The proof of the theorem relies on a simple strategy involving the Hodge filtration, namely showing, under the appropriate hypotheses, that:
\begin{enumerate}
\item $F_0 \cH^q_Z (\shO_X) = 0$, and
\item $F_\bullet \cH^q_Z (\shO_X)$ is generated at level $0$.
\end{enumerate}
The second condition essentially means that the entire Hodge filtration is determined by the initial term $F_0 \cH^q_Z (\shO_X)$, up to applying differential operators. The key technical tool is therefore a local vanishing criterion for the generation level of the Hodge filtration, in the style of \cite[Theorems 17.1]{MP1} in the case of hypersurfaces; this is stated as Theorem \ref{gen_level} below.

Numerous works have studied bounds on the local cohomological dimension in terms of the depth of the local rings at 
points of $Z$. Theorem \ref{thm:lcd} leads to a unified approach to previously known such bounds (for example some of the statements in \cite{Hartshorne}, \cite{Ogus}, \cite{Varbaro}, \cite{DT}), as well as to new results. Among the latter, we show
in Corollary \ref{cor:quotient} that if $Z$ has \emph{quotient singularities} and codimension $r$, then 
$${\rm lcd}(X, Z) = n - {\rm depth}(\shO_Z) = r.$$
Due to results of Ogus, this in turn implies that subvarieties $Z \subseteq \PP^n$ with quotient singularities behave like local complete intersections in other respects as well, for instance satisfying a Barth-Lefschetz-type result (see Corollary \ref{like-LCI}). We refrain from including more material here; for further details and examples see \S\ref{scn:lcd}.

\noindent
{\bf The Du Bois complex and differentials on singular spaces.}
Our results regarding the Hodge filtration on local cohomology, and perhaps somewhat surprisingly the characterization of local cohomological dimension in Theorem \ref{thm:lcd}, can be applied to the study of the 
Du Bois complex and of various types of differentials on a reduced closed subscheme $Z \subseteq X$.

Recall that the \emph{Du Bois complex} $\underline{\Omega}_Z^{\bullet}$ is an object in the derived category of filtered complexes on $Z$. The shifted graded pieces $\underline{\Omega}_Z^p:={\rm Gr}^p_F\underline{\Omega}_Z^{\bullet}[p]$ are objects in the derived category of coherent sheaves on $Z$, playing a role similar to that of the bundles of holomorphic forms $\Omega_Z^p$ on a smooth $Z$. There are in fact canonical morphisms $\Omega_Z^p\to \underline{\Omega}_Z^p$ that are isomorphisms over the smooth locus of $Z$. By definition, 
the one for $p=0$ is an isomorphism precisely when
$Z$ has Du Bois singularities. See Ch.\ref{ch:DuBois} for more details.

Following the terminology from \cite{Saito_et_al}, we say that $Z$ has only \emph{higher $p$-Du Bois singularities} if the canonical morphisms $\Omega_Z^k\to \underline{\Omega}_Z^k$ are isomorphisms for all $0\leq k\leq p$. The first result concerning varieties with this property was obtained in \cite{MOPW}, where is was shown that if $Z$ is a hypersurface whose minimal exponent is $\geq p+1$, then $Z$ has only higher
$p$-Du Bois singularities; recall that the minimal exponent of $Z$, which can be defined via the Bernstein-Sato polynomial of $Z$, roughly describes how close the Hodge filtration and pole order filtration are on the localization $\shO_X(*Z)$. The converse to this result was obtained in
\cite{Saito_et_al}. The Hodge filtration on local cohomology allows us to extend these results to all local complete intersections.

Concretely, if $Z$ is reduced and a local complete intersection of pure codimension $r$, then the \emph{singularity level} of the Hodge filtration on $\shH^r_Z \shO_X$ is 
$$p (Z) := {\rm sup}\{k ~| ~ F_k \shH^r_Z \shO_X = O_k \shH^r_Z \shO_X \},$$
with the convention that $p (Z) = -1$ if there are no such $k$. 
We show that this invariant only depends on $Z$ and not on its embedding in a smooth variety. It is easy to check that 
$p(Z) = \infty$ if and only if $Z$ is smooth (see Corollary \ref{sm-equiv}). In fact, we have the following explicit upper bound
for singular $Z$ (see Theorem~\ref{thm_upper_bound}):
\begin{equation}\label{eq_upper_bound}
p(Z)\leq \frac{\dim(Z)-1}{2}.
\end{equation}

By a result of Saito \cite{Saito-MLCT}, if  $Z$ is a hypersurface in $X$, then $p(Z)=[ \widetilde{\alpha}(Z)]-1$, where
$\widetilde{\alpha}(Z)$ is the minimal exponent of $Z$.
We expect that in general, $p(Z)$ can be described in terms of the Bernstein-Sato polynomial of $Z$ studied in \cite{BMS} 
(see Conjecture \ref{conj-BS} for the statement).  An interpretation of $p(Z)$ in terms of the Hodge ideals associated to products of equations defining $Z$ is given by Proposition \ref{charact_equality}, leading to restriction and semicontinuity results for this invariant; see Theorems \ref{thm_restriction} and \ref{thm_semicont}. The proof of (\ref{eq_upper_bound}) makes use of this semicontinuity property of $p(Z)$.

The following is our main result, relating $p(Z)$ to the behavior of the Du Bois complex of $Z$. The proof builds on the case of hypersurfaces which, as already mentioned, is treated in \cite{MOPW} and \cite{Saito_et_al}.

\begin{intro-theorem}\label{thm-DB-main}
If $Z$ is a reduced, local complete intersection closed subscheme of the smooth, irreducible variety $X$, then for every nonnegative integer $p$, we have
$p(Z)\geq p$ if and only if $Z$ has only higher $p$-Du Bois singularities.
\end{intro-theorem}

We also prove a related result concerning the vanishing of individual cohomology sheaves $\shH^i \underline{\Omega}_Z^p$, 
 with $i > 0$, in terms of the size of the locus in $Z$ where $p(Z)<p$ (see 
 Theorem~\ref{thm_van_DB} for the precise statement). A consequence is that
 if the singular locus of the local complete intersection $Z$ has dimension $s$, then for all $p \ge 0$ we have
$$\cH^i(\underline{\Omega}_Z^p) =0 \,\,\,\,\,\,{\rm for} \,\,\,\,\,\,1\leq i < \dim Z - s - p -1.$$


\smallskip

In a different direction, the criterion in Theorem \ref{thm:lcd} can be rephrased in terms of the Du Bois complex (see Corollary \ref{cor:lcd}), thanks to a result of Steenbrink \cite{Steenbrink}. This allows us to obtain in \S\ref{scn:hdiff} results on differentials on the singular variety $Z$ as consequences of bounds on ${\rm lcd}(X, Z)$. We state here one result regarding 
$h$-differentials; their theory is one of the possible approaches to differential forms on singular spaces, as explained in  \cite{HJ}, where it is also shown that they are isomorphic to $\shH^0 \underline{\Omega}_Z^k$.

In the ideal situation, $h$-differentials coincide with the reflexive differentials $\Omega_Z^{[k]} : = (\Omega_Z^k)^{\vee \vee}$, and a recent result of Kebekus-Schnell \cite[Corollary 1.12]{KeS} states that this is indeed the case for all $k$ if $Z$ is a variety with rational singularities. We show the following improvement for low $k$ when $Z$ has isolated singularities, using in addition input from mixed Hodge theory:

\begin{intro-theorem}\label{thm:hisolated}
If $Z$ is a variety with isolated singularities and ${\rm depth} (\shO_Z) \ge k +2$, then the $h$-differentials and reflexive differentials of degree $k$ on $Z$ coincide.
\end{intro-theorem}

Further applications of results on local cohomological dimension to $h$-differentials are obtained in \S\ref{scn:hdiff}, including a statement analogous to Theorem \ref{thm:hisolated} but depending only on the codimension of $Z$.

There are numerous conjectures and open problems suggested by this work that are scattered throughout the text, and 
that we believe are important for further developments. Here is an informal sample: the connection between the Hodge filtration on local cohomology and the Bernstein-Sato polynomial (and perhaps a version of the $V$-filtration) for local complete intersections, including applications to rational singularities; the equivalence between our and Ogus' characterization of local cohomological dimension; further local vanishing with applications to local cohomological dimension and reflexive differentials.

\medskip

\noindent
{\bf Acknowledgements.} We are indebted to C.~Raicu, whose answers to our questions were crucial for the material 
in \S\ref{scn:order}, and to C.~Schnell for many discussions on mixed Hodge modules. We thank 
S.-J. Jung, I.-K. Kim, M. Saito, and Y.~Yoon for sharing with us an early
version of \cite{Saito_et_al}. We also thank B.~Bhatt, L.~Ma, M.~Saito,  K.~Schwede and J.~Witaszek for their comments on a preliminary version of this paper. In particular, a suggestion of M.~Saito led us to a simplification of the proof Theorem~\ref{thm:lcd}. Finally, we thank the referee for comments that helped improve the exposition.

\medskip

\noindent {\bf Notation}. We collect here some notation that appears throughout the paper, with references to where each item is described. Note that $X$ will always be a smooth, irreducible complex variety of dimension $n$, and $Z$ a closed subscheme of $X$.
\begin{enumerate}

\item ${\rm lct}(X, Z)$, the local cohomological dimension of $Z$ in $X$, see formula (\ref{def_local_cohom_dim}). 

\item $\Dmod_X$, the sheaf of differential operators on  $X$.

\item ${\mathbf Q}_X^H[n]$, the trivial pure Hodge module on  $X$, see Section~\ref{scn:Dmod}.

\item $(\Mmod,F)(q)$, the Tate twist by $q$ of the filtered $\Dmod_X$-module $(\Mmod,F)$, see Section~\ref{scn:Dmod}.

\item ${\rm Gr}^F_k{\rm DR}_X(\Mmod, F)$, the $k^{\rm th}$ graded piece of the de Rham complex of the filtered $\Dmod_X$-module $(\Mmod,F)$, see Section~\ref{scn:Dmod}.

\item $\cH^q_Z(\Mmod)$, the $q^{\rm th}$ local cohomology of $\Mmod$ with support in $Z$, see Section~\ref{scn:LC}.

\item $\shO_X(*Z)$, the sheaf of rational functions on $X$ with poles along $Z$, when $Z$ is a hypersurface, see Example~\ref{ex:divisor}.

\item $F_k\cH^q_Z(\shO_X)$, the $k^{\rm th}$ term of the Hodge filtration on $\cH^q_Z(\shO_X)$, see Section~\ref{scn:HF}.

\item $O_k\cH^q_Z(\shO_X)$, the $k^{\rm th}$ term of the order filtration on $\cH^q_Z(\shO_X)$, see Definition~\ref{definition_order_filtration}.

\item $E_k\cH^q(\shO_X)$, the $k^{\rm th}$ term of the Ext filtration on $\cH^q_Z(\shO_X)$, see Definition~\ref{definition_Ext_filtration}.

\item $p(Z)$, the singularity level of the Hodge filtration on $\cH^r_Z(\shO_X)$, when $Z$ is a local complete intersection of codimension $r$, see
Definition~\ref{definition_singularity_level}.

\item $\Omega_Z^p$, the sheaf of $p$-K\"{a}hler differentials on $Z$.

\item $\Omega_Z^{[k]}=(\Omega_Z^k)^{\vee\vee}$, the sheaf of reflexive $p$-K\"{a}hler differentials on $Z$.

\item $\underline{\Omega}_Z^p$, the (shifted) $p^{\rm th}$ truncation of the Du Bois complex, see Section~\ref{ch:DuBois}.

\end{enumerate}

\section{Background and study of the Hodge filtration}

\subsection{$\Dmod$-modules and mixed Hodge modules}\label{scn:Dmod}
Given a smooth, irreducible, $n$-dimensional complex algebraic variety $X$, we denote by $\Dmod_X$ the sheaf of differential operators on $X$. For basic facts in the theory of $\Dmod_X$-modules, we refer to \cite{HTT}. 

We only recall here a couple of things: first, an exhaustive filtration $F_{\bullet}$ on a left $\Dmod_X$-module $\Mmod$ (always assumed to be compatible with the filtration $F_\bullet \Dmod_X$ by the order of differential operators)
is \emph{good} if $F_p\Mmod$ is coherent for all $p$ and there is $q$ such that
$$F_{q+k}\Mmod=F_k\Dmod_X\cdot F_q\Mmod\quad\text{for all}\quad k\geq 0.$$
In this case, we say that the filtration is \emph{generated at level} $q$. 

Next, there is an equivalence of categories between left and right $\Dmod_X$-modules, such that if $\Mmod^r$ is the right $\Dmod_X$-module corresponding to the left $\Dmod_X$-module $\Mmod$,
we have an isomorphism of underlying $\shO_X$-modules $\Mmod^r\simeq\omega_X\otimes_{\shO_X}\Mmod$. 
We will mostly work with left $\Dmod_X$-modules, but it will sometimes be convenient to work with their right counterparts. We note that if they are endowed with a good filtration as above, then the left-right rule for the filtration is by convention 
\begin{equation}\label{eq_left_right_shift}
F_{p-n}\Mmod^r=\omega_X\otimes_{\shO_X}F_p\Mmod\quad\text{for all}\quad p\in {\mathbf Z}.
\end{equation}

\smallskip

For the basic notions and results on mixed Hodge modules we refer to \cite{Saito-MHP} and \cite{Saito-MHM}. 
In what follows, we refer to a mixed Hodge module on $X$ simply as a \emph{Hodge module}.
Recall that such an object consists of a $\Dmod_X$-module on $X$ (always assumed to be holonomic and with regular singularities), endowed with a good filtration called the \emph{Hodge filtration}, and with several other pieces of data 
satisfying an involved set of conditions that we will not be directly concerned with in this paper. We will say that this filtered 
$\Dmod_X$-module \emph{underlies} the respective Hodge module. 
An important fact is that every morphism of Hodge modules is strict with respect to the Hodge filtrations. See also 
\cite[Ch. C]{MP1} and \cite[\S1]{MP5} for further review.


An important example of (pure) Hodge module is  the trivial Hodge module ${\mathbf Q}_X^H[n]$. The underlying 
$\Dmod_X$-module is $\shO_X$, while the Hodge filtration is defined by the condition ${\rm Gr}^F_p(\shO_X)=0$ for all $p\neq 0$. 

Given a filtered $\Dmod_X$-module $(\Mmod,F)$ and $q\in {\mathbf Z}$, the \emph{Tate twist} $(\Mmod,F)(q)$ is the filtered 
$\Dmod_X$-module $\big(\Mmod, F[q]\big)$, where 
$$F[q]_p\Mmod=F_{p-q}\Mmod \,\,\,\,\,\,{\rm  for~ all} \,\,\,\, p\in {\mathbf Z}.$$

For every left $\Dmod_X$-module $\Mmod$, the de Rham complex ${\rm DR}_X(\Mmod)$ is the complex
$$0\to \Mmod\to \Omega_X^1\otimes_{\shO_X}\Mmod\to\cdots\to\omega_X\otimes_{\shO_X}\Mmod\to 0,$$
placed in cohomological degrees $-n,\ldots,0$. If $\Mmod$ carries a good filtration $F$, the de Rham complex has an induced filtration
such that for every integer $k$, the graded piece ${\rm Gr}^F_k{\rm DR}_X(\Mmod,F)$ is given by
$$0\to {\rm Gr}^F_k \Mmod \to \Omega_X^1\otimes_{\shO_X}{\rm Gr}^F_{k+1} \Mmod \to\cdots\to\omega_X\otimes_{\shO_X}{\rm Gr}^F_{k+n} \Mmod \to 0,$$
where ${\rm Gr}^F_i \Mmod =F_i\Mmod/F_{i-1}\Mmod$ for all $i\in {\mathbf Z}$.
Note that this is a complex of coherent $\shO_X$-modules. The filtered de Rham complex of a filtered right $\Dmod_X$-module is the 
filtered de Rham complex of the corresponding left $\Dmod_X$-module. When $(\Mmod, F)$ underlies a Hodge module $M$, we sometimes use alternatively the notation ${\rm DR}_X(M)$ and ${\rm Gr}^F_p{\rm DR}_X(M)$.

Since morphisms of Hodge modules are strict with respect to the Hodge filtration, ${\rm Gr}^F_k{\rm DR}_X(-)$ gives an exact functor from the abelian category of Hodge modules on $X$ to the abelian category of bounded complexes of coherent sheaves on $X$. This induces an exact functor, also denoted 
${\rm Gr}^F_k{\rm DR}_X(-)$,  from the derived category ${\bf D}^b\big({\rm MHM}(X)\big)$ of Hodge modules on $X$ to the derived category
${\bf D}^b\big({\rm Coh}(X)\big)$. By considering the truncation functors associated to the standard $t$-structure on ${\bf D}^b\big({\rm MHM}(X)\big)$, one obtains for every $u\in {\bf D}^b\big({\rm MHM}(X)\big)$ and every $k\in {\mathbf Z}$ a spectral sequence
\begin{equation}\label{eq_spec_seq_DR}
E_2^{pp'}={\mathcal H}^p{\rm Gr}^F_k{\rm DR}_X\big({\mathcal H}^{p'}(u)\big)\Rightarrow {\mathcal H}^{p+p'}{\rm Gr}^F_k{\rm DR}_X(u).
\end{equation}

Given a morphism $f\colon Y\to X$ of smooth complex varieties, we use the notation $f_*$ for the push-forward of Hodge modules. Note that at the level of underlying $\Dmod$-modules, this corresponds to the usual push-forward $f_+$, defined for right $\Dmod$-modules as 
$$ f_+ \Mmod: = \derR f_* \big(\Mmod \overset{\derL}{\otimes}_{\Dmod_Y} \Dmod_{Y\to X} \big),$$
where the object on the right is in the derived category of right $\Dmod_X$-modules. Here $\Dmod_{Y\to X} : = \shO_Y \otimes_{f^{-1} \shO_X} f^{-1} \Dmod_X$ is the associated transfer $(\Dmod_Y, f^{-1} \Dmod_X)$-bimodule, which 
is isomorphic to $f^* \Dmod_X$ as an $\shO_Y$-module, and is filtered by 
$f^* F_k \Dmod_X$. See \cite[\S1.5]{HTT} for more details.  In general, $f_+$ is different from the push-forward $\derR f_*$
on quasi-coherent $\shO$-modules, but the two definitions agree if $f$ is an open immersion, and in this case we also use the $\derR f_*$  notation.

Moreover, if $f$ is proper and if we denote by ${\rm FM}(\Dmod_X)$ the category of filtered $\Dmod$-modules on $X$,  then
there is a functor
$$f_+ \colon {\bf D}^b \big({\rm FM}(\Dmod_Y)\big) \rightarrow {\bf D}^b \big({\rm FM}(\Dmod_X)\big)$$
defined in \cite{Saito-MHP}, 
which is compatible with the functor $f_+$ above. 
If $(\Mmod,F)$ underlies a Hodge module $M$ on $X$,  we write $f_+(\Mmod,F)$ for the object in ${\bf D}^b\big({\rm FM}(\Dmod_Y)\big)$ underlying $f_* M$.

An important feature of the push-forward of Hodge modules under projective morphisms is the \emph{strictness} property of the Hodge filtration; see \cite[Theorem~2.14]{Saito-MHM}. 
This says that if $f\colon Y\to X$ is projective and $(\Mmod,F)$ underlies a Hodge module on $Y$, then 
$f_+ (\Mmod, F)$ is strict as an object in ${\bf D}^b \big({\rm FM}(\Dmod_X)\big)$ (and moreover, each $\cH^i f_+ (\Mmod, F)$ underlies a Hodge module). Concretely, this means that the natural mapping 
\begin{equation}\label{strictness_formula}
R^i f_* \big(F_k (\Mmod \overset{\derL}{\otimes}_{\Dmod_Y} \Dmod_{Y\to X}) \big) \longrightarrow 
R^i f_* (\Mmod \overset{\derL}{\otimes}_{\Dmod_Y} \Dmod_{Y\to X})
\end{equation}
is injective for every $i, k\in\ZZ$. The filtration on $\cH^if_+(\Mmod,F)$ is obtained by taking $F_k\cH^if_+(\Mmod,F)$ to be the image of this map. Note that this strictness property is a vast generalization of the degeneration at $E_1$ of the Hodge-to-de Rham spectral sequence (cf. e.g. \cite[\S4]{MP1}).

\subsection{Brief review of local cohomology}\label{scn:LC}
Let $X$ be a smooth, irreducible $n$-dimensional complex algebraic variety and $Z$ a proper closed subscheme of $X$ defined by the coherent ideal sheaf $\I_Z$. For a quasi-coherent $\shO_X$-module $\Mmod$ and $j\geq 0$, we denote by $\cH^q_Z(\Mmod)$ the \emph{$q^{\rm th}$ local cohomology sheaf} of $\Mmod$, with support in $Z$. This is the $q^{\rm th}$ derived functor of the functor $\underline{\Gamma_Z} (-)$ given by the subsheaf of local sections 
with support in $Z$. The sheaves $\cH^q_Z(\Mmod)$
only depend on the support of $Z$, so in many situations it is convenient to assume that $Z$ is reduced.
For the basic facts on local cohomology see \cite{Hartshorne-LC}.

The sheaf $\cH^q_Z(\Mmod)$ is a quasi-coherent sheaf, whose local sections are annihilated by suitable powers of $\I_Z$. 
For every affine open subset $U\subseteq X$, if 
$$A=\shO_X(U), \,\,\,\,\,\,I=\I_Z(U),\,\,\,\,\,\, {\rm and} \,\,\,\,\,\,M=\Mmod(U),$$ 
then 
$\cH^q_Z(\Mmod)\vert_U$ is the sheaf associated to the local cohomology module $H^q_I(M)$. This can be computed as follows:
if $I=(f_1,\ldots,f_r)$ (or, more generally, if $I$ and $(f_1,\ldots,f_r)$ have the same radical)
and for a subset $J\subseteq\{1,\ldots,r\}$, we put $f_J:=\prod_{i\in J}f_i$, then we have the \v{C}ech-type complex
 \begin{equation}\label{eq_Cech}
C^{\bullet} : \quad 0\to  C^0\to C^1 \to  \cdots\to  C^r\to  0,
\end{equation}
with 
$$C^p=\bigoplus_{|J|=p}A_{f_J},$$
and with the maps given (up to suitable signs) by the localization homomorphisms. 
With this notation, we have functorial isomorphisms 
\begin{equation}\label{eq1_loc_coh}
H^q_I(M)\simeq H^q(C^{\bullet}\otimes_AM).
\end{equation}
The two main cases we will be interested in are those when $\Mmod$ is either $\shO_X$ or $\omega_X$.
Note that we have $\cH^0_Z(\shO_X) = \cH^0_Z(\omega_X) = 0$. 

In general, given $X$ and $Z$ as above, we write $U=X\smallsetminus Z$ and let $j\colon U\hookrightarrow X$
be the inclusion map. For every $\Mmod$, we have a functorial exact sequence
\begin{equation}\label{eq2_loc_coh}
0\to \cH^0_Z(\Mmod)\to\Mmod\to j_*(\Mmod\vert_U)\to \cH^1_Z(\Mmod)\to 0
\end{equation}
and functorial isomorphisms
\begin{equation}\label{eq3_loc_coh}
R^qj_*(\Mmod\vert_U)\simeq \cH^{q+1}_Z(\Mmod)\quad\text{for all}\quad q\geq 1.
\end{equation}

\begin{example}\label{ex:divisor}
If $Z$ has pure codimension 1 (that is, it is an effective divisor on $X$), then 
$\cH^1_Z(\shO_X)\simeq\shO_X(*Z)/\shO_X$ and $\cH^q_Z(\shO_X)=0$ for all $q\neq 1$. 
Here $\shO_X(*Z)$ denotes the sheaf of rational functions on $X$ with poles along $Z$.
\end{example}

\begin{remark}\label{van_criterion}
Since $X$ is smooth, hence Cohen-Macaulay, the usual description of depth in terms of local cohomology
(see \cite[Theorem~3.8]{Hartshorne-LC}) 
implies that 
$${\rm codim}_X(Z)=\min\{q\mid\cH_Z^q(\shO_X)\neq 0\}.$$
Moreover, from the description via the \v{C}ech complex in (\ref{eq1_loc_coh}) it follows that if $Z$ is locally cut out set-theoretically by $\leq N$ equations, then $\cH^q_Z(\shO_X)=0$ for $q>N$. In particular, we conclude that if $Z$ is a local complete intersection of pure codimension $r$, then $\cH^q_Z(\shO_X)=0$ for all $q\neq r$. 
\end{remark}

If $Z'\subseteq Z$ is another closed subset of $X$, then for every quasi-coherent sheaf $\Mmod$, we have natural maps
$$\cH^q_{Z'}(\Mmod)\to \cH^q_Z(\Mmod).$$
These are induced by the natural transformation of functors
$\cH^0_{Z'}(-)\to  \cH^0_Z(-)$.

It is a standard fact that if $\Mmod$ is a left $\Dmod_X$-module, then each $\cH_Z^q(\Mmod)$ has a canonical structure of 
left $\Dmod_X$-module. Locally, this can be seen by using the description in (\ref{eq1_loc_coh}). Indeed, each localization 
$M_{f_J}$ has an induced $\Dmod_X(U)$-module structure such that the morphisms in the complex $C^{\bullet}\otimes_AM$
are morphisms of $\Dmod_X(U)$-modules. Therefore each cohomology
module $H^q_I(M)$ has an induced $\Dmod_X(U)$-module structure and one can easily check that this is independent of the choice 
of generators for $I$. Hence these structures glue to a $\Dmod_X$-module structure on $\cH^q_Z(\Mmod)$. 

Note also that if $\Mmod$ is a left $\Dmod_X$-module, then the sheaves $R^qj_*(\Mmod\vert_U)$ are left $\Dmod_X$-modules as well; they are the cohomology sheaves of the $\Dmod$-module push-forward of $\Mmod$ via $j$. Moreover, in this case the exact sequence (\ref{eq2_loc_coh}) and the isomorphism
(\ref{eq3_loc_coh}) also hold at the level of $\Dmod_X$-modules. 

Similar considerations apply for right $\Dmod_X$-modules. It is easy to see, using the local description, 
that for every $q\geq 0$ we have a canonical isomorphism
$$\cH^q_Z(\Mmod)^r\simeq\cH^q_Z(\Mmod^r).$$
In particular, we have a canonical isomorphism of right $\Dmod_X$-modules 
$$\cH^q_Z(\shO_X)^r\simeq\cH^q_Z(\omega_X).$$

\subsection{The Hodge filtration on local cohomology}\label{scn:HF}
We continue to use the notation introduced in the previous sections. 
Denoting by $i\colon Z\hookrightarrow X$ the inclusion, there are objects 
$j_*\QQ_U^H[n]$ and $i_*i^!\QQ_X^H[n]$ in the derived category of Hodge modules, and an exact triangle
\begin{equation}\label{eq5_loc_coh}
i_*i^!\QQ_X^H[n]\longrightarrow \QQ_X^H[n]\longrightarrow j_*\QQ_U^H[n]\overset{+1}\longrightarrow,
\end{equation}
as shown in  \cite[Section~4]{Saito-MHM}. 
The cohomologies of these objects are Hodge modules whose underlying $\Dmod_X$-modules we have already seen:
$$\cH^q\big(i_*i^!\QQ_X^H[n]\big)=\cH_Z^q(\shO_X)\quad\text{and}\quad \cH^q\big(j_*\QQ_U^H[n]\big)=R^qj_*\shO_U.$$
In particular, these $\Dmod_X$-modules carry canonical Hodge filtrations; note that these only depend on the reduced subscheme $Z_{\rm red}$.
Moreover, the long exact sequence of cohomology corresponding to (\ref{eq5_loc_coh}) gives the counterparts of (\ref{eq2_loc_coh}) and
(\ref{eq3_loc_coh}) in this setting: an exact sequence of filtered $\Dmod_X$-modules
\begin{equation}\label{eq22_loc_coh}
0\to \shO_X\to j_* \shO_U \to \cH_Z^1(\shO_X)\to 0
\end{equation}
and an isomorphism of filtered $\Dmod_X$-modules
\begin{equation}\label{eq32_loc_coh}
R^qj_* \shO_U \simeq\cH_Z^{q+1}(\shO_U)\quad\text{for all}\quad q\geq 1.
\end{equation}

For example, when $Z = D$ is a reduced effective divisor, we have a canonical Hodge filtration on $\shO_X(*D) = j_* \shO_U$, which underlies the Hodge module $j_*\QQ_U^H[n]$; this is analyzed in \cite{MP1}. Up to modding out by $\shO_X$, this is therefore equivalent to the Hodge filtration on the local cohomology sheaf $\cH^1_D(\shO_X)$.
It turns out that for arbitrary $Z$ the Hodge filtration on the local cohomology sheaves 
$\cH^q_Z(\shO_X)$ can be defined using the one on sheaves of the form $\shO_X(*D)$. Note first that it is enough 
to describe the Hodge filtration in each affine chart $U$. In this case, if we consider $f_1,\ldots,f_r\in A=\shO_X(U)$ that generate an ideal having the same radical as $\I_Z(U)$, then each localization $A_{f_J}$ carries a Hodge filtration such that the corresponding \v{C}ech-type complex (\ref{eq_Cech}) is a complex of filtered 
$\Dmod$-modules. In fact, the maps come from morphisms of mixed Hodge modules: each component is, up to sign, induced by the canonical map
$${j_U}_*\QQ^H_U[n]\to {j_V}_*\QQ^H_V[n],$$
where $V\subseteq U$ are complements of suitable hypersurfaces, and $j_U\colon U\hookrightarrow X$ and $j_V\colon V\hookrightarrow X$
are the inclusion maps. Therefore the maps in the
complex (\ref{eq_Cech}) are strict and the Hodge filtration on the cohomology sheaves $\cH^q_Z(\shO_X)$ is the induced filtration.

\begin{remark}\label{rem_trivial}
We have 
$$F_p\cH_Z^q(\shO_X)=0 \,\,\,\,\,\,{\rm for ~all}\,\,\,\,p<0  \,\,\,\,{\rm and} \,\,\,\, q\geq 0.$$
Indeed, arguing locally and using the computation of local cohomology
via the \v{C}ech-type complex (\ref{eq_Cech}), we deduce this assertion from the fact that if $D$ is a reduced effective divisor on $X$, then
$F_p\shO_X(*D)=0$ for $p<0$ (see for example \cite[Section~9]{MP1}).
\end{remark}

\begin{remark}\label{rmk_functoriality}
Suppose that $Z'\subseteq Z$ is another closed subset. If $U'=X\smallsetminus Z'$, then $j$ factors as
$$U\overset{k}\hookrightarrow U'\overset{j'}\hookrightarrow X.$$
In the derived category of Hodge modules on $X$, we have an isomorphism
$$j_*\QQ^H_U[n]\simeq j'_*\big(k_*\QQ^H_U[n]\big).$$
The canonical morphism $\QQ^H_{U'}[n]\to k_*\QQ^H_U[n]$
induces therefore
$$j'_*\QQ^H_{U'}[n]\to j_*\QQ^H_U[n],$$
so that after passing to cohomology, we get morphisms of Hodge modules
$$R^qj'_*\QQ^H_{U'}[n]\to R^qj_*\QQ^H_U[n].$$
We deduce using (\ref{eq22_loc_coh}) and (\ref{eq32_loc_coh}) 
that the canonical morphisms 
$$\big(\cH_{Z'}^q(\shO_X), F\big)\to \big(\cH_Z^q(\shO_X), F\big)$$
are (strict) morphisms of filtered $\Dmod_X$-modules.
\end{remark}

\begin{remark}\label{comparison_two_embeddings}
It is not hard to compare the Hodge filtrations on the local cohomology sheaves along $Z$ with respect to two different embeddings in smooth varieties. Indeed, suppose that $k\colon X\hookrightarrow X'$ is a closed embedding, with $X'$ smooth, and 
$\dim(X)=n$ and $\dim(X')=n+d$. 
In this case, if $i\colon Z\hookrightarrow X$ and $i'\colon Z\hookrightarrow X'$ are the two embeddings, we have an isomorphism
$$i'_*{i'}^!\QQ_{X'}[n+d] \simeq \big(k_*i_*i^!\QQ_X^H[n]\big)[-d](-d),$$
where we recall that $(-d)$ denotes the Tate twist defined in \S\ref{scn:Dmod}.
By taking cohomology, we see that for every $q$, we have an isomorphism of filtered $\Dmod_{X'}$-modules
$$\big(\cH_Z^{q+d}(\shO_{X'}), F\big)\simeq \big(k_+\cH^{q}_Z(\shO_X)(-d), F\big).$$
Explicitly, if $y_1,\ldots,y_{n+d}$ are local coordinates on $X'$ such that $X$ is defined by $(y_1,\ldots,y_d)$, then we have an isomorphism
$$\cH_Z^{q+d}(\shO_{X'})\simeq\bigoplus_{\alpha_1,\ldots,\alpha_d\geq 0}\cH^q_Z(\shO_X)\otimes\partial_{y_1}^{\alpha_1}\cdots\partial_{y_d}^{\alpha_d}$$
such that the Hodge filtrations are related by
$$F_p\cH_Z^{q+d}(\shO_{X'})\simeq\bigoplus_{\alpha_1,\ldots,\alpha_r\geq 0}
F_{p-\sum_i\alpha_i}\cH^q_Z(\shO_X)\otimes\partial_{y_1}^{\alpha_1}\cdots\partial_{y_d}^{\alpha_d}.$$
\end{remark}

\begin{remark}
For simplicity, we only considered the local cohomology of $\shO_X$. More generally, one can consider an arbitrary Hodge module $M$ on $X$, with underlying filtered $\Dmod_X$-module $\Mmod$. In this case the local cohomology
sheaves $\cH_Z^q(\Mmod)$ continue to underlie Hodge modules, and thus carry canonical Hodge filtrations. 
\end{remark}

We end this section with two examples: the case of a smooth subvariety, and that of subsets defined by monomial ideals. We note in addition that the Hodge filtration on the local cohomology Hodge modules with support in generic determinantal ideals is studied extensively in \cite{PR} and \cite{Perlman}.

\begin{example}[Smooth subvarieties]\label{smooth_subvariety}
If $Z$ is a smooth, irreducible subvariety of $X$ of codimension $r$, then $\cH_Z^q(\shO_X)=0$ for $q\neq r$; see Remark~\ref{van_criterion}. We claim that
\begin{equation}\label{eq_smooth_subvariety}
F_p\cH^r_Z(\shO_X)=\{u\in\cH^r_Z(\shO_X)\mid\I_Z^{p+1}u=0\}.
\end{equation}

In order to see this, we may assume that $X$ is affine with $A=\shO_X(X)$ and we have global algebraic coordinates $x_1,\ldots,x_n$ on $X$ 
such that $Z$ is defined by $I=(x_1,\ldots,x_r)$. In this case, as we have previously discussed,
$H_I^r(A)$ is filtered isomorphic to the cokernel of the map
$$\bigoplus_{i=1}^rA_{x_1\cdots \widehat{x_i}\cdots x_r} \to A_{x_1\cdots x_r}.$$
Moreover, since $x_1\cdots x_r$ defines a simple normal crossing divisor on $X$, the Hodge filtration on
$A_{x_1\cdots x_r}$ is given by 
$$F_pA_{x_1\cdots x_r}=F_p\Dmod_X\cdot\frac{1}{x_1\cdots x_r}$$
(see \cite[Section~8]{MP1}).
This is generated over $A$ by the classes of $\frac{1}{x_1^{a_1}\cdots x_r^{a_r}}$,
with $a_1,\ldots,a_r\geq 1$,  such that $a_1+\cdots+a_r=p+r$.
The fact that this is equal to the right-hand side of (\ref{eq_smooth_subvariety}) follows easily from the fact that $x_1,\ldots,x_r$ form
a regular sequence in $A$ (for example, by reducing to the case when $A$ is the polynomial ring ${\mathbf C}[x_1,\ldots,x_r]$). 

\smallskip
Another way to obtain this description, relying on the formalism of mixed Hodge modules, is the following. 
If $i\colon Z\hookrightarrow X$ is the inclusion, then
$$
i^!\QQ_X^H[n]=\big(\QQ_Z^H[n-r]\big)[-r](-r).
$$
Applying $i_*$ and taking cohomology,  we get an isomorphism of filtered $\Dmod_X$-modules
\begin{equation}\label{eq_smooth_subvariety_MHM}
\cH^r_Z(\shO_X)\simeq i_+\shO_Z(-r).
\end{equation}
Explicitly, if $x_1,\ldots,x_n$ are local coordinates on $X$ such that $Z$ is defined by
$(x_1,\ldots,x_r)$, then we have an isomorphism
$$
\cH^r_Z(\shO_X)\simeq\bigoplus_{\alpha_1,\ldots,\alpha_r\geq 0}\shO_Z\otimes\partial_{x_1}^{\alpha_1}\cdots \partial_{x_r}^{\alpha_r}
$$
such that the Hodge filtration is given by
$$F_p\cH^r_Z(\shO_X)\simeq \bigoplus_{\alpha_1+\cdots+\alpha_r\leq p}\shO_Z\otimes \partial_{x_1}^{\alpha_1}\cdots \partial_{x_r}^{\alpha_r}.$$
Note that the element $1\otimes 1$ corresponds to the class of $\frac{1}{x_1\cdots x_r}$ in the \v{C}ech description of local cohomology,
hence we obtain the same description of the Hodge filtration as in ($\ref{eq_smooth_subvariety}$). For completeness, we note that it also follows from 
(\ref{eq_smooth_subvariety_MHM}) that $\cH^r_Z(\shO_X)$ underlies a pure Hodge module of weight $n+r$. 
\end{example}

\begin{example}[Monomial ideals]\label{monomial_ideal}
We now consider the case when $X=\AAA^n$ and $I\subseteq A=\CC[x_1,\ldots,x_n]$ is a monomial ideal.
The $(\CC^*)^n$-action on $\AAA^n$ induces a $(\CC^*)^n$-action on each $H^q_I(A)$, which translates into 
a $\ZZ^n$-grading on $H^q_I(A)$. The action of every element of $(\CC^*)^n$ induces an isomorphism of 
filtered $\Dmod_X$-modules $H^q_I(A)\to H^q_I(A)$, hence every $F_pH^q_I(A)$ is a graded $A$-submodule
of $H^q_I(A)$. 

Let us denote by $e_1,\ldots,e_n$ the standard basis of $\ZZ^n$.
It is easy to see that multiplication by $x_i$ induces an isomorphism
$$H^q_I(A)_u\to H^q_I(A)_{u+e_i}$$
for every $u\in\ZZ^n$, unless $u_i=-1$. This follows, for example, 
by computing the local cohomology via the \v{C}ech-type complex associated
to a system of reduced monomial generators of ${\rm Rad}(I)$; see \cite[Theorem~1.1]{EMS}.
Similarly, multiplication by $\partial_i$ induces an isomorphism
\begin{equation}\label{eq_eg_monomial}
H^q_I(A)_u\to H^q_I(A)_{u-e_i}
\end{equation}
for every $u\in\ZZ^n$, unless $u_i=0$. 

For every $u\in\ZZ^n$, we write $x^u$ for the corresponding Laurent monomial. Given $J\subseteq \{1,\ldots,n\}$,
let $u_J=\sum_{i\in J}e_i$ and $x_J=x^{u_J}$. If for $u=(u_1,\ldots,u_n)\in\ZZ^n$, we put ${\rm neg}(u):=\{i\mid u_i<0\}$,
then $A_{x_J}=\bigoplus_{{\rm neg}(u)\subseteq J}\CC x^u$. Furthermore, 
the explicit description of the Hodge filtration on $A_{x_J}$ that we have seen in Example~\ref{smooth_subvariety} can be rewritten as
$$F_pA_{x_J}=\bigoplus_u\CC x^u,$$
where the direct sum is over those $u$ with ${\rm neg}(u)\subseteq J$ and such that $\sum_{i\in {\rm neg}(u)}(-u_i-1)\leq p$. 
By computing the local cohomology via the \v{C}ech-type complex thus gives
$$F_pH^q_I(A)=\bigoplus_uH^q_I(A)_u,$$
where the sum is over those $u\in\ZZ^n$ such that $\sum_{i\in {\rm neg}(u)}(-u_i-1)\leq p$. 
In particular, we see that
$$F_0H^q_I(A)=\bigoplus_uH^q_I(A)_u,$$
where the sum is over all $u\in\ZZ^n$ such that $u_i\geq -1$ for all $i$. 
\end{example}

\subsection{Birational description and strictness}\label{scn:birational}
For us it will be important to have a description of the Hodge filtration on $\cH^q_Z(\shO_X)$ via a log resolution of the pair
 $(X,Z)$. 
To this end, it is more convenient to use the corresponding right Hodge modules, with corresponding 
$\Dmod_X$-modules $\cH^q_Z(\omega_X)$.

Suppose that $f\colon Y\to X$ is such a resolution; more precisely, we require $f$ to be a projective morphism that is an isomorphism over $U=X\smallsetminus Z$, such that $Y$ is a smooth variety, and the reduced inverse image
$f^{-1}(Z)_{\rm red}$ is a simple normal crossing divisor $E$. If $j'\colon V=Y\smallsetminus E\hookrightarrow Y$ is the inclusion map, then we have a commutative diagram
\begin{equation}\label{eq_diag_res}
\begin{tikzcd}
V\rar{j'}\dar  & Y
\dar{f} \\
U \rar{j} & X,
\end{tikzcd}
\end{equation}
in which the left vertical map is an isomorphism. We thus have an isomorphism 
$$j_*\QQ_U^H[n]\simeq f_*j'_*\QQ_V^H[n]$$
in the derived category of Hodge modules. As indicated in the paragraph after ($\ref{eq32_loc_coh}$), since $E$ is a divisor the underlying filtered right $\Dmod_Y$-module of the Hodge module $j'_*\QQ_V^H[n]$ is $\big(\omega_Y (*E), F\big)$, where $F$ is the Hodge filtration; see \cite[\S8]{MP1} for a more precise description in the present SNC setting. Taking cohomology in the isomorphism
above, we therefore obtain isomorphisms of filtered right $\Dmod_X$-modules
\begin{equation}\label{eqn:iso}
\shH^q f_+\omega_Y(*E) \simeq R^qj_*\omega_U \simeq 
\begin{cases}
\cH^{q+1}_Z(\omega_X) \,\,\,\,{\rm for} \,\,\,\, q \ge 1 \\
\\
j_* \omega_U \onto \cH^{1}_Z(\omega_X) \simeq j_* \omega_U/ \omega_X  \,\,\,\,{\rm for} \,\,\,\, q =0
\end{cases}
\end{equation}
where for the last isomorphism we use ($\ref{eq22_loc_coh}$) and ($\ref{eq32_loc_coh}$).

On the other hand, we have a filtered resolution of $\omega_Y(*E)$ by induced right $\Dmod_Y$-modules, given by the complex 
\begin{equation}\label{eq_filtered_res}
0\to\Dmod_Y\to\Omega_Y^1(\log E)\otimes_{\shO_Y}\Dmod_Y\to\cdots\to\omega_Y(E)\otimes_{\shO_Y}\Dmod_Y\to 0,
\end{equation}
placed in degrees $-n,\ldots,0$; see \cite[Proposition~3.1]{MP1}.  Using the notation and discussion in \S\ref{scn:Dmod}, by tensoring with $\Dmod_{Y \to X}$ over $\Dmod_Y$, we thus obtain a filtered complex
$$A^{\bullet}:\quad
0\to f^*\Dmod_X\to \Omega_Y^1(\log E)\otimes_{\shO_Y}f^*\Dmod_X\to\cdots\to\omega_Y(E)\otimes_{\shO_Y}f^*\Dmod_X\to 0,$$
which is filtered quasi-isomorphic to 
$\omega_Y (*E) \overset{\derL}{\otimes}_{\Dmod_Y} \Dmod_{Y\to X}$, and therefore can be used to compute 
$f_+ \omega_Y(*E)$. Note that the filtration on $A^{\bullet}$ is such that 
 $F_{p-n}A^{\bullet}$, for $p \ge 0$,  is the subcomplex
$$0\to f^*F_{p-n}\Dmod_X\to \Omega_Y^1(\log E)\otimes_{\shO_Y}f^*F_{p-n+1}\Dmod_X\to\cdots\to\omega_Y(E)\otimes_{\shO_Y}f^*F_p\Dmod_X\to 0.$$
A detailed exposition of all this can be found in \cite[\S2 and \S3]{MP1}.

Now by ($\ref{eqn:iso}$) and the definition of $\Dmod$-module push-forward, for all $q \ge 0$ we have 
$$R^qj_* \omega_U \simeq R^qf_*A^{\bullet}$$
as right $\Dmod_X$-modules. The filtration on the right hand side is described at the end of \S\ref{scn:Dmod}, and since 
these filtered $\Dmod$-modules underlie Hodge modules, the strictness property of the Hodge filtration in 
($\ref{strictness_formula}$) leads to the following key Hodge-theoretic consequence for this birational interpretation:

\begin{proposition}\label{comput_log_res}
With the above notation, for every $p,q\geq 0$, the inclusion $F_{p-n}A^{\bullet}\hookrightarrow A^{\bullet}$ induces an injective map
$$R^qf_*F_{p-n}A^{\bullet}\hookrightarrow R^qf_*A^{\bullet}\simeq R^qj_*\omega_U,$$
whose image is $F_{p-n}R^qj_*\omega_U$. Moreover, via ${\rm (}\ref{eqn:iso}{\rm )}$ this image coincides with $F_{p-n} \cH^{q+1}_Z(\omega_X)$ for $q \ge 1$, while for $q=0$, its quotient by $\shO_X$ gives $F_{p-n} \cH^1_Z(\omega_X)$.
\end{proposition}

As a concrete example, we obtain the following description of the lowest term of the filtration (note that by
(\ref{eq_left_right_shift}) and Remark~\ref{rem_trivial}, we have $F_p\cH_Z^q(\omega_X)=0$ for all $q$ and all $p<-n$):

\begin{corollary}\label{cor_F0}
For every $q\geq 2$, we have a canonical isomorphism
$$F_{-n}\cH^q_Z(\omega_X)\simeq R^{q-1}f_*\omega_Y(E).$$
Moreover, we have a short exact sequence
$$0\to\omega_X\to f_*\omega_Y(E)\to F_{-n}\cH^1_Z(\omega_X)\to 0.$$
\end{corollary}

\begin{proof}
The assertion follows from Proposition~\ref{comput_log_res} and the fact that $F_{-n}A^{\bullet}=\omega_Y(E)$. 
\end{proof}

\begin{remark}[Analytic setting]\label{rmk:analytic1}
The results in this section also apply when $Z$ is an analytic subspace of a complex manifold $X$, due to the study of open direct images in the analytic category (see \cite[Proposition 2.11 and Corollary 2.20]{Saito-MHM}) and the fact that the strictness theorem \cite[Theorem 2.14]{Saito-MHM} holds for any projective morphism between analytic spaces, e.g. a resolution of singularities. Cf. also \cite[\S2.1]{JKSY}.
\end{remark}

\subsection{An injectivity theorem}\label{scn:compatibility}
We continue to use the notation in the previous section. Using the exact sequence
$$0 \to \omega_Y \to \omega_Y(E) \to \omega_E \to 0$$
and the Grauert-Riemenschneider vanishing theorem, Corollary \ref{cor_F0} tells us that for all $q \ge 1$, as a consequence of strictness for the Hodge filtration, we have an isomorphism 
$$\gamma_q \colon R^{q-1} f_* \omega_E \to F_{-n}\cH^q_Z(\omega_X).$$
We also consider the inclusion maps
$$i_q \colon F_{-n}\cH^q_Z(\omega_X) \hookrightarrow \cH^q_Z(\omega_X).$$

Theorem \ref{inclusions} in the Introduction follows therefore once the following 
compatibility is established:

\begin{proposition}\label{compatibility}
For each $q$, the composition 
$$i_q \circ \gamma_q \colon R^{q-1} f_* \omega_E \to \cH^q_Z(\omega_X)$$
coincides with the morphism on cohomology described in the statement of Theorem \ref{inclusions}.
\end{proposition}
\begin{proof}
For simplicity we write down the argument for $q \ge 2$, when $ R^{q-1} f_* \omega_E \simeq  R^{q-1} f_* \omega_Y(E)$.
The argument for $q =1$ is similar, only this time by definition one needs to consider $f_* \omega_Y(E) / \omega_Y$ (as opposed 
to $f_* \omega_Y(E)$).

\noindent
\emph{Step 1.} We consider the commutative diagram (\ref{eq_diag_res}) and we identify $U$ and $V$ via the left vertical map. Applying $\derR f_*$ to the canonical inclusion $\omega_Y (E) \hookrightarrow j'_* \omega_U \simeq \omega_Y (*E)$ and passing to cohomology, we obtain morphisms
\begin{equation}\label{eqn:step1}
R^{q-1} f_* \omega_Y (E) \to R^{q-1} j_* \omega_U \simeq \cH^q_Z (\omega_X)
\end{equation}
for each $q$, where the last isomorphism is the canonical isomorphism in ($\ref{eq3_loc_coh}$). We claim that these morphisms can be identified with the compositions $i_q \circ \gamma_q$.

To see this, recall that we have identified
$$\cH^q_Z (\omega_X) \simeq \cH^{q-1}f_+ \omega_Y (*E) : = R^{q-1} f_* \big(
\omega_Y (*E) \overset{\derL}{\otimes}_{\Dmod_Y} \Dmod_{Y\to X} \big).$$
The transfer module admits a canonical morphism $\Dmod_Y \to \Dmod_{Y \to X}$
of left $\Dmod_Y$-modules (induced by $T_Y \to f^* T_X$), which in turn induces a morphism 
$$\rho_q \colon R^{q-1} f_* \omega_Y(*E) \to  \cH^{q-1}f_+ \omega_Y (*E).$$
Now the morphism $i_q \circ \gamma_q$, defined using the resolution ($\ref{eq_filtered_res}$) of $\omega_Y(*E)$, is obtained more precisely by pushing forward the inclusion $\omega_Y (E) \hookrightarrow \omega_Y (*E)$, and then considering the composition
$$R^{q-1} f_* \omega_Y(E)  \longrightarrow R^{q-1} f_* \omega_Y(*E) \overset{\rho_q }{\longrightarrow}
 \cH^{q-1}f_+ \omega_Y (*E).$$
 But since $f \circ j' = j$, $\rho_q$ can also be identified canonically with the isomorphism $R^{q-1} j_* \omega_U \to \cH^{q-1} j_+ \omega_U$ appearing above, and we are done.

\noindent
\emph{Step 2.}
In this step we discuss the following general situation: assume that $W$ is a closed subscheme in $X$, with ideal sheaf $\J$. If $j \colon V \hookrightarrow X$ is the inclusion map of the complement $V = X \smallsetminus W$, we have a diagram of exact triangles
$$
\begin{tikzcd}
\omega_X  \dar{=} \rar & \derR \shH om_{\shO_X} (\J , \omega_X) \rar \dar & \derR \shH om_{\shO_X} (\shO_W , \omega_X)[1] \dar  \overset{+1}{\longrightarrow} \\
\omega_X \rar & \derR j_*\omega_V  \rar & \derR \underline{\Gamma_W} (\omega_X)[1]  \overset{+1}{\longrightarrow} 
 \end{tikzcd}
$$
where the vertical map on the right is the canonical morphism and where the middle vertical map can be described as the canonical morphism
\begin{equation}\label{descr_can_morphism}
\derR \shH om_{\shO_X} (\J , \omega_X) \to \derR j_* j^* \derR \shH om_{\shO_X} (\J , \omega_X)=\derR j_*\omega_V,
\end{equation}
since $\J_{|V} \simeq \shO_V$.

Moreover, in the presence of a proper morphism $f \colon Y \to X$, assumed to be an isomorphism over $V$, we can consider the ideal $\J_Y : = \J \cdot \shO_Y$ defining $f^{-1} (W)$, and completely similarly we have a canonical 
morphism 
$$\derR \shH om_{\shO_Y} (\J_Y ,f^{!} \omega_X) \to \derR j^{\prime}_*((f^{!} \omega_X)_{|V}) = \derR j^{\prime}_*
 \omega_V,$$
 where $j'$ is the inclusion of $V$ in $Y$. Applying $\derR f_*$ to this morphism, and Grothendieck duality to the first term, we obtain a diagram
 $$
\begin{tikzcd}
\derR f_* \derR \shH om_{\shO_Y} (\J_Y ,f^{!} \omega_X) \dar{\simeq} \rar & \derR f_* \derR j^{\prime}_* \omega_V \rar{=} &\derR f_* \derR j^{\prime}_* \omega_V \dar{\simeq} \\
\derR \shH om_{\shO_X} (\derR f_* \J_Y , \omega_X) \rar & \derR \shH om_{\shO_X} (\J , \omega_X) \rar &
\derR j_* \omega_V
\end{tikzcd}
$$
where we included the natural factorization of the bottom morphism through the object $\derR \shH om_{\shO_X} (\J , \omega_X)$, 
induced by the canonical morphism $\J \to \derR f_* \J_Y$; this factorization holds due to the description of the morphism in (\ref{descr_can_morphism}) and the fact that Grothendieck duality is compatible with restriction to open subsets.
 
 \noindent
\emph{Step 3.}
In this step we apply the  constructions in Step 2 to give another description of the morphism ($\ref{eqn:step1}$), which 
will finish the proof. First, ($\ref{eqn:step1}$) can be rewritten as the composition
$$R^{q-1} f_* \derR\shH om_{\shO_Y}\big(\shO_Y (-E), \omega_Y \big)  \to R^{q-1} f_* \derR\shH om_{\shO_Y} \big(\I_{f^{-1}(Z)}, \omega_Y \big) \to R^{q-1} j_* \omega_U,$$
where the factorization through the middle term holds since $E$ is the reduced structure on $f^{-1} (Z)$. Applying Grothendieck duality, this composition can be rewritten as
$$\shE xt^{q-1}_{\shO_X} \big(\derR f_* \shO_Y (-E), \omega_X \big)  \to \shE xt^{q-1}_{\shO_X}  \big(\derR f_* \I_{f^{-1}(Z)}, \omega_X \big) \to R^{q-1} j_* \omega_U,$$
and the map on the right factors further through $\shE xt^{q-1}  \big( \I_Z, \omega_X \big)$, as described in the last diagram 
in Step 2. Moreover, it is straightforward to see that for $q \ge 2$ we have canonical isomorphisms 
$$\shE xt^{q-1}_{\shO_X} \big(\derR f_* \shO_Y (-E), \omega_X \big) \simeq \shE xt^q_{\shO_X} \big(\derR f_* \shO_E , \omega_X \big)$$
and 
$$\shE xt^{q-1}_{\shO_X} \big(\I_Z, \omega_X \big) \simeq \shE xt^q_{\shO_X} \big(\shO_Z , \omega_X \big).$$
Altogether, the morphism  ($\ref{eqn:step1}$) can be identified with the natural composition 
$$\shE xt^q_{\shO_X} \big(\derR f_* \shO_E , \omega_X \big) \to \shE xt^q_{\shO_X} \big(\shO_Z , \omega_X \big) \to \cH^q_Z (\omega_X),$$
which is the same as the map on cohomology described in the statement of Theorem  \ref{inclusions}. (Note that Grothendieck duality gives $ \derR f_* \omega_E^\bullet \simeq \derR \shH om_{\shO_X} (\derR f_* \shO_E, \omega_X[n])$, while 
$\omega_Z^\bullet \simeq \derR \shH om_{\shO_X} (\shO_Z, \omega_X[n])$.)
\end{proof}

The upshot of Theorem \ref{inclusions} (and Proposition \ref{compatibility}) is that for each $q\ge1$ we have a commutative
diagram
$$
\begin{tikzcd}
R^{q-1} f_* \omega_E \rar{\alpha_q} \dar{\simeq}  & \E xt^q \big(\shO_Z, \omega_X \big)  \dar \\
 F_{-n} \mathcal{H}^q_Z \omega_X \rar &  \mathcal{H}^q_Z \omega_X,
 \end{tikzcd}
$$
where the left vertical map is the isomorphism in Corollary \ref{cor_F0}, the bottom horizontal map is the inclusion, the right vertical map is the natural map to local cohomology, while $\alpha_q$ is the morphism on cohomology induced by 
$$\alpha \colon \derR f_* \omega_E^\bullet \to \omega_Z^\bullet$$
in ${\bf D}^b\big({\rm Coh} (X)\big)$, obtained  in turn by dualizing the natural morphism $\shO_Z \to \derR f_* \shO_E$. (Here 
$\omega_E^\bullet = \omega_E [n -1]$ since $E$ is Gorenstein.)

As a consequence, the top horizontal map is injective. In other words this gives another proof of the injectivity result of Kov\'acs and Schwede, Corollary \ref{KS-injection} in the Introduction, applied in \cite{KS} to the study of deformations of Du Bois singularities.

\subsection{A local vanishing theorem}\label{scn:local-vanishing}
We use the constructions in \S\ref{scn:birational} to prove a Nakano-type vanishing result for log resolutions of arbitrary closed subsets. This generalizes a result for hypersurfaces due to Saito \cite[Corollary~3]{Saito-LOG} (cf. also \cite[Theorem~32.1]{MP1}).

\begin{proof}[Proof of  Theorem \ref{local-vanishing}]
We note that by the definition of $c$ and by Remark~\ref{van_criterion}, we have
\begin{equation}\label{eq1_thm_local_vanishing}
\cH_Z^{j}(\omega_X)=0\quad\text{if either}\quad j<r\quad\text{or}\quad j>c. 
\end{equation}

The vanishing in the statement holds trivially for $p>n$, hence we may assume $p\leq n$.  Note that in this case
the conditions in both i) and ii) imply $q\geq 1$.
We first check the case $p =n$.\footnote{Note that when $Z$ is a hypersurface, in which case $r=c=1$, this is the well-known 
fact that $R^q f_* \omega_Y (E) = 0$ for $q> 0$, a special case of the Local Vanishing theorem for multiplier ideals.}  This follows 
in both cases i) and ii) from (\ref{eq1_thm_local_vanishing}) and Corollary \ref{cor_F0}.

We next prove the theorem by descending induction on $p$. Let $p < n$ and consider the complex
$$ C^\bullet : = F_{-p} A^\bullet [p -n],$$
placed in cohomological degrees $0,\ldots,n-p$,
where $A^\bullet$ is as in Section~\ref{scn:birational}. 
Note that since $p+q\geq n+1$, we deduce from Proposition~\ref{comput_log_res}
that 
$$R^{p+q-n} f_* F_{-p} A^\bullet \hookrightarrow R^{p+q-n}f_* A^\bullet \simeq \cH_Z^{p+q-n+1} (\omega_X)\quad\text{for all}\quad q,$$
hence using again (\ref{eq1_thm_local_vanishing}), we see that in both cases i) and ii), we have
\begin{equation}\label{eqn:range}
R^q f_* C^\bullet = R^{p+q-n} f_* F_{-p} A^\bullet=0.
\end{equation}

Consider now the hypercohomology spectral sequence
$$E_1^{i,j}=R^jf_*C^i\Rightarrow R^{i+j}f_*C^{\bullet}.$$
By definition we have 
$$C^i= \Omega_Y^{p+i}(\log E)\otimes_{\shO_Y}f^*F_i\Dmod_X\quad\text{for}\quad 0\leq i\leq n-p,$$
hence we want to show that $E_1^{0,q}=0$. First, ($\ref{eqn:range}$) implies that in both cases i) and ii) we have 
$E_{\infty}^{0,q}=0$.

Now for every $k \ge 1$ we clearly have $E_k^{-k,q+k-1}=0$, since this is a first-quadrant spectral sequence.
On the other hand, we also have $E_k^{k,q-k+1}=0$ by induction. Indeed, this is a subquotient of
$$E_1^{k,q-k+1}=R^{q-k+1}f_*C^k=R^{q-k+1}f_*\Omega^{p+k}_Y(\log E)\otimes_{\shO_X}F_k\Dmod_X$$
and the right-hand side vanishes by induction. 
We thus conclude that $E_1^{0,q}=E_{\infty}^{0,q}=0$,
completing the proof. 
\end{proof}

\begin{remark}
(1) In the statement of the theorem, we may replace $c = {\rm lcd}(Z, X)$ by any $s$ such that $Z$ is locally cut out by $s$ equations. Indeed, the fact that $c\leq s$ follows from Remark~\ref{van_criterion}.

\noindent
(2) There exist further useful upper bounds on $c$ that depend only on the codimension $r$ of $Z$. We only list a couple here. In complete generality, Faltings \cite{Faltings} showed that 
$$c \le n - \left[ \frac{n-1}{r} \right].$$
Among many other improvements, Huneke and Lyubeznik \cite{HL} showed that if $Z$ is normal, then
$$c \le n - \left[ \frac{n}{r + 1} \right] -  \left[ \frac{n-1}{r + 1} \right].$$
Further results along these lines, assuming $S_k$ conditions on $Z$, appear in \cite{DT}.  We obtain in particular:

\begin{corollary}\label{explicit-vanishing}
Let $Z$ be a closed subscheme of codimension $r$ in a smooth, irreducible $n$-dimensional variety $X$. If $f\colon Y \to X$ 
is a log resolution of $(X, Z)$ which is an isomorphism away from $Z$, and $E = f^{-1}(Z)_{\rm red}$, then
$$R^qf_* \Omega_Y^p (\log E) = 0 \,\,\,\,\,\,{\rm for} \,\,\,\, p + q \ge 2n -  \left[ \frac{n-1}{r} \right].$$
If moreover $Z$ is assumed to be normal, then the same holds for 
$$p + q \ge 2n -  \left[ \frac{n}{r + 1} \right] -  \left[ \frac{n-1}{r + 1} \right].$$
\end{corollary}
\end{remark} 

\medskip

We conclude by noting that in \cite[Theorem 32.1]{MP1} it is shown that when $Z$ is a Cartier divisor, in order to have local vanishing as in Theorem \ref{local-vanishing} it is enough to assume only that $X$ is smooth away from $Z$. It is therefore natural to ask:

\begin{question}
\emph{Is there an appropriate generalization of Theorem \ref{local-vanishing} that does not assume $X$ to be smooth?}
\end{question}

\section{Order and Ext filtrations, and some comparisons}

\subsection{Order and Ext filtration}\label{scn:order}
We now aim to define analogues of the pole order filtration associated to hypersurfaces, and compare them
with the Hodge filtration. We thank C. Raicu, whose answers to our questions have helped shape the material in this section.

We start by recording the following basic property of the Hodge filtration:

\begin{proposition}\label{inclusion_filtration}
For every $p,q\geq 0$, we have
$$\I_Z\cdot F_p\cH_Z^q(\shO_X)\subseteq F_{p-1}\cH_Z^q(\shO_X).$$
In particular, we have
$$\I_Z^{p+1}F_p\cH^q_Z(\shO_X)=0\quad\text{for every}\quad p\geq 0.$$
\end{proposition}

\begin{proof}
The first assertion is a general property of filtered $\Dmod_X$-modules underlying Hodge modules whose support is
contained in $Z$; see \cite[Lemma~3.2.6]{Saito-MHP}. The second assertion then follows from the fact that $F_{-1}\cH^q_Z(\shO_X)=0$ (see Remark~\ref{rem_trivial}).
\end{proof}

\begin{remark}\label{rmk_inclusion_filtration}
If $Z$ is a reduced divisor, the second assertion in the above proposition is equivalent with the fact that 
$F_p\shO_X(*Z)\subseteq\shO_X\big((p+1)Z\big)$, i.e. the Hodge filtration is contained in the pole order filtration; see \cite[Proposition 0.9]{Saito-B}, and also \cite[Lemma 9.2]{MP1}. Moreover, in terms of Hodge ideals, the first assertion says that $I_p(Z)\subseteq I_{p-1}(Z)$, see \cite[Proposition~13.1]{MP1}.
\end{remark}

\smallskip

\begin{definition}\label{definition_order_filtration}
The \emph{order filtration} on $\cH^q_Z(\shO_X)$ is the increasing filtration given by 
$$O_k\cH^q_Z(\shO_X):=\{u\in \cH^q_Z(\shO_X)\mid \I_Z^{k+1}u=0\},\,\,\,\,\,\, k \ge 0.$$
Note that we have a canonical isomorphism
$$O_k\cH^q_Z(\shO_X) \simeq \cH om_{\shO_X} \big(\shO_X /\I_Z^{k+1},   \cH^q_Z(\shO_X)\big).$$
\end{definition}

If $Z$ is a reduced divisor, then 
$$O_k\cH^1_Z(\shO_X) = \shO_X \big((k+1)Z\big) / \shO_X,$$
and the inclusions $F_k \shO_X(*Z) \subseteq \shO_X \big((k+1)Z\big)$ in Remark~\ref{rmk_inclusion_filtration}
say that 
$$F_k\cH^1_Z(\shO_X) \subseteq O_k\cH^1_Z(\shO_X) \,\,\,\,\,\,{\rm for~all} \,\,\,\,k \ge 0.$$
This last fact continues to be true in general:

\begin{proposition}\label{one_inclusion}
For arbitrary $Z$, and for every $k$ and $q$, we have 
$$F_k\cH^q_Z(\shO_X) \subseteq O_k\cH^q_Z(\shO_X).$$
\end{proposition}
\begin{proof}
The statement is equivalent to the second assertion in Proposition \ref{inclusion_filtration}.
\end{proof}

\begin{remark}\label{infinite-generated}
The order filtration on $\cH^q_Z(\shO_X)$ is compatible with the filtration on $\Dmod_X$ by order of differential operators:
$$F_\ell \Dmod_X \cdot O_k\cH^q_Z(\shO_X) \subseteq O_{k + \ell} \cH^q_Z(\shO_X)\quad\text{for all}\quad k,\ell\geq 0.$$
However, unless $q = r = {\rm codim}_X(Z)$, the sheaves $O_k\cH^q_Z(\shO_X)$ are not 
coherent as long as $\cH^q_Z(\shO_X)\neq 0$. This makes the order filtration less suitable for $q > r$; this is 
a rather deep result of Lyubeznik \cite[Corollary 3.5]{Lyubeznik}.\footnote{The same statement was proved by Huneke-Koh \cite{HK} in positive characteristic.} For $q = r$ the situation is better; see Proposition \ref{E=O-codim} below.
\end{remark}

We also consider a related filtration. Recall that a well-known characterization of 
local cohomology (see \cite[Theorem~2.8]{Hartshorne-LC}) is
$$\cH_Z^q(\shO_X) = \underset{\underset{k}{\longrightarrow}}{\rm lim} ~\shE xt^q_{\shO_X} \big(\shO_X/ \I_Z^k, \shO_X\big),$$
where the morphisms 
\begin{equation}\label{ext-maps}
\shE xt^q_{\shO_X} \big(\shO_X/ \I_Z^k, \shO_X\big) \longrightarrow 
\shE xt^q_{\shO_X} \big(\shO_X/ \I_Z^{k+1}, \shO_X\big)
\end{equation}
between the terms in the direct limit are induced from the short exact sequence
$$0 \longrightarrow \I_Z^k /\I_Z^{k+1} \longrightarrow  \shO_X /\I_Z^{k+1}  \longrightarrow \shO_X/\I_Z^k 
\longrightarrow 0.$$

\medskip

\begin{definition}\label{definition_Ext_filtration}
The \emph{Ext filtration} on $\cH^q_Z(\shO_X)$ is the increasing filtration given by 
$$E_k\cH^q_Z(\shO_X):={\rm Im} ~\big[ \shE xt^q_{\shO_X} \big(\shO_X/ \I_Z^{k+1}, \shO_X\big) \to 
\cH_Z^q(\shO_X) \big],\,\,\,\,\,\, k \ge 0.$$
\end{definition}

\smallskip

\begin{remark}
It is clear from the definitions of the order and Ext filtrations that they depend on the scheme-theoretic structure of $Z$ and not just
on the underlying set. A natural choice, which gives the deepest such filtrations, is to take $Z$ to be reduced. However, it can be convenient 
to have the flexibility of allowing filtrations associated to non-reduced schemes (for example, in the case of set-theoretic complete intersections).
\end{remark}

\begin{remark}\label{rmk:injective}
If $q = r = {\rm codim}_X(Z)$, then 
$$E_k\cH^q_Z(\shO_X) = \shE xt^r_{\shO_X} \big(\shO_X/ \I_Z^{k+1}, \shO_X\big),$$
i.e. the maps in the above definition are injective. Indeed, in this case the maps in ($\ref{ext-maps}$) are all injective, since
$$\shE xt^{r-1}_{\shO_X} \big(\I_Z^k /\I_Z^{k+1} , \shO_X\big) =0.$$
This last fact follows from the following well-known (see e.g. \cite[Proposition~1.17]{BS}):

\begin{lemma}\label{lem_BS}
If $\shF$ is a coherent sheaf on a smooth variety, then 
$$\shE xt^i_{\shO_X} (\shF, \shO_X) = 0, \,\,\,\,\,\,{\rm for~ all}\,\,\,\,i < {\rm codim}~{\rm Supp}(\shF).$$
\end{lemma}
\end{remark}

In terms of comparing these two natural filtrations, we clearly have
\begin{equation}\label{E_in_O}
E_k\cH^q_Z(\shO_X) \subseteq O_k\cH^q_Z(\shO_X), \,\,\,\,\,\, {\rm for~all} \,\,\,\, k \ge 0.
\end{equation}

\begin{remark}
This time we obviously have that $E_k\cH^q_Z(\shO_X)$ are coherent sheaves;  by Remark \ref{infinite-generated}, this means in particular that $E_k \neq O_k$ when $q >r$ and $\cH^q_Z(\shO_X)\neq 0$. On the other hand, in general it is not clear any more whether the Ext filtration
is compatible with the filtration on $\Dmod_X$.
\end{remark}

Let's try to understand the inclusion (\ref{E_in_O}) more canonically. According to the proof of \cite[Proposition 3.1(i)]{HK}, for any ideal sheaf $\mathcal{J}\subseteq \shO_X$ such that ${\rm Supp} \left(\shO_X/ \mathcal{J} \right) = Z$, there is a spectral sequence
\begin{equation}\label{spec_seq_O}
E^{p,q}_2 = \shE xt^p_{\shO_X} \big(\shO_X/ \mathcal{J},   \cH^q_Z(\shO_X)\big) \implies 
H^{p+q} = \shE xt^{p+q}_{\shO_X} \big(\shO_X/ \mathcal{J}, \shO_X\big),
\end{equation}
which is simply the spectral sequence of the composition of the functors $\cH^0_Z (-)$ and $\cH om_{\shO_X} (\shO_X/\mathcal{J}, \cdot)$, since
$$\cH om_{\shO_X} \big(\shO_X/\mathcal{J}, \cH^0_Z(\Mmod)\big) \simeq \cH om_{\shO_X} (\shO_X/\mathcal{J}, 
\Mmod),$$
for every $\shO_X$-module $\Mmod$ and $\cH^0_Z (-)$ takes injective objects to injective objects.

Hence taking $\mathcal{J} = \I_Z^{k+1}$, in general the picture is this: $O_k\cH^q_Z(\shO_X)$ is the $E^{0,q}_2$-term of this spectral sequence, we have 
$E^{0,q}_\infty \hookrightarrow E^{0,q}_2$ (as there are no non-trivial differentials coming into $E^{0,q}_r$), while
$E^{0,q}_\infty$ is a quotient of $H^q = \shE xt^{q}_{\shO_X} \big(\shO_X/ \mathcal{I}_Z^{k+1}, \shO_X\big)$ which is indentified with $E_k\cH^q_Z(\shO_X)$.

The drawbacks for both the order and the Ext filtration disappear when $q =r={\rm codim}_X(Z)$ due to the following:

\begin{proposition}\label{E=O-codim}
We have $E_k\cH^r_Z(\shO_X) = O_k\cH^r_Z(\shO_X)$ for all $k \ge 0$.
\end{proposition}

\begin{proof}
Using the spectral sequence (\ref{spec_seq_O}), since $\cH^q_Z(\shO_X) = 0$ for $q < r$, 
for any ideal sheaf $\mathcal{J}\subseteq \shO_X$ such that ${\rm Supp} \left(\shO_X/ \mathcal{J} \right) = Z$ we have
$$\cH om_{\shO_X} \big(\shO_X/ \mathcal{J},   \cH^r_Z(\shO_X)\big) \simeq 
\shE xt^r_{\shO_X} \big(\shO_X/ \mathcal{J}, \shO_X\big).$$
The statement follows again by taking $\mathcal{J} = \I_Z^{k+1}$; see also Remark \ref{rmk:injective}.
\end{proof}

In particular, for $q =r$, Proposition \ref{one_inclusion} can be reinterpreted as saying that 
\begin{equation}\label{FinE}
F_k\cH^r_Z(\shO_X) \subseteq E_k\cH^r_Z(\shO_X) \,\,\,\,\,\,{\rm for ~all}\,\,\,\,k\ge0.
\end{equation}
A natural question, potentially interesting for the study of the singularities of $Z$, is whether this extends to higher 
values of $q$ as well.

\begin{question}\label{question2}
\emph{(When) do we have inclusions $F_k \cH_Z^q(\shO_X)  \subseteq E_k \cH_Z^q(\shO_X)$ for $q> r$?} 
\end{question}

We will answer this question positively for $k =0$ and all $q$ in Proposition \ref{F0E0} below. However we first note that in the smooth case we indeed have equality between all three filtrations, as expected. If $Z$ is a smooth irreducible subvariety of $X$ of codimension $r$, then $\cH_Z^q(\shO_X)=0$ for $q\neq r$ (just as for any local complete intersection; see Remark~\ref{van_criterion}). 

\begin{example}\label{smooth_subvariety2}
If $Z$ is a smooth, irreducible subvariety of $X$ of codimension $r$, then 
$$F_k\cH^r_Z(\shO_X) =  O_k \cH_Z^r(\shO_X) = E_k \cH_Z^r (\shO_X) \,\,\,\,\,\,{\rm for~all}\,\,\,\,k \ge 0.$$ 
Indeed, we have seen the first equality in Example~\ref{smooth_subvariety} and 
the second equality follows from Proposition \ref{E=O-codim}.
\end{example}

\begin{remark}
We will see in Corollary~\ref{sm-equiv} below that if $Z$ is a singular local complete intersection in $X$, of pure codimension $r$, then 
$F_k\cH^r_Z(\shO_X) \neq  O_k \cH_Z^r(\shO_X)$ for $k\gg 0$. However, this can fail beyond the local complete intersection case,
even for nice varieties. For example, C.~Raicu pointed out to us that if $X$ is the variety of $m\times n$ matrices, with $m>n$, and
$Z$ is the subset consisting of matrices of rank $\leq p$ (so that ${\rm codim}_X(Z)=r=(m-p)(n-p)$, then one can show using
\cite[Corollary~1.6]{Perlman} that
$$F_k\cH^r_Z(\shO_X) = O_k \cH_Z^r(\shO_X)\quad\text{for all}\quad k\in\ZZ.$$
\end{remark}

\begin{remark}\label{F0_smooth}
Another way of thinking about the isomorphism
$$F_0 \cH_Z^r(\shO_X)  \simeq \shE xt^r_{\shO_X} (\shO_Z, \shO_X) \simeq \omega_Z \otimes \omega_X^{-1}$$
when $Z$ is smooth is in terms of the description in Corollary \ref{cor_F0}. Indeed, considering the log resolution of $(X, Z)$  to be the blow-up of $X$ along $Z$, it translates into the isomorphism $R^{r-1}f_* \omega_E \simeq \omega_Z$, which is well known.
\end{remark}

As promised, in general we have a positive answer to Question \ref{question2} for the lowest piece of the Hodge filtration.

\begin{proposition}\label{F0E0}
For every $q \ge 0$ we have an inclusion $F_0 \cH^q_Z(\shO_X) \subseteq E_0\cH^q_Z(\shO_X)$.
\end{proposition}

\begin{proof}
Equivalently, the statement says that there is an inclusion
$$F_{-n}\cH^q_Z(\omega_X) \subseteq E_0\cH^q_Z(\shO_X)\otimes_{\shO_X}\omega_X=
{\rm Im} ~\big[ \shE xt^q_{\shO_X} \big(\shO_Z, \omega_X\big)\to 
\cH_Z^q(\omega_X) \big].$$
This is an immediate consequence of Theorem \ref{inclusions}, in which we established the existence of commutative 
diagrams 
\begin{equation}\label{triangle}
\begin{tikzcd}
R^{q-1} f_* \omega_E \rar \drar& \shE xt^q_{\shO_X} (\shO_Z, \omega_X) \dar \\
& \cH^q_Z(\omega_X) 
\end{tikzcd}
\end{equation}
where the diagonal map is injective, identified with the inclusion $F_{-n}\cH^q_Z(\omega_X) \hookrightarrow 
\cH^q_Z(\omega_X)$ via Corollary \ref{cor_F0}.
\end{proof}

\begin{remark}[Normal schemes]
If $Z$ is normal of codimension $r$, then the inclusion
$$F_0 \cH_Z^r(\shO_X) \subseteq E_0 \cH_Z^r(\shO_X) \simeq \omega_Z \otimes \omega_X^{-1}.$$
has an alternative interpretation: since $Z$ is normal, the dualizing sheaf 
$$\omega_Z=\shE xt^r_{\shO_X} 
(\shO_Z, \shO_X) \otimes \omega_X$$ and the canonical sheaf $i_* \omega_U$, where $i\colon U\hookrightarrow Z$ is the inclusion of the smooth locus
of $Z$, are isomorphic. Since $\omega_Z \otimes \omega_X^{-1}$ is reflexive, and coincides with $F_0 \cH_Z^r(\shO_X)$ on $U$, the conclusion follows. 
\end{remark}

\subsection{The lowest term and Du Bois singularities.}
When $Z$ is a reduced divisor, it is known that the pair $(X,Z)$ is log canonical if and only if $Z$ has du Bois singularities; see \cite[Corollary 6.6]{KS2}. On the other hand, the condition of being log-canonical is equivalent to the equality 
$F_0 \shO_X(*Z) = P_0 \shO_X(*Z)$, where $P_\bullet$ is the pole order filtration, or in other words 
$I_0 (Z) = \shO_X$ in the language of Hodge ideals; see \cite[Corollary 10.3]{MP1}.

We now show that Theorem \ref{inclusions} (and the discussion in \S\ref{scn:compatibility}), together with a criterion for Du Bois singularities due to Steenbrink and Schwede, implies that in general, if $Z$ has Du Bois singularities, then $F_0 \cH_Z^q(\shO_X) = E_0 \cH_Z^q(\shO_X)$ for all $q$;
moreover, the converse holds if $Z$ is Cohen-Macaulay.

\begin{proof}[Proof of Theorem \ref{char_DuBois}]
It follows from work of Steenbrink \cite{Steenbrink} (see Theorem \ref{Steenbrink} below), that $Z$ has Du Bois singularities if and only if the canonical morphism $\shO_Z\to\derR f_*\shO_E$ is an isomorphism; see also Schwede's \cite[Theorem~4.6]{Schwede} for a more general criterion. Via duality, this is equivalent to the map
$$\alpha \colon \derR f_* \omega_E^\bullet  \to \omega_Z^\bullet$$
in the statement of Theorem \ref{inclusions} being an isomorphism, hence to the horizontal map in ($\ref{triangle}$) being an isomorphism for each $q$. This shows the first assertion.

Under the extra assumption of $Z$ being Cohen-Macaulay of pure codimension $r$, we have 
$$\shE xt^q_{\shO_X}(\shO_Z,\omega_X) = 0 \,\,\,\,{\rm for} \,\,\,\,q \neq r,$$
so by Proposition \ref{F0E0} we also have $F_0\cH^q_Z(\shO_X) = 0$ for all $q \neq r$.
Now as explained in Remark \ref{rmk:injective}, 
for $q =r$ the vertical map in ($\ref{triangle}$) is injective, so $F_{-n}\cH^r_Z(\omega_X)= E_{-n}\cH^r_Z(\omega_X)$
is equivalent (cf. Proposition \ref{compatibility}) to the canonical map 
$$\shE xt^r_{\shO_X}(\derR f_*\shO_E,\omega_X)\to \shE xt^r_{\shO_X}(\shO_Z,\omega_X)$$ 
being an isomorphism. Since all the other $\shE xt$ sheaves are zero, it follows (using Grothendieck duality) 
that the morphism $\alpha$ is an isomorphism, which as noted is equivalent to $Z$ being Du Bois.
\end{proof}

\smallskip

\begin{remark}[Non-Cohen-Macaulay case]\label{rmk_Ma}
L.~Ma has pointed out that when $Z$ is not Cohen-Macaulay, it can happen that $Z$ is not Du Bois, but 
$F_0 \cH_Z^q(\shO_X) = E_0 \cH_Z^q(\shO_X)$ for all $q$. For example, this is the case if
$Z={\rm Spec}\big(\CC[s^4,s^3t,st^3,t^4]\big)\hookrightarrow {\mathbf A}^4$. We leave the argument
for Ch.~\ref{ch:DuBois}, in which we discuss some basic facts about Du Bois complexes (see Example~\ref{eg_Ma}).
\end{remark}

On a related note, the following corollary of Theorem \ref{inclusions} recovers \cite[Theorem~B]{MSS} in the case when the ambient space $X$ is smooth; see also Remark \ref{Gorenstein} below for the general case of this result.

\begin{corollary}\label{cor_MSS}
If $Z\subseteq X$ is a closed subscheme with Du Bois singularities, then the natural maps 
$${\mathcal Ext}^q_{\shO_X}(\shO_Z,\shO_X) \to \cH^q_Z(\shO_X)$$
are injective for all $q$.
\end{corollary}

\begin{proof}
As above, if $Z$ is Du Bois the natural morphism $\shO_Z \to \derR f_* \shO_E$ is an isomorphism, and therefore 
the canonical maps
$$\shE xt^q_{\shO_X}(\derR f_*\shO_E,\omega_X)\to \shE xt^q_{\shO_X}(\shO_Z,\omega_X)$$ 
are isomorphisms for each $q$. The result then follows from Theorem \ref{inclusions}.
\end{proof}

\begin{remark}\label{Gorenstein}
The question of when this injectivity holds is asked in \cite{EMS}; see \cite{MSS} for further discussion and applications. 
In fact, as L.~Ma has pointed out, one can deduce from Corollary~\ref{cor_MSS} the full statement of \cite[Theorem~B]{MSS}
(in which the ambient variety $X$ is only assumed to be Gorenstein). Indeed, we may assume that
$X={\rm Spec}(R)$ and $Z={\rm Spec}(S)$ are affine, with $S=R/I$. 
The assertion in the corollary implies via
\cite[Proposition~2.1]{DDSM} that $S$ is $i$-cohomologically full for every $i$ (equivalently, in the language of \cite{KK}, $S$ has 
liftable local cohomology). This implies that for every maximal ideal ${\mathfrak m}$ of $R$ and every $i$ and $k$, the natural map
$H_{{\mathfrak m}}^i(R/I^k)\to H^i_{{\mathfrak m}}(R/I)$ is surjective. If $R$ is Gorenstein, then $\omega_{R_{{\mathfrak m}}}\simeq R_{\mathfrak m}$,
and local duality implies that the natural map
$${\rm Ext}_R^{n-i}(R/I,R)\to {\rm Ext}_R^{n-i}(R/I^k,R)$$
is injective, where $n={\rm dim}(R_{\mathfrak m})$. By taking the direct limit over $k$, we obtain the injectivity of
$${\rm Ext}_R^{n-i}(R/I,R)\to H_I^{n-i}(R).$$
\end{remark}

\medskip

We conclude this section by noting that a study of the equality $F_1 \cH_Z^q(\shO_X) = E_1 \cH_Z^q(\shO_X)$ should also be very interesting. Recall that in \cite[Theorem C]{MP1} it is shown that for a reduced hypersurface $D$ the equality 
$F_1 \shO_X(*D) = P_1 \shO_X(*D)$ (or equivalently $I_1 (D) = \shO_X$) implies that $D$ has rational singularities. By analogy we make the following:

\begin{conjecture}\label{F_1=E_1}
If $Z$ is a local complete intersection of pure codimension $r$ in $X$ and if
$F_1 \cH_Z^r(\shO_X) = E_1\cH_Z^r(\shO_X)$, then $Z$ has rational singularities.
\end{conjecture}

We will show in Lemma~\ref{preservation} below that in the setting of the conjecture, the condition $F_1=E_1$ implies that $F_0=E_0$ also holds. 
We also note that the assertion in the conjecture follows from the stronger Conjecture \ref{conj-BS} below.
It is an interesting question whether a similar condition on the Hodge filtration on all local cohomology sheaves $\cH^q_Z(\shO_X)$ would imply
the fact that $Z$ has rational singularities for any reduced $Z$.

\subsection{The case of local complete intersections}\label{section_lci}
All throughout this section $Z$ is assumed to be a local complete intersection subscheme of $X$, of pure codimension $r$, defined by the ideal 
$\I_Z$. In this case we have
$\cH^q_Z(\shO_X)=0$ for $q\neq r$ by Remark \ref{van_criterion}. We have seen in Propositions \ref{one_inclusion}
and \ref{E=O-codim} that 
$$F_k\cH^r_Z(\shO_X)\subseteq E_k \cH^r_Z(\shO_X) = O_k\cH^r_Z(\shO_X)$$
for all $k \ge 0$.  

We start by giving more precise descriptions of the order and Ext filtrations; even though they coincide, each of the two filtration provides interesting information. We denote by $\shN_{Z/X}$ the normal sheaf $(\I_Z/\I_Z^2)^{\vee}$.

\begin{lemma}\label{descriptionE}
For every $k\geq 0$, the quotient 
$${\rm Gr}_k^E\cH^r_Z(\shO_X)=E_k \cH^r_Z(\shO_X)/E_{k-1} \cH^r_Z(\shO_X)$$
is a locally free $\shO_Z$-module; in fact, it is isomorphic to ${\rm Sym}^{k}(\shN_{Z/X})\otimes\omega_Z\otimes\omega_X^{-1}$. 
\end{lemma}
\begin{proof}
By Remark \ref{rmk:injective} we have  $E_k\cH^r_Z(\shO_X) =  {\mathcal Ext}_{\shO_X}^r(\shO_X/\I_Z^{k+1},\shO_X)$, 
and moreover each exact sequence
$$0\to \I_Z^{k}/\I_Z^{k+1}\to\shO_X/\I_Z^{k+1}\to\shO_X/\I_Z^{k}\to 0$$
induces an exact sequence
$$0\to {\mathcal Ext}^r_{\shO_X}(\shO_X/\I_Z^{k},\shO_X)\to {\mathcal Ext}^r_{\shO_X}(\shO_X/\I_Z^{k+1},\shO_X)\to
{\mathcal Ext}^r_{\shO_X}(\I_Z^{k}/\I_Z^{k+1},\shO_X)\to 0.$$
We thus see that 
$${\rm Gr}_k^E\cH^r_Z(\shO_X)\simeq {\mathcal Ext}^r_{\shO_X}(\I_Z^{k}/\I_Z^{k+1},\shO_X)$$
$$\simeq
{\rm Sym}^k(\shN_{Z/X}^{\vee})^{\vee}\otimes {\mathcal Ext}^r_{\shO_X}(\shO_Z,\omega_X)\otimes\omega_X^{-1}
\simeq {\rm Sym}^k(\shN_{Z/X})\otimes \omega_Z\otimes\omega_X^{-1}.$$
\end{proof}

Furthermore, working locally  we may assume that $Z$ is the closed subscheme of $X$ defined by $f_1,\ldots,f_r\in\shO_X(X)$, with $f_1,\ldots,f_r$ forming a regular sequence at every point of $Z$. 
Given this, an easy computation shows that as the case of smooth subvarieties in Example~\ref{smooth_subvariety}, we have:

\begin{lemma}\label{descriptionO}
The sheaf $O_k\cH^r_Z(\shO_X)$ is generated over $\shO_X$ by the classes of
$\frac{1}{f_1^{a_1}\cdots f_r^{a_r}}$, where $a_1,\ldots,a_r\geq 1$, with $\sum_ia_i\leq k+r$. 
\end{lemma}

\medskip

By analogy with the case of hypersurfaces \cite{MP1}, one of the main questions to understand is when, given $p\geq 0$, we have 
$$F_k\cH_Z^r(\shO_X)= O_k\cH^r_Z(\shO_X) \,\,\,\,\,\,{\rm for} \,\,\,\,k\leq p.$$
We first note that it suffices to check the equality for $k = p$:

\begin{lemma}\label{preservation}
If $F_p \cH_Z^r(\shO_X)= O_p \cH^r_Z(\shO_X)$, then 
$$F_k  \cH_Z^r(\shO_X)= O_k \cH^r_Z(\shO_X) \,\,\,\,\,\,{\rm for ~all}\,\,\,\,k \le p.$$
\end{lemma}
\begin{proof}
It suffices to check this for $k = p-1$. We use the notation $F_k$ and $O_k$ for simplicity. Note first that by Proposition 
\ref{inclusion_filtration} and Proposition \ref{one_inclusion} we have
$$\\I_Z\cdot O_p = \\I_Z \cdot F_p \subseteq F_{p -1} \subseteq O_{p-1}.$$
On the other hand, a brief inspection of the concrete description of $O_k$ given in Lemma \ref{descriptionO} shows that 
$\I_Z \cdot O_p = O_{p-1}$. We conclude by combining these two facts.
\end{proof}

The next lemma shows that this question regarding the comparison between the Hodge and order filtration is interesting only if we
assume that $Z$ is reduced.

\begin{lemma}\label{case_nonreduced}
If $Z$ is non-reduced, then
$F_0\cH_Z^r(\shO_X)\neq O_0\cH^r_Z(\shO_X)$.
\end{lemma}

\begin{proof}
Let $O'_k\cH^r_Z(\shO_X)$ be the order filtration on $\cH^r_Z(\shO_X)$ corresponding to $Z_{\rm red}$. Since we have the inclusions
$$F_0\cH^r_Z(\shO_X)\subseteq O'_0\cH^r_Z(\shO_X)\subseteq O_0 \cH^r_Z(\shO_X),$$
 it is enough to show that $O'_0\cH^r_Z(\shO_X)\neq O_0\cH^r_Z(\shO_X)$.

Note that $Z$ is Cohen-Macaulay, being a local complete intersection; since it is not reduced, it is not generically reduced. After restricting to a suitable open subset, we may thus assume that $X$ is affine, with coordinates $x_1,\ldots,x_n$, such that the ideal of $Z_{\rm red}$ is generated by $x_1,\ldots,x_r$, and if we denote by $f_1,\ldots,f_r$  the generators of $\I_Z$, and write
$f_i=\sum_{j=1}^ra_{i,j}x_j$, then ${\rm det}(a_{i,j})\in (x_1,\ldots,x_r)$. The assertion in the lemma follows from the fact that via the isomorphisms
$$\cH_Z^r(\shO_X)\simeq\shO(X)_{f_1\cdots f_r}/\sum_{i=1}^r\shO(X)_{f_1\cdots\widehat{f_i}\cdots f_r}\simeq
\shO(X)_{x_1\cdots x_r}/\sum_{i=1}^r\shO(X)_{x_1\cdots\widehat{x_i}\cdots x_r}$$
given by the \v{C}ech-complex description in \S\ref{scn:LC}, the class of $\frac{1}{x_1\cdots x_r}$, which generates
$O'_0\cH_Z^r(\shO_X)$, corresponds to ${\rm det}(a_{i,j})\frac{1}{f_1\cdots f_r}$. On the other hand, 
$O_0\cH^r_Z(\shO_X)$ is generated by the class of $\frac{1}{f_1\cdots f_r}$, hence it is different from $O'_0\cH_Z^r(\shO_X)$.
\end{proof}

Before stating the next result, we note that a stronger bound will be obtained in Theorem \ref{thm_upper_bound}, however with much more work; the simple argument here is sufficient for establishing Corollary \ref{sm-equiv}.

\begin{proposition}\label{char_smooth_for_lci}
If $Z$ is not smooth, then
$$F_k\cH_Z^r(\shO_X)\subsetneq O_k\cH^r_Z(\shO_X)\quad\text{for every}\quad k\geq n-r+1.$$
\end{proposition}

\begin{proof}
We may assume that $X$ is affine and $\I_Z$ is generated by $f_1,\ldots,f_r$. Since $Z$ is not smooth, it follows that there is a point $Q\in Z$
defined by the ideal $\I_Q$ such that, after possibly renumbering and replacing $f_1$ by a linear combination of $f_1,\ldots,f_r$, we have $f_1\in\I_Q^2$. We now need to appeal to a result that will be proved later, Theorem~\ref{gen_level}, saying that the Hodge filtration on $\cH_Z^r(\shO_X)$ is generated at level $n-r$. If $k\geq n-r+1$, we thus have
$$F_k\cH_Z^r(\shO_X)\subseteq F_1\Dmod_X\cdot F_{k-1}\cH_Z^r(\shO_X).$$
Recall that $O_{k-1}\cH_Z^r(\shO_X)$ is generated by the classes of $\frac{1}{f_1^{a_1}\cdots f_r^{a_r}}$, with $a_i\geq 1$ for all $i$ and 
$\sum_ia_i\leq k-1+r$.
Moreover, using the fact that $f_1,\ldots,f_r$ form a regular sequence in $\shO_{X,Q}$, it is easy to see that these elements form a minimal system of
generators of $O_{k-1}\cH_Z^r(\shO_X)$ at $Q$. A similar assertion holds for $O_k\cH_Z^r(\shO_X)$. 

Since $f_1\in \I_Q^2$, a straightforward calculation shows that
$$F_k\cH_Z^r(\shO_X)\subseteq F_1\Dmod_X\cdot O_{k-1}\cH_Z^r(\shO_X)\subseteq \I_Q\cdot\frac{1}{f_1^{k+1}f_2\cdots f_r}+\sum_{a_1,\ldots,a_r}
\shO_X\cdot\frac{1}{f_1^{a_1}\cdots f_r^{a_r}},$$
where the last sum is over those $a_1,\ldots,a_r$ such that $a_i\geq 1$ for all $i$, with the inequality being strict for some $i\geq 2$, and such that
$\sum_ia_i=k+r$. This shows that $F_k\cH_Z^r(\shO_X)$ is a proper subset of $O_k\cH_Z^r(\shO_X)$ at $Q$.
\end{proof}

In particular the coincidence of the two filtrations characterizes smoothness, similarly to \cite[Theorem A]{MP1} for hypersurfaces
(for another approach to the same result, see Theorem~\ref{thm_upper_bound} below).

\begin{corollary}\label{sm-equiv}
$Z$ is smooth if and only if $F_k\cH_Z^r(\shO_X) =  O_k\cH^r_Z(\shO_X)$ for all $k$.
\end{corollary}

\smallskip

We formalize the discussion above by introducing a measure of singularities that will figure prominently in the study of the Du Bois complex of $Z$.

\begin{definition}\label{definition_singularity_level}
The \emph{singularity level} of the Hodge filtration on $\shH^r_Z \shO_X$ is 
$$p (Z) := {\rm sup}\{~k ~| ~ F_k \shH^r_Z \shO_X = O_k \shH^r_Z \shO_X \},$$
with the convention that $p (Z) = -1$ if equality never holds.
\end{definition}

\begin{remark}\label{indep_embedding}
The invariant $p(Z)$ only depends on $Z$, and not on the embedding in a smooth ambient variety. In order to see this, let 
us temporarily denote by $p(Z\hookrightarrow X)$ the invariant corresponding to a closed embedding in a smooth variety $X$. 
Given a closed immersion of smooth varieties $X\hookrightarrow Y$, by Remark~\ref{comparison_two_embeddings} we have
$$p(Z\hookrightarrow X)=p(Z\hookrightarrow X\hookrightarrow Y).$$
Suppose next that $Z$ is affine and consider closed immersions $i\colon Z\hookrightarrow  X$ and $i'\colon Z\hookrightarrow X'$,
with $X$ and $X'$ smooth and affine. In order to show that $p(Z\hookrightarrow X)=p(Z\hookrightarrow X')$, by the identity above we may assume that $X={\mathbf A}^m$
and $X'={\mathbf A}^n$.  There is a morphism $f\colon {\mathbf A}^m\to {\mathbf A}^n$ such that $f\circ i=i'$.
Since $({\rm id}_{{\mathbf A}^m},f)$ is a closed immersion with $ ({\rm id}_{{\mathbf A}^m},f)\circ i=(i,i') \colon X\to {\mathbf A}^m\times {\mathbf A}^n$, it follows from what we have already seen that
$p(X\hookrightarrow {\mathbf A}^m)=p(X\hookrightarrow {\mathbf A}^m\times {\mathbf A}^n)$ and we similarly see that
$p(X\hookrightarrow {\mathbf A}^m\times {\mathbf A}^n)=p(X\hookrightarrow {\mathbf A}^n)$.
Noting that given an open cover $Z=\bigcup V_i$, we have $p(Z)=\min_ip(V_i)$, we leave it as an exercise for the reader to deduce the general case of our assertion.
\end{remark}

\begin{remark}
(1) By Lemma \ref{preservation} we have $F_k \shH^r_Z \shO_X = O_k \shH^r_Z \shO_X$  for all $k \le p (Z)$.

\noindent 

\noindent
(2) Corollary \ref{sm-equiv} says that $Z$ is smooth if and only if $p (Z) = \infty$, while Lemma \ref{case_nonreduced} says that if $Z$ is not 
reduced, then $p (Z) = -1$. 

\noindent
(3) We have seen in Theorem \ref{char_DuBois} that $p (Z) \ge 0$ if and only if $Z$ has Du Bois singularities. 
Moreover, Conjecture \ref{F_1=E_1} states that if $p (Z) \ge 1$, then $Z$ has rational singularities. Conjecture \ref{conj-BS} predicts an
explicit formula for $p(Z)$ in terms of the Bernstein-Sato polynomial of $Z$.

\noindent
(4) Another interpretation of the singularity level $p(Z)$, in terms of the Hodge ideals associated to products of equations defining $Z$ is given by Proposition \ref{charact_equality} below. Note that an algorithm for computing Hodge ideals is provided in \cite{Blanco}.
\end{remark}

\begin{remark}[Bernstein-Sato polynomial]\label{rmk:BS}
For a reduced hypersurface $D\subseteq X$, we have 
$$p (D) = [ \widetilde{\alpha}(D) ] - 1,$$
where $\widetilde{\alpha} (D)$ is the \emph{minimal exponent} of $D$, i.e. the negative of the largest root of the reduced 
Bernstein-Sato polynomial $b_D( s)/(s+1)$; see \cite[Corollary 1]{Saito-MLCT}. This interpretation and its extension to $\QQ$-divisors have proven to be very useful for studying both the Hodge filtration and the minimal exponent, see \cite{MP2}, \cite{MP3}.

When $Z$ is a reduced local complete intersection, it would similarly be interesting to translate the condition 
$F_k\cH_Z^r(\shO_X)= O_k\cH^r_Z(\shO_X)$ for $k\leq p$
in terms of the Bernstein-Sato polynomial $b_{Z}(s)$ (see \cite{BMS} for its definition).
Note that by \cite[Theorem~1]{BMS}, the largest root of $b_{Z}(s)$ is $-{\rm lct}(X,Z)$, the negative of the log canonical threshold of the pair $(X,Z)$.
Since $Z$
is reduced, $-r$ is a root of $b_{Z}(s)$, and it is natural to consider the reduced version 
$\widetilde{b}_{Z}(s)=b_{Z}(s)/ (s+r)$.
If we write $\widetilde{\alpha} (Z)$ for the negative of the largest root of $\widetilde{b}_{Z}(s)$, then 
we see that ${\rm lct}(X, Z)=\min\{r,\widetilde{\alpha} (Z) \}$, and
it follows from \cite[Theorem~4]{BMS} that $\widetilde{\alpha}(Z)>r$ if and only if $Z$ has rational singularities. 

\begin{conjecture}\label{conj-BS}
If $Z$ is a reduced local complete intersection of codimension $r$, then 
$$p (Z) = \max\{[ \widetilde{\alpha}(Z)] - r,-1\}.$$
\end{conjecture}
Note that this conjecture implies Conjecture \ref{F_1=E_1}.
\end{remark}

We next give a criterion for the coincidence between the Hodge filtration and the order filtration in terms of Hodge ideals of hypersurfaces, and deduce some applications to the behavior of the singularity level. We continue to work locally, assuming that $X = {\rm Spec}(A)$ is affine and $Z$ is defined by equations 
$f_1, \ldots, f_r \in A$ which form a regular sequence at every point of $Z$. We define the ideal  
$$J_k (f_1,\ldots,f_r):=\big(f_1^{b_1}\cdots f_r^{b_r}\mid 0 \leq b_i\leq k, \,\, \sum_i  b_i = k(r-1) \big) \subseteq A.$$
We will assume that $f : =f_1\cdots f_r$ defines a reduced divisor $D$. This is a harmless assumption: we are interested in understanding $p(Z)$
and in light of Lemma~\ref{case_nonreduced}, this is only interesting when $Z$ is reduced; in this case the assumption on $f$ is always satisfied after 
possibly replacing $f_1,\ldots,f_r$ by general linear combinations.
We use the notation $I_k (f)$ for the Hodge ideals $I_k (D)$ of \cite{MP1}.

\begin{example}\label{ex:SNC}
When $x_1, \ldots, x_r$ are part of a system of algebraic coordinates, so that $f = x_1 \cdots x_r$ defines a divisor with simple normal crossings, we have 
$$I_k (f) = J_k (x_1, \ldots, x_r) \,\,\,\,\,\, {\rm for ~all~} \,\,\,\,k \ge 0,$$
by \cite[Proposition 8.2]{MP1}. 
\end{example}

\begin{lemma}\label{incl_in_J}
If $f$ defines a reduced divisor, then $I_k (f)\subseteq J_k (f_1,\ldots,f_r)$ for all $k \ge 0$.
\end{lemma}

\begin{proof}
This follows from the Restriction Theorem for Hodge ideals, applied to the morphism 
$\varphi\colon X \to \AAA^r$ defined by $(f_1, \ldots, f_r)$ and the divisor $(x_1\cdots x_r = 0) \subseteq \AAA^r$,
in combination with Example \ref{ex:SNC}. We note that the key case of the Restriction Theorem is for closed immersions,
for which we refer to \cite[Corollary~3.4]{MP4}. The case of an arbitrary morphism between smooth varieties
is an immediate consequence;
see \cite[Corollary~3.17]{MP6} for this assertion (in a more general setting).
\end{proof}

The criterion we are after is the following:

\begin{proposition}\label{charact_equality}
If $f$ defines a reduced divisor, then
for a nonnegative integer $k$, we have
$F_k \cH_Z^r(\shO_X)= O_k \cH_Z^r(\shO_X)$ if and only if  
$$J_k (f_1,\ldots,f_r)\subseteq I_k (f)+(f_1^{k+1},\ldots,f_r^{k+1})$$ 
in a neighborhood of $Z$.
\end{proposition}

\begin{proof}
Let $I = (f_1, \ldots, f_r) \subseteq A$. 
Recall that we have an exact sequence
$$\bigoplus_{i=1}^rA_{f_1\cdots\widehat{f_i}\cdots f_r}\to A_f\to H_I^r(A)\to 0,$$
with the maps being strict with respect to the Hodge filtration. We also recall from Lemma \ref{descriptionO} that
$$O_k H_I^r(A)=\bigoplus_{a_1,\ldots,a_r} A\cdot\left[\frac{1}{f_1^{a_1}\cdots f_r^{a_r}}\right],$$
where the sum is over those $a_1,\ldots,a_r$ such that $a_i\geq 1$ for all $i$ and $\sum_ia_i \le r+k $, 
while by the definition of Hodge ideals the Hodge filtration on $A_f$ can be written as
$$F_k A_f = \frac{1}{f^{k+1}}\cdot I_k(f).$$
We thus have $O_k H^r_I (A)\subseteq F_k H^r_I (A)$ if and only if for every $a_1,\ldots,a_r$ as above, we have
$$\frac{1}{f_1^{a_1}\cdots f_r^{a_r}}\in \frac{1}{f^{k+1}}\cdot I_k(f)+\sum_{i=1}^r A_{f_1\cdots\widehat{f_i}\cdots f_r}.$$

Let us first prove the ``if" part of the assertion. Given $a_1,\ldots,a_r$ as above, let $b_i=k +1-a_i$. Note that we have $b_i\geq 0$ for all $i$
and $\sum_ib_i\geq k (r-1)$. We thus have 
$$\prod_if_i^{b_i}\in J_k (f_1,\ldots,f_r)\subseteq I_k (f)+(f_1^{k +1},\ldots,f_r^{k+1}).$$
If we write $\prod_if_i^{b_i}=h+\sum_{i=1}^rf_i^{k+1}u_i$, with $h\in I_k (f)$ and $u_i\in A$, it follows that
$$\frac{1}{f_1^{a_1}\cdots f_r^{a_r}}-\frac{h}{f^{k+1}}\in \sum_{i=1}^r A_{f_1\cdots\widehat{f_i}\cdots f_r}.$$

For the proof of the ``only if" part, we simply reverse the above arguments, using the fact that if $u\in A\cdot \frac{1}{f_1^{k+1}\cdots f_r^{k+1}}$
lies in $\sum_{i=1}^rA_{f_1\cdots\widehat{f_i}\cdots f_r}$, then
$$u\in \sum_{i=1}^rA\cdot\frac{1}{f_1^{k+1}\cdots \widehat{f_i^{k+1}}\cdots f_r^{k +1}}\quad\text{in a neighborhood of}\,\,Z.$$
This follows from the lemma below.
\end{proof} 

\begin{lemma}\label{lem_charact_equality}
Let $(R,{\mathfrak m})$ be a local Noetherian ${\mathbf C}$-algebra and $f_1,\ldots,f_r$ a regular sequence of elements in ${\mathfrak m}$. 
If $p$ is a positive integer and $u\in R\cdot\frac{1}{f_1^p\cdots f_r^p}$ lies in $\sum_{i=1}^rR_{f_1\cdots\widehat{f_i}\cdots f_r}$, then
$u\in\sum_{i=1}^rR\cdot\frac{1}{f_1^p\cdots\widehat{f_i^p}\cdots f_r^p}$.
\end{lemma}

\begin{proof}
After replacing $f_1,\ldots,f_r$ by $f_1^p,\ldots,f_r^p$, we may assume that $p=1$. By assumption, we can write
$$u=\frac{h}{f_1\cdots f_r}=\sum_{i=1}^r\frac{g_i}{(f_1\cdots\widehat{f_i}\cdots f_r)^N}$$
for some nonnegative integer $N$ and some $h,g_1,\ldots,g_r\in R$. 
In this case, we have
$$h\in \big((f_1^N,\ldots,f_r^N)\colon (f_1\cdots f_r)^{N-1}\big)$$
and the assertion in the lemma follows if we have the inclusion
\begin{equation}\label{inclusion_colon}
\big((f_1^N,\ldots,f_r^N)\colon (f_1\cdots f_r)^{N-1}\big)\subseteq (f_1,\ldots,f_r).
\end{equation}
This follows from the fact that $f_1,\ldots,f_r$ forms a regular sequence: indeed, this implies that the local ring homomorphism
$$\varphi\colon S={\mathbf C}[x_1,\ldots,x_r]_{(x_1,\ldots,x_r)}\to R,\quad\varphi(x_i)=f_i$$
is flat and thus the inclusion (\ref{inclusion_colon}) follows from the corresponding inclusion in $S$, with the $f_i$ replaced by the $x_i$.
\end{proof}

Due to the fact that the behavior of Hodge ideals under restriction and deformation is quite well understood, the criterion
in Proposition \ref{charact_equality} leads to similar behavior for the singularity level of the Hodge filtration on the local cohomology
of local complete intersections.

\begin{theorem}[Restriction theorem for the singularity level]\label{thm_restriction}
Let $X$ be a smooth, irreducible variety, $Z$ a local complete intersection subscheme of $X$ of pure codimension $r$, and $H$ a smooth hypersurface in $X$ such that $Z\vert_H=Z\cap H$ is a nonempty subscheme of $H$, also of pure codimension $r$. Then:
\begin{enumerate}
\item[i)] After possibly restricting to a neighborhood of $H$, we have $p (Z) \ge p( Z\vert_{H})$.
\item[ii)] If $H$ is general (for example a general element of a base-point free linear system), then 
we also have $p (Z)\leq  p( Z|_{H})$.
\end{enumerate}
\end{theorem}

\begin{proof}
By taking an affine open cover of $X$, we may assume that $X={\rm Spec}(A)$ is affine and let $h\in A$ be an equation of $H$. 
We begin by proving i). We assume that $p(Z\vert_H)\geq 0$, since otherwise the inequality in i) is trivial.
In particular, it follows from Lemma \ref{case_nonreduced} 
that $Z\vert_H$ is reduced, hence after replacing $Z$ by a neighborhood of $H$, we have that $Z$ is reduced too. 

We choose general generators $f_1,\ldots,f_r$ for the ideal of $Z$ in $A$ such that if 
$g_i=f_i\vert_H$ and we put as before $f=f_1\cdots f_r$ and $g=g_1\cdots g_r$, then the divisors defined by $f$ and $g$ in $X$ and $H$, respectively,
are reduced. In this case, the Restriction Theorem for Hodge ideals 
\cite[Theorem A]{MP4} gives
\begin{equation}\label{eq1_restriction_thm}
I_q(g)\subseteq I_q(f)\cdot (A/h) \quad\text{for all}\quad q\geq 0.
\end{equation}
Let $q=p( Z\vert_{H})$. By Proposition~\ref{charact_equality}, we then have
\begin{equation}\label{eq2_restriction_thm}
J_q(g_1,\ldots,g_r)\subseteq I_q(g)+(g_1^{q+1},\ldots,g_r^{q+1}).
\end{equation}
Since $J_q(g_1,\ldots,g_r)=J_q(f_1,\ldots,f_r)\cdot (A/h)$, it follows from (\ref{eq2_restriction_thm}) and (\ref{eq1_restriction_thm}) that
\begin{equation}\label{eq3_restriction_thm}
J_q(f_1,\ldots,f_r)\subseteq I_q(f)+(f_1^{q+1},\ldots,f_r^{q+1})+(h).
\end{equation}

We claim that in some neighborhood of $Z\vert_H$, we have in fact
\begin{equation}\label{eq4_restriction_thm}
J_q(f_1,\ldots,f_r)\subseteq I_q(f)+(f_1^{q+1},\ldots,f_r^{q+1})+h\cdot \big(J_q(f_1,\ldots,f_r)+(f_1^{q+1},\ldots,f_r^{q+1})\big).
\end{equation}
Indeed, if $u\in J_q(f_1,\ldots,f_r)$, then it follows from (\ref{eq3_restriction_thm}) that we can write $u=v+w+hQ$,
where $v\in I_q(f)$, $w\in (f_1^{q+1},\ldots,f_r^{q+1})$, and $Q\in A$. Since $I_q(f)\subseteq J_q(f_1,\ldots,f_r)$ by (\ref{incl_in_J}), we conclude that
$hQ\in J_q(f_1,\ldots,f_r)+(f_1^{q+1},\ldots,f_r^{q+1})$. Moreover $h,f_1,\ldots,f_r$ form a regular sequence at every point of $Z\vert_H$, so we conclude that $Q\in J_q(f_1,\ldots,f_r)+(f_1^{q+1},\ldots,f_r^{q+1})$ in a neighborhood of $Z\vert_H$. (For example, we could argue as in the proof of Lemma~\ref{lem_charact_equality} to reduce this to the case when $f_1,\ldots,f_r,h$ are variables $x_1,\ldots,x_r,x_{r+1}$ in a polynomial ring.) 

The inclusion (\ref{eq4_restriction_thm}) gives by Nakayama's lemma 
$$J_k(f_1,\ldots,f_r)\subseteq I_k(f)+(f_1^{q+1},\ldots,f_r^{q+1})\quad\text{in an open neighborhood}\,\,V\,\,\text{of}\,\,Z\vert_H.$$
Proposition~\ref{charact_equality} then implies that if we consider the neighborhood $U=V\cup (X\smallsetminus Z)$ of $H$, then
$p(Z\cap U)\geq q$. This completes the proof of i).

The proof of ii) is easier: note first that we may assume that $Z$ is reduced, since otherwise the inequality to prove is trivial. Since $H$ is general,
it follows that $Z\vert_H$ is reduced too. Using the previous notation, the inequality thus follows 
from Proposition~\ref{charact_equality} and the fact that if $H$ is general, then $I_q(g)=I_q(f)\cdot (A/h)$; see again \cite[Theorem A]{MP4}.
\end{proof}

\begin{theorem}[Semicontinuity theorem for the singularity level]\label{thm_semicont}
Let $\pi\colon {\mathcal X}\to T$ be a smooth morphism of complex algebraic varieties, and for every $t\in T$ let ${\mathcal X}_t=\pi^{-1}(t)$.
Suppose that $f_1,\ldots,f_r\in\shO_{\mathcal X}({\mathcal X})$
define a closed subscheme ${\mathcal Z}$ of ${\mathcal X}$ such that for every $t\in T$, the restriction ${\mathcal Z}_t={\mathcal Z}\cap 
{\mathcal X}_t$ is a closed subscheme of ${\mathcal X}_t$ of pure codimension $r$. Let $s\colon T\to {\mathcal X}$ be a section of $\pi$ such that $s(T)\subseteq V(f_1,\ldots,f_r)$. 
For every $t_0\in T$ there is an open neighborhood $U$ of $t_0$ such that for all $t\in U$ we have 
$$p ({\mathcal Z}_{t}) \ge p ({\mathcal Z}_{t_0}),$$ 
where each $p ({\mathcal Z}_{t})$ is considered in a small neighborhood of $s(t)$.
\end{theorem}

\begin{proof}
\emph{Step 1}.  We show that if $T$ is smooth, then there is a nonempty open subset $W$ of $T$ such that for all $t\in W$, the equalities 
\begin{equation}\label{eqn:level1}
F_q\cH_{{\mathcal Z}_{t}}^r(\shO_{{\mathcal X}_{t}})=O_q\cH_{{\mathcal Z}_{t}}^r(\shO_{{\mathcal X}_{t}})
\,\,\,\,\,\,{\rm for ~all} \,\,\,\, q \le p ({\mathcal Z}_{t_0})
\end{equation}
hold in a neighborhood of $s(t)$. Note to begin with that after applying the first assertion in Theorem~\ref{thm_restriction} 
$\dim(T)$-times, we deduce from the hypothesis that
$$F_q\cH_{\mathcal Z}^r(\shO_{\mathcal X})=O_q\cH_{\mathcal Z}^r(\shO_{\mathcal X})
\,\,\,\,\,\,{\rm for ~all} \,\,\,\, q \le p ({\mathcal Z}_{t_0})$$
in some open neighborhood $V$ of $s(t_0)$. 
Using now the second assertion in Theorem~\ref{thm_restriction}, we can choose an open subset $T_0\subseteq T$
such that for every $t\in T_0$, we have $p({\mathcal Z}_t\cap V)\geq p({\mathcal Z}\cap V)$. This implies that if
we take $W=T_0\cap s^{-1}(V)$, then the equalities 
($\ref{eqn:level1}$) 
hold on $V\cap {\mathcal X}_t$ for all $t\in W$. This completes the proof of Step 1.

\noindent
\emph{Step 2}. We next show that the assertion in Step 1 holds for every (irreducible) variety $T$. 
Given such $T$, consider a resolution of singularities $T'\to T$ and the induced family $\pi'\colon {\mathcal X}'={\mathcal X}\times_TT'\to T'$,
as well as the induced section $s'\colon T'\to {\mathcal X}'$. If $t'_0\in T'$ lies over $t_0$, applying Step 1 we can find an open subset $W'\subseteq T'$
such that for all $t'\in W'$, the equalities ($\ref{eqn:level1}$)  hold
in a neighborhood of $s'(t')$. We can then simply take $W$ to be any open subset of $T$ contained in the image of $W'$.

\noindent
\emph{Step 3}. We now prove the general statement by induction on $\dim(T)$. If $\dim(T)=1$, then we are done: if $W$ is an open subset as in Step 2, then $U=W\cup\{t_0\}$ is an open neighborhood of $t_0$ that satisfies the assertion. Suppose now that ${\rm dim}(T)\geq 2$, and let $W\subseteq T$ be an open subset as in Step 2. If $t_0\in W$, then we are done by taking $U=W$. Suppose now that $t_0\not\in W$.
After possibly removing from $T$ the irreducible components of $Z=T\smallsetminus W$ that do not contain $t_0$, we may assume that the
 irreducible
components $Z_1,\ldots,Z_m$ of $Z$ satisfy $t_0\in Z_i$ for all $i$. Applying the inductive hypothesis for every morphism $\pi^{-1}(Z_i)\to Z_i$, 
we find open neighborhoods $W_i$ of $t_0$ in $Z_i$ such that for every $t\in W_i$, the equalities ($\ref{eqn:level1}$) 
hold in a neighborhood of 
$s(t)$. In this case $U=W\cup\left(Z\smallsetminus \bigcup_{i=1}^m(Z_i\smallsetminus W_i)\right)$ is an open neighborhood of $t_0$ that satisfies the assertion in the theorem.
\end{proof}

We next use our results on the behavior of $p(Z)$ to give an upper bound for this invariant when $Z$ is singular; this can be seen as a generalization of \cite[Theorem A]{MP1}. We begin with the following:

\begin{definition}
Suppose that $Z$ is a closed subscheme of a smooth variety $X$, defined by the ideal $\I_Z$ and let $x\in Z$. For a nonzero regular function $f$ defined in an open neighborhood of $x$,
we denote by ${\rm mult}_x(f)$ the multiplicity at $x$ of the hypersurface defined by $f$. We denote by $\beta_x(Z)$ the largest ${\rm mult}_x(f)$,
where $f$ is part of a minimal system of generators of $\I_{Z,x}\subseteq\shO_{X,x}$. It is straightforward to see that this only depends on the pair $(Z,x)$, and not on $X$. Note that $Z$ is smooth at $x$ if and only if $\beta_x(Z)=1$. 
\end{definition}

\begin{theorem}\label{thm_upper_bound}
If $Z$ is a reduced closed subscheme of $X$ which is a local complete intersection, and if $x\in Z$ is a singular point with 
$\beta_x(Z)=d$, then
$$p(Z)\leq \frac{\dim(Z)-d+1}{d}.$$
In particular, for every singular $Z$, we have $p(Z)\leq \frac{\dim(Z)-1}{2}$.
\end{theorem}

\begin{proof}
After possibly replacing $X$ by a suitable neighborhood of $x$, we may assume that $X$ is affine, and that we have a system of coordinates $x_1,\ldots,x_n$, centered at $x$, and defined globally on $X$. We may also assume that $Z$ is defined by a regular sequence $f_1,\ldots,f_r\in\shO_X(X)$
and ${\rm mult}_x(f_1)=d$. After a suitable change of coordinates, we may assume that $f_1,x_2,\ldots,x_r$ form a regular sequence in a neighborhood of $x$.
Let $T={\mathbf A}^2$, with coordinates $s,t$ and let
$F_2,\ldots,F_r$ be the regular functions on $X\times T$ given by $F_i=sf_i+tx_i$. 
Let ${\mathcal X}=X\times T$ and ${\mathcal Z}$ the closed subscheme of ${\mathcal X}$ defined by $f_1, F_2,\ldots,F_r$. 
We consider the open subset 
$T_0\subseteq T$ consisting of those $(s_0,t_0)\in T$ such that the fiber of ${\mathcal Z}$ over $(s_0,t_0)$ has codimension $r$ in $X$ in a neighborhood of $x$.
Note that $(1,0), (0,1)\in T_0$. We can find an open neighborhood ${\mathcal X}_0$ of $\{x\}\times T_0$ in $X\times T_0$ 
such that we can apply Theorem~\ref{thm_semicont} to the restriction $\pi\colon {\mathcal X}_0\to T_0$ of the projection and to 
${\mathcal Z}_0={\mathcal Z}\cap {\mathcal X}_0$. Note that for a general $(s_0,t_0)\in T_0$, the fiber ${\mathcal Z}_{(s_0,t_0)}$
is cut out in $X$ by one equation of multiplicity $d$ and $r-1$ equations of multiplicity 1.

It follows that in order to prove the first bound in the theorem, we may assume that ${\rm mult}_x(f_i)=1$ for every $i\geq 2$. In this case, after possibly replacing $Z$ by an open neighborhood of $x$, we have a closed embedding as a hypersurface $Z\hookrightarrow X'$ of multiplicity $d$ at $x$, where $X'$ is smooth.
In this case we have $p(Z)=[\widetilde{\alpha}(Z)]-1$, see Remark \ref{rmk:BS}, and the first bound in the theorem follows from the upper bound
$$\widetilde{\alpha}(Z)\leq\frac{\dim(X')}{d}=\frac{\dim(Z)+1}{d}$$
 for the minimal exponent (see \cite[Theorem~E]{MP2}). The second bound follows since $d\geq 2$.
\end{proof}

We deduce the following lower bound for the dimension of the singular locus of $Z$ in terms of $p(Z)$. 
The proof proceeds as in the case of hypersurfaces, which is the content of \cite[Lemma~2.1]{MOPW}.

\begin{corollary}\label{cor_upper_bound}
If $Z$ is a reduced closed subscheme of $X$ which is a local complete intersection of pure dimension 
and whose singular locus $Z_{\rm sing}$ is nonempty, then
$${\rm codim}_Z(Z_{\rm sing})\geq 2p(Z)+1.$$
\end{corollary}

\begin{proof}
We may assume that $Z$ is affine. If ${\rm dim}(Z_{\rm sing})=s$, then after successively cutting $Z$ with $s$ general hyperplanes that meet $Z_{\rm sing}$,
we obtain a smooth subvariety $Y$ of $X$ of codimension $s$, such that $Z\cap Y$ is a singular reduced local complete intersection.  On one hand, the second assertion in Theorem~\ref{thm_restriction} gives
$$p(Z\cap Y)\geq p(Z).$$
On the other hand, we deduce from Theorem~\ref{thm_upper_bound} that
$$p(Z\cap Y)\leq \frac{\dim(Z)-s-1}{2}.$$
The bound in the statement follows by combining these two inequalities.
\end{proof}

\section{Generation level and local cohomological dimension}

We now address one of  the main applications of this paper, namely a characterization of local cohomological dimension in terms 
of resolution of singularities. Before doing this, in the section below we establish the key technical tool, namely a criterion for the generation level of the Hodge filtration on the highest nontrivial local cohomology.

\subsection{The generation level of the Hodge filtration}
We fix as always a closed subscheme of a smooth, irreducible complex $n$-dimensional variety $X$, and a log resolution $f\colon Y\to X$ 
of $(X, Z)$ as at the beginning of \S\ref{scn:birational}. We denote by $j\colon U=X\smallsetminus Z\hookrightarrow X$ the inclusion.

Recall that a good filtration $F_\bullet \Mmod$ on a left $\Dmod_X$-module $\Mmod$  is said to be generated at level 
$\ell \in \ZZ$ if 
$$F_{\ell+k} \Mmod=F_k\Dmod_X\cdot F_{\ell} \Mmod \quad\text{for all}\quad k\geq 0.$$
Equivalently, this means that for every $m \ge \ell$ we have
$$F_{m + 1} \Mmod =F_1\Dmod_X\cdot F_m \Mmod.$$

\begin{lemma}\label{lem:gen}
The filtration $F_\bullet \Mmod$ is generated at level $k$ if and only if 
$$\cH^0{\rm Gr}^F_{i-n}{\rm DR}_X(\Mmod,F)=0 \,\,\,\,\,\,{\rm  for~all} \,\,\,\,  i>k.$$
\end{lemma}

\begin{proof}
It is easy to check that the filtration $F_\bullet \Mmod$ is generated at level $k$ if and only if
the natural multiplication map 
$$T_X\otimes_{\shO_X} {\rm Gr}^F_{i-1} \Mmod \to {\rm Gr}^F_{i} \Mmod$$
is surjective for every $i>k$. After tensoring by $\omega_X$, this surjectivity is in turn equivalent by definition to the vanishing of 
$\cH^0{\rm Gr}^F_{i-n}{\rm DR}_X(\Mmod,F)$. 
\end{proof}

The result we are after is:

\begin{theorem}\label{gen_level}
If $q\geq 1$ is such that $\cH^j_Z(\shO_X)=0$ for all $j>q$, then the Hodge filtration on $\cH_Z^q(\shO_X)$ 
is generated at level $k\in \ZZ$ if and only if
$$R^{q-1+i}f_*\Omega_Y^{n-i}(\log E)=0\quad {\rm for ~all}\quad i>k.$$
In particular, it is always generated at level $n-q$.
\end{theorem}

A special case, which will lead to the proof of Theorem \ref{thm:lcd} in the next section, is the following:

\begin{corollary}\label{cor:LCD}
If $q\geq 1$ is such that $\cH^j_Z(\shO_X)=0$ for all $j>q$, then
$$\cH^q_Z(\shO_X)=0 \iff R^{q-1 +i} f_* \Omega_Y^{n-i} (\log E) = 0 \,\,\,\,{\rm for ~all} \,\,\,\, i \ge 0.$$
\end{corollary}
\begin{proof} 
We apply the theorem with $k =-1$. Since by Remark \ref{rem_trivial} we have $F_{-1} \cH^q_Z(\shO_X) = 0$, it follows that 
$\cH^q_Z(\shO_X)$ is generated at level $-1$ if and only if it is zero.
\end{proof} 

A key point in the proof of Theorem \ref{gen_level} is a formula for the graded pieces of the de Rham complex of $j_*\QQ^H_U[n]$.

\begin{lemma}\label{lem_grDR}
For every $i\in\ZZ$, we have an isomorphism in ${\bf D}^b \big({\rm Coh} (X)\big)$:
$${\rm Gr}^F_{i-n}{\rm DR}_X\big(j_*\QQ^H_U[n]\big)\simeq \derR f_*\Omega_Y^{n-i}(\log E)[i].$$
\end{lemma}

\begin{proof}
We use the approach and notation in Section~\ref{scn:birational}. Recall that we have an isomorphism
\begin{equation}\label{eq1_lem_DB_DR}
j_*\QQ_U^H[n]\simeq f_*j'_*\QQ_V^H[n].
\end{equation}
The filtered resolution (\ref{eq_filtered_res}) gives an isomorphism 
$${\rm Gr}^F_{i-n}{\rm DR}_Y\big(j'_*\QQ^H_V[n]\big)\simeq \Omega_Y^{n-i}({\rm log}\,E)[i]$$
(see also \cite[\S6]{MP1}). Using the isomorphism (\ref{eq1_lem_DB_DR}) and Saito's strictness-type result on the commutation of the direct image functor with the graded pieces of the de Rham complex (see e.g. \cite[Section~2.3.7]{Saito-MHP}) we deduce that
$${\rm Gr}^F_{i-n}{\rm DR}_X\big(j_*\QQ_U^H[n]\big)\simeq \derR f_*{\rm Gr}^F_{i-n}{\rm DR}_Y\big(j'_*\QQ^H_V[n]\big)\simeq
\derR f_*\Omega_Y^{n-i}({\rm log}\,E)[i].$$
\end{proof}

\begin{proof}[Proof of Theorem~\ref{gen_level}]
Recall from \S\ref{scn:HF} that the filtered left $\Dmod_X$-modules $R^{q-1}j_*\shO_U$ underlie the cohomologies of $j_*\QQ_U^H[n]$,  and that they coincide with the local cohomology modules $\cH^q_Z(\shO_X)$ (modulo $\shO_X$ when $q =1$). In particular the Hodge filtration on $\cH^q_Z(\shO_X)$ is generated at level $k$
if and only if the Hodge filtration on $R^{q-1}j_*\shO_U$ is generated at level $k$. On the other hand, by Lemma \ref{lem:gen}, the latter is generated at level $k$ if and only if 
$$\cH^0{\rm Gr}^F_{i-n}{\rm DR}_X(R^{q-1}j_*\shO_U) = 0  \,\,\,\,\,\,{\rm  for~all} \,\,\,\,  i>k.$$
Therefore the statement of the theorem follows once we prove the following:

\noindent
\emph{Claim:} For $q$ as in the hypothesis, we have
$$\cH^0{\rm Gr}^F_{i-n}{\rm DR}_X(R^{q-1}j_*\shO_U)\simeq R^{q-1+i}f_*\Omega_Y^{n-i}(\log E)\quad\text{for all}\quad i\in\ZZ.
$$

To this end, we need to compare the graded quotients of the filtered de Rham complex of $j_*\QQ_U^H[n]$ with those of the filtered de Rham complexes of its cohomology sheaves. We use the spectral sequence
$$E_2^{pp'}=\cH^p{\rm Gr}_{i-n}^F{\rm DR}_X(R^{p'}j_*\shO_U)\implies \cH^{p+p'}{\rm Gr}^F_{i-n}{\rm DR}_X\big(j_*\QQ_U^H[n]\big).$$
given by ($\ref{eq_spec_seq_DR}$).
Since the de Rham complex of a $\Dmod_X$-module is supported in nonpositive cohomological degrees, we have $E_2^{pp'}=0$ if $p>0$. 
On the other hand, by hypothesis $R^{p'}j_*\shO_U=0$ for $p'\geq q$, hence $E_2^{pp'}=0$ if $p'\geq q$. First, this immediately implies that
$$E_2^{0,q-1}=E_{\infty}^{0,q-1}.$$
Second, we see that for all $p\neq 0$ we have $E_2^{p,q-1-p}=0$, hence $E_{\infty}^{p,q-1-p}=0$ as well.
Looking at all the terms for which $p + p' = q-1$, we thus have
$$\cH^0{\rm Gr}_{i-n}^F{\rm DR}_X(R^{q-1}j_*\shO_U)\simeq \cH^{q-1}{\rm Gr}^F_{i-n}{\rm DR}_X\big(j_*\QQ_U^H[n]\big)
\simeq R^{q-1+i}f_*\Omega_Y^{n-i}(\log E),$$
where the second isomorphism follows from Lemma~\ref{lem_grDR}. This proves the claim.
\end{proof}

\medskip

\begin{remark}\label{rmk_independence}
With the notation in Theorem~\ref{gen_level}, we note that each $R^if_*\Omega_Y^j({\rm log}\,E)$ is independent of the log resolution $f$.
This follows in the usual way, comparing two log resolutions with a third one that dominates both of them, using the fact that 
if $E$ is a reduced SNC divisor on the smooth variety $Y$ and if $g\colon Y'\to Y$ is a proper morphism that is an isomorphism over $Y\smallsetminus
{\rm Supp}(E)$, and such that $E'=g^*(E)_{\rm red}$ is again an SNC divisor, then for all $j$ we have
$$g_*\Omega_{Y'}^j({\rm log}\,E')=\Omega_Y^j({\rm log}\,E)\quad\text{and}\quad R^ig_*\Omega_{Y'}^j({\rm log}\,E')=0\quad\text{for}\quad i>0$$
(see \cite[Lemmas~1.2 and 1.5]{EV} or \cite[Theorem~31.1]{MP1}). 
\end{remark}

\begin{example}[Top cohomology]\label{eg_top_coh}
It follows from Theorem  \ref{gen_level} that the Hodge filtration $F_\bullet \cH_Z^n(\shO_X)$ is generated at level $0$ for every $Z \subseteq X$.
\end{example}

\begin{example}[Next to top cohomology]\label{eg_next_top_coh}
We will see in Corollary~\ref{generation_n-1} below that 
$F_\bullet \cH_Z^{n-1} (\shO_X)$ is generated at level $0$  for every $Z \subseteq X$.
\end{example}

\begin{example}[Smooth subvarieties]
The explicit description in Example~\ref{smooth_subvariety} implies that if $Z$ is a smooth, irreducible subvariety of $X$,
of codimension $r$, then the Hodge filtration on $\cH^r_Z(\shO_X)$ is generated at level $0$. 
We can also deduce this from Theorem~\ref{gen_level}. Indeed, a log resolution of $(X,Z)$ is given by the blow-up
$f\colon Y\to X$ along $Z$. If $E$ is the exceptional divisor, then the condition in Theorem~\ref{gen_level} for having generation at level $0$ is that 
$$R^{r-1+i}f_*\Omega_Y^{n-i}(\log E)=0\quad\text{for}\quad i>0.$$
This follows from the fact that all the fibers of $f$ have dimension $\leq r-1$. 
\end{example}

Here are also some classes of examples where the generation level is known for different reasons.

\begin{example}[Monomial ideals]\label{ex:monomial}
If $A=\CC[x_1,\ldots,x_n]$ and $I\subseteq A$ is a monomial ideal, then the Hodge filtration on $H^q_I(A)$ is generated at level $0$ for all $q$. As we have mentioned in Example~\ref{monomial_ideal}, the multiplication map 
$$H^q_I(A)_{u-e_i}\overset{x_i}\longrightarrow H^q_I(A)_u$$
is an isomorphism whenever $u_i\neq 0$. Moreover, an inverse is given by left multiplication with $\frac{1}{u_i}\partial_i$. 
The assertion follows immediately from this observation and the description of the Hodge filtration on $H_I^q(A)$
in Example~\ref{monomial_ideal}.
\end{example}

\begin{example}[Determinantal varieties]
Let $X\simeq \CC^{n^2}$ be the space of $n\times n$ matrices, and let $Z_k \subset X$ be the determinantal subvariety 
consisting of matrices of rank $\le k$. Perlman \cite[Corollary~1.6]{Perlman} shows that the generation level 
of $\shH^q_{Z_k} (\shO_X)$ is equal to $(n^2 - k - q)/2$. He also proves a similar result, though more technical to state, for
arbitrary matrices.
\end{example}

\subsection{A criterion for local cohomological dimension}\label{scn:lcd}
We now discuss the characterization of the local cohomological dimension 
${\rm lcd}(X, Z)$ of a closed subscheme $Z$ in a smooth complex variety $X$ of dimension $n$ in terms of 
a log resolution $f \colon Y \to X$ of the pair $(X, Z)$.  We assume that $f$ is an isomorphism over $X \smallsetminus Z$; as always, we denote $E = f^{-1}(Z)_{\rm red}$. The criterion is stated as Theorem \ref{thm:lcd} in the Introduction.
Given what was shown in the previous section, the proof is now immediate:



\begin{proof}[Proof of Theorem \ref{thm:lcd}]
We argue by descending induction on $c$, the case $c\geq n$ (when both conditions are clearly satisfied) being clear.
Suppose now that $c\geq 1$ and we know the assertion for $c+1$. Then the condition 
$$R^{j + i} f_* \Omega_Y^{n-i} (\log E) = 0 \,\,\,\,\,\,{\rm for ~all}\,\,\,\,j \ge c, \, i \ge 0$$
is equivalent to 
$${\rm lcd}(Z, X) \le c + 1 \,\,\,\,{\emph and} \,\,\,\, R^{c + i} f_* \Omega_Y^{n-i}(\log E) = 0 \,\,\,\,{\rm for~all}\,\,\,\, i \ge 0.$$
On the other hand, Corollary \ref{cor:LCD} shows that this is in turn equivalent to the condition 
${\rm lcd}(Z, X) \le c$.
\end{proof}

\begin{example}\label{first-cases-lcd}
We write down the criterion in the theorem explicitly in the first few cases. Note that we remove the 
terms of the form $R^n f_* (-)$ from the list in Theorem \ref{thm:lcd}, since these vanish automatically.

\smallskip

\begin{enumerate}
\item $~~~~~~~~{\rm lcd}(X, Z) \le n-1 \iff R^{n-1} f_* \omega_Y (E) = 0$.
\bigskip
\item  $~~~~~~~~{\rm lcd}(X, Z) \le n-2 \iff
\begin{cases} 
      R^{n-1}f_* \Omega_Y^{n-1} (\log E) = 0 \\
      R^{n-1} f_* \omega_Y (E) = 0 \\
     R^{n-2} f_* \omega_Y (E) = 0
   \end{cases}
$.
\end{enumerate}
\end{example}

\begin{remark}[Ogus' Theorem and comparison]\label{rmk:Ogus}
A previous full characterization of local cohomological dimension was provided by Ogus \cite[Theorem 2.13]{Ogus}.
His criterion is of a quite different topological flavor, in terms of Hartshorne's local algebraic de Rham cohomology (which on complex algebraic varieties can be identified with singular cohomology). As Ogus observes, it depends crucially on the behavior of local de Rham cohomology at non-closed points as well; he notes in \emph{loc. cit.} (see the remark before Example 2.17) that it is desirable to have a criterion that works at a single point and does not depend on generizations. Theorem \ref{thm:lcd} here provides such a criterion, expressed in terms of finitely many coherent algebraic sheaves.

This being said, Ogus' and our conditions should be equivalent. This seems quite delicate; at the moment we do not understand this even in relatively simple cases, cf. for instance Example \ref{Ogus-unresolved} below. The connection may have to do with properties of the (local) cohomology of Du Bois complexes yet to be discovered.
\end{remark}

\begin{remark}
B.~Bhatt and M.~Saito independently pointed out to us that one can use the Riemann-Hilbert correspondence to give the following characterization
of local cohomological dimension in terms of the perverse cohomology of the constant sheaf ${\mathbf C}_{Z^{\rm an}}$:
$${\rm lcd}(X, Z)={\rm dim}(X)-\min\big\{j\in {\mathbf Z}\mid {}^{p}{\mathcal H}^j({\mathbf C}_{Z^{\rm an}})\neq 0\big\}.$$
See \cite{Bhatt+} and \cite{Saito-LCD}.
For example, this is used in \cite{Bhatt+} in order to give another proof of Ogus' characterization of local cohomological dimension.
\end{remark} 

\begin{remark}[Analytic setting]\label{rmk:analytic2}
While this paper is written in the language of algebraic varieties for uniformity, we point out that the characterization of 
local cohomological dimension in Theorem \ref{thm:lcd} (as well as Theorem \ref{local-vanishing} on local vanishing) holds 
when $Z$ is an analytic subspace of a complex manifold $X$ as well. The reason is that all arguments, including Theorem \ref{gen_level}, are based on the construction and formal properties discussed in \S\ref{scn:birational}, which work in this setting; see Remark \ref{rmk:analytic1}.
\end{remark}



\smallskip

Before moving on to applications, we recall that much of the focus in the literature is on bounds on ${\rm lcd}(X, Z)$ in terms of ${\rm depth}(\shO_Z)$. We note that in our context this depth has a clear role related to the vanishing of the first step of the Hodge filtration:

\begin{lemma}\label{lem:depth}
For every $Z$, we have 
$$F_0 \shH^q_Z (\shO_X) = 0\quad\text{for all}\quad q > n-{\rm depth}(\shO_Z)={\rm pd}(\shO_Z).$$
Moreover, if $Z$ has du Bois singularities, then 
$$F_0 \shH^q_Z (\shO_X) \neq 0\quad\text{for}\quad q = n-{\rm depth}(\shO_Z)={\rm pd}(\shO_Z).$$
\end{lemma}

\begin{proof}
The equality $n-{\rm depth}(\shO_Z)={\rm pd}(\shO_Z)$ is given by the Auslander-Buchsbaum formula. Then a well-known
characterization of ${\rm pd}(\shO_Z)$ gives
$$n-{\rm depth}(\shO_Z)=\max\big\{q\, \vert\, {\mathcal Ext}^q_{\shO_X}(\shO_Z,\shO_X) \neq 0\big\}.$$
The first assertion is then an immediate application of Corollaries~\ref{cor_F0} and \ref{KS-injection}.
For the converse in the case of Du Bois singularities, we use in addition Theorem \ref{char_DuBois}.
\end{proof}


It is helpful to record that this allows us to bypass Theorem \ref{gen_level} if the generation level of the Hodge filtration is a priori known to be optimal.

\begin{corollary}\label{zero-gen}
If the Hodge filtration on $\cH^q_Z (\shO_X)$ is generated at level $0$, then 
$$ \cH^q_Z (\shO_X) = 0 \iff R^{q -1} f_* \omega_E =0.\footnote{Recall that for $q \ge 2$ we have an isomorphism
$R^{q -1} f_* \omega_E \simeq R^{q -1} f_* \omega_Y (E)$.}$$
In particular, if the hypothesis holds for the Hodge filtration on all $\cH^q_Z (\shO_X)$, then
$${\rm lcd}(X, Z) \le n-{\rm depth}(\shO_Z).$$
\end{corollary}

\begin{proof}
This follows immediately from Lemma~\ref{lem:depth} and the results cited in its proof. 
\end{proof}

\begin{remark}\label{rmk_zero-gen}
Note that if $Z$ is either Cohen-Macaulay or has Du Bois singularities, then 
$${\rm lcd}(X, Z)\geq n - {\rm depth}(\shO_Z).$$
In the Du Bois case, this follows directly from
Lemma~\ref{lem:depth}. On the other hand, if $Z$ is Cohen-Macaulay,
then $n - {\rm depth}(\shO_Z)=n-{\rm dim}(Z)={\rm codim}_X(Z)$ and
the inequality follows from Remark~\ref{van_criterion}.
\end{remark}

In order to state the consequences of Theorem \ref{thm:lcd} in the desired level of generality, it is convenient to also consider the local cohomological dimension
at a (possibly non-closed) point $\xi\in Z$. We define 
$${\rm lcd}_{\xi}(X,Z):=\max_{\xi\in U}~{\rm lcd}(U,Z\cap U),$$
where the maximum is taken over the open neighborhoods $U$ of $\xi$. 
Recall that the support of every local cohomology sheaf $\cH^q_Z(\shO_X)$ is closed: this follows for instance from the fact that 
$\cH^q_Z(\shO_X)$ has the structure of a coherent $\Dmod_X$-module, hence its support is the image of its characteristic variety, which is a conical subvariety of the cotangent bundle $T^* X$. 
Since local cohomology commutes with localization, we see that if $R=\shO_{X,\xi}$ and ${\mathfrak a}= \I_Z\cdot R$, then
$${\rm lcd}_{\xi}(X,Z)=\max\{q\mid H^q_{\mathfrak a}(R)\neq 0\}.$$
In particular, we have ${\rm lcd}_{\xi}(X,Z)\leq {\rm codim}_X(\xi)$.
It is clear that we also have
$${\rm lcd}(X,Z)=\min_{\xi\in Z}~{\rm lcd}_{\xi}(X,Z)= \min_{x\in Z}~{\rm lcd}_x(X,Z),$$
where $\xi$ runs over all (possibly non-closed) points of $Z$ and $x$ runs over all closed points of $Z$. 
For instance, here is Example \ref{first-cases-lcd} revisited in this more general setting.

\begin{example}\label{first-cases-lcd-2}
Let $\xi\in Z$ be a point with ${\rm codim}_X(\xi)=r$. Note that for every $p$ and every $q\geq r$ we have $R^qf_*\Omega_Y^p({\rm log}\,E)=0$ in some neighborhood of $\xi$ (in which the fibers of $f$ have dimension $\leq r-1$).

\smallskip

\begin{enumerate}
\item $~~~~~~~~{\rm lcd}_{\xi}(X, Z) \le r-1 \iff R^{r-1} f_* \omega_Y (E)_{\xi} = 0$.
\bigskip
\item  $~~~~~~~~{\rm lcd}_{\xi}(X, Z) \le r-2 \iff
\begin{cases} 
      R^{r-1}f_* \Omega_Y^{n-1} (\log E)_{\xi} = 0 \\
      R^{r-1} f_* \omega_Y (E)_{\xi} = 0 \\
     R^{r-2} f_* \omega_Y (E)_{\xi} = 0
   \end{cases}
$.
\end{enumerate}
\end{example}

Moving on to applications, Theorem \ref{thm:lcd} leads to previously known results on local cohomological dimension, as well as to new results, in a unified fashion. We organize this in decreasing order of the possible values of ${\rm lcd}(X, Z)$.

\noindent
{\bf Cohomological dimension $n-1$.}
Here we obtain rather quickly the following special case (i.e. when $X$ is smooth) of a well-known result; see \cite[Theorem 3.1]{Hartshorne}, \cite[Corollary 2.10]{Ogus}.

\begin{corollary}[Hartshorne-Lichtenbaum Theorem, smooth case]\label{cor_HL}
We have 
$${\rm lcd}(X,Z)\leq n-1 \iff Z {\rm ~has~ no~ isolated ~points}.$$
More generally, if $\xi\in Z$ is a point with ${\rm codim}_X(\xi)=r$, then
${\rm lcd}_{\xi}(X, Z) \le r-1$ if and only if $\overline{\{\xi\}}$ is not an irreducible component of $Z$. 
\end{corollary}

\begin{proof}
The first assertion is a special case of the second one, by letting $\xi$ run over the closed points of $Z$.
We thus focus on the second assertion. Consider the local ring $R=\shO_{X,\xi}$, with maximal ideal ${\mathfrak m}$, and ${\mathfrak a}=\I_Z\cdot R$. If $\overline{\{\xi\}}$ is an irreducible component of 
$Z$, then ${\mathfrak a}$ is ${\mathfrak m}$-primary, hence $H^r_{\mathfrak a}(R)=H_{\mathfrak m}^r(R)\neq 0$. 

Suppose now that $W:=\overline{\{\xi\}}$ is not an irreducible component of $Z$. In particular, we have $r\geq 2$.
By Example \ref{first-cases-lcd-2}(1), in order to show that ${\rm lcd}_{\xi}(X,Z)\leq r-1$, it is enough to show that $ R^{r-1} f_* \omega_Y (E)_{\xi} = 0$.

After possibly replacing $X$ by a suitable open neighborhood of $\xi$, we may assume that the image of every irreducible component of $E$ contains $W$. 
We write $E=E_1+\cdots+E_m+F$, where the $E_i$ are the prime components of $E$ such that $f(E_i)=W$. 
Since $W$ is not an irreducible component of $Z$, it follows that $f(F)=Z$.
Moreover, since the fiber $f^{-1}(\xi)$ is connected, it follows that after possibly reordering the $E_i$, we may assume that for every $i$, with $1\leq i\leq m$,
the intersection $E_i\cap (E_{i+1}+\cdots+E_m+F)$ dominates $W$. 

If $m\geq 1$, let us write $E=E_1+E'$. The short exact sequence
$$0\to \omega_Y(E')\to\omega_Y(E)\to \omega_{E_1}(E'\vert_{E_1})\to 0$$
gives an exact sequence 
$$R^{r-1}f_*\omega_Y(E')\to R^{r-1}f_*\omega_Y(E)\to R^{r-1}f_*\omega_{E_1}(E'\vert_{E_1}).$$
Note that $R^{r-1}f_*\omega_{E_1}(E'\vert_{E_1})_{\xi}=0$. Indeed, $E_{1,\xi}:=E_1\times_W{\rm Spec}\,\CC(\xi)$
is a smooth projective variety over ${\rm Spec}\,\CC(\xi)$ of dimension $r-1$, 
and the pull-back of $E'\vert_{E_1}$ to $E_{1,\xi}$ is a nonzero effective Cartier divisor $T$, so that
$$R^{r-1}f_*\omega_{E_1}(E'\vert_{E_1})_{\xi}\simeq H^{r-1}\big(E_{1,\xi},\omega_{E_{1,\xi}}(T)\big)\simeq
 H^0\big(E_{1,\xi},\shO_{E_{1,\xi}}(-T)\big)^{\vee}=0.$$

We thus conclude that it is enough to show that $R^{r-1}f_*\omega_Y(E')_{\xi}=0$. 
Iterating the above argument, we see that it is therefore enough to show that $R^{r-1}f_*\omega_Y(F)_{\xi}=0$. 
The short exact sequence
$$0\to\omega_Y\to\omega_Y(F)\to\omega_F\to 0$$
and Grauert-Riemenschneider vanishing (recall that $r\geq 2$) finally imply that it suffices to show that $R^{r-1}f_*\omega_F=0$
in a neighborhood of $\xi$.
This follows from the fact that there is such a neighborhood over which all fibers of $F\to W$ have dimension $\leq r-2$.
\end{proof}

\noindent
{\bf Cohomological dimension $n-2$.}
Without any extra work, Theorem  \ref{gen_level} guarantees that if $\cH^n_Z(\shO_X) = 0$, the Hodge filtration on $\cH^{n-1}_Z (\shO_X)$ is generated at level $1$. It turns out that things are in fact always better:

\begin{corollary}\label{generation_n-1}
The Hodge filtration on $\cH^{n-1}_Z (\shO_X)$ is generated at level $0$.
\end{corollary}

\begin{proof} 
We may assume $n\geq 2$, since the assertion is trivial if $n=1$. 
Around an isolated point of $Z$, we have $\cH^i_Z(\shO_X)=0$ for all $i<n$, hence the assertion is clear.
On the other hand, on the complement of the isolated points of $Z$, we have $\cH^n_Z(\shO_X)=0$ by Corollary~\ref{cor_HL}.
We can thus apply Theorem  \ref{gen_level} and the only vanishing that needs to be checked is
$$R^{n-1} f_* \Omega^{n-1}_Y (\log E)=0.$$
This is a general phenomenon for any $Z$ and any log resolution, the subject of Theorem~\ref{n-1} which will be treated separately at the end of this section.
\end{proof}

This leads us to the following alternative to the characterization of Hartshorne and Ogus \cite[Corollary 2.11]{Ogus} (see also 
Remark \ref{Ogus-unresolved}):

\begin{corollary}\label{n-2}
If $Z$ has no isolated points, then 
 $${\rm lcd}(X, Z) \le n-2 \iff R^{n-2} f_* \omega_Y (E) = 0.$$
 More generally, if $\xi\in Z$ is a point with ${\rm codim}_X(\xi)=r$ and such that $\overline{\{\xi\}}$ is not an irreducible
 component of $Z$, then 
 $${\rm lcd}_{\xi}(X, Z) \le r-2 \iff R^{r-2} f_* \omega_Y (E)_{\xi} = 0.$$
\end{corollary}

\begin{proof}
The first part follows immediately from  Corollaries~\ref{zero-gen} and \ref{generation_n-1}.

More generally, for the second assertion, by hypothesis and Corollary~\ref{cor_HL}, we already know that $R^{r-1} f_* \omega_Y (E)_{\xi} = 0$. In view of Example \ref{first-cases-lcd-2}(2), we are done if we have $R^{r-1} f_* \Omega^{n-1}_Y (\log E)_{\xi}=0$ as well.  This is again part of the general Theorem \ref{n-1} below.
\end{proof}

As an immediate consequence of Corollary~\ref{n-2} and of a suitable extension 
of Lemma \ref{lem:depth}, we recover Ogus' result \cite[Remark p. 338-339]{Ogus}:

\begin{corollary}[Ogus' Theorem]
If $\xi\in Z$ is such that ${\rm codim}_X(\xi)=r$ and we have 
${\rm depth}(\shO_{Z,\xi}) \ge 2$, then ${\rm lcd}_{\xi}(X, Z) \le r-2$.
\end{corollary}

This happens for instance if $Z$ is normal and $\xi$ is a point of codimension at least $2$. 

\begin{remark}\label{Ogus-unresolved}
This case of local cohomological dimension $n-2$ provides the first instance of the intricacy of the comparison with Ogus' theorem
discussed in Remark \ref{rmk:Ogus}. Recall that, in answer to a conjecture of Hartshorne, Ogus showed in \cite[Corollary 2.11]{Ogus} that, say at a closed point $x \in Z$, one has ${\rm lcd}_x(X, Z) \le n-2$ if and only if the punctured spectrum
${\rm Spec}(\shO_{Z,x}) \smallsetminus \{\mathfrak{m}\}$ is formally geometrically connected, which in turn is equivalent to the vanishing of the local singular cohomology $H^1_x (Z, \CC)$. Thus the equivalence between our criterion in Corollary \ref{n-2} and Ogus' criterion becomes, with the notation above:
$$R^{n-2} f_* \omega_Y (E)_x = 0 \iff H^1_x (Z, \CC) = 0.$$
With the currently available methods, at the moment we only know how to prove that our condition implies Ogus', i.e. the implication from left to right.
\end{remark}

To complete this circle of ideas, we need to establish the missing ingredient in the proof of Corollary \ref{generation_n-1} and Corollary \ref{n-2}. We show that 
$$R^{n-1} f_* \Omega^j_Y (\log E)=0 \quad\text{for all}\quad j\leq n-1.$$
for any proper closed subscheme $Z$ of a smooth $n$-dimensional variety $X$, and any log resolution of $(X, Z)$. 
In fact, we prove the following more general result taking into account non-closed points; it can be read completely independently of the rest of the paper.

\begin{theorem}\label{n-1}
With the notation above, if $\xi\in Z$ is a point with ${\rm codim}_X(\xi)=r\geq 2$, then 
$$R^{r-1}f_*\Omega^j_Y(\log E)_{\xi}=0\quad\text{for all}\quad j\leq r-1.$$
Moreover, if $\overline{\{\xi\}}$ is not an irreducible component of $Z$,
then the same vanishing holds for all $j\in {\mathbf Z}$.
\end{theorem}

Before giving the proof, we need some preparations. We begin with a lemma that allows us to reduce the proof
of statements like the one in Theorem~\ref{n-1} to the case when $\xi$ is a closed point. 
We consider the following setup.
Suppose that $H$ is a general member of a base-point free linear system on $X$ and $g\colon H_Y\to H$ is the morphism induced 
by $f$, where $H_Y=f^*H$. Note that by the Kleiman-Bertini theorem, $H$ and $H_Y$ are smooth (though possibly disconnected) and $H_Y+E$
is a simple normal crossing divisor. Therefore $g$ is a log resolution of $(H,Z\cap H)$ which is an isomorphism over $H\smallsetminus Z$, and $E\vert_{H_Y}=g^{-1}(Z\cap H)_{\rm red}$. 

\begin{lemma}\label{dim_reduction}
With the above notation, we have 
$$R^if_*\Omega^j_Y(\log E)\cdot\shO_H\simeq R^ig_*\big(\Omega_{Y}^j(\log E)\cdot\shO_{H_Y}\big)\quad\text{for all}\quad i,j.$$
Moreover, for every $i$ and every $q$, if 
$$R^ig_*\Omega_{H_Y}^j(\log E\vert_{H_Y})=0\quad\text{for all}\quad j\leq q,\quad\text{then}$$
$$R^if_*\Omega^j_Y(\log E)\cdot\shO_H=0\quad\text{for all}\quad j\leq q.$$
\end{lemma}

\begin{proof}
Let $\alpha\colon H \hookrightarrow X$ and $\beta\colon H_Y\hookrightarrow Y$ be the inclusion maps.
Since $Y$ and $H$ are Tor-independent over $X$, the canonical base-change morphism
$$\derL\alpha^*\derR f_*\Omega_Y^j(\log E)\to \derR g_*\derL \beta^*\Omega_Y^j(\log E)$$
is an isomorphism (see \cite[Lemma~36.22.5]{Stacks}). Since $H$ is general, we have
$$ {\mathcal H}^i\big(\derL\alpha^*\derR f_*\Omega_Y^j(\log E)\big)=
R^if_*\Omega_Y^j(\log E)\cdot\shO_H\quad\text{and}$$ 
$${\mathcal H}^i\big(\derR g_*\derL \beta^*\Omega_Y^j(\log E)\big)
=R^ig_*\big(\Omega_{Y}^j(\log E)\cdot\shO_{H_Y}\big).$$
This gives the first assertion in the lemma. 

For the second assertion, note that since $H_Y+E$ is a simple normal crossing divisor, we have an exact sequence
of locally free sheaves on $H_Y$:
$$0\to \shO_{H_Y}(-H_Y)\to \Omega_Y^1(\log E)\cdot\shO_{H_Y}\to \Omega_{H_Y}^1(\log E\vert_{H_Y})\to 0.$$
This implies that for every $j$, the sheaf $\Omega_Y^j(\log E)\cdot\shO_{H_Y}$ has a filtration with successive quotients
$$\shO_{H_Y}(-pH_Y)\otimes_{\shO_{H_Y}}\Omega_{H_Y}^{j-p}(\log E\vert_{H_Y})\quad\text{for}\quad 0\leq p\leq j.$$
The  statement then follows by taking the long exact sequences for higher direct images,
using the first assertion and the fact that by the projection formula, we have
$$R^ig_*\big(\shO_{H_Y}(-pH_Y)\otimes_{\shO_{H_Y}}\Omega_{H_Y}^{j-p}(\log E\vert_{H_Y})\big)
\simeq \shO_H(-pH)\otimes_{\shO_H}R^ig_*\Omega_{H_Y}^{j-p}(\log E\vert_{H_Y}).$$
\end{proof}

We next give a couple of easy lemmas regarding the push-forward of sheaves of log differentials under smooth blow-ups.
Let us fix first some notation. Suppose that $W$ is a smooth irreducible codimension $r$ subvariety of the smooth $n$-dimensional variety $X$. 
Let $\pi\colon X'\to X$ be the blow-up of $X$ along $W$, with exceptional divisor $F$. Suppose that $D$ is a reduced SNC
divisor on $X$, that also has SNC with $W$, and let $D'=\widetilde{D}+F$, where $\widetilde{D}$ is the strict transform of $D$. 

\begin{lemma}\label{lem1_n-1}
With the above notation, if $W\subseteq {\rm Supp}(D)$, then 
\[
R^i\pi_*\Omega^j_{X'}({\rm log}\,D')= \left\{
\begin{array}{cl}
\Omega_X^j({\rm log}\,D), & \text{if}\quad i=0; \\[2mm]
0, & \text{otherwise}.
\end{array}\right.
\]
\end{lemma}

\begin{proof}
This is well known, see for example \cite[Lemmas~1.2 and 1.5]{EV} or \cite[Theorem~31.1]{MP1} for a more
general statement.
\end{proof}

\begin{lemma}\label{lem2_n-1}
With the above notation, if $W\not\subseteq {\rm Supp}(D)$ and $r\geq 2$, then
\[
R^i\pi_*\Omega^j_{X'}({\rm log}\,D')= \left\{
\begin{array}{cl}
\Omega_X^j({\rm log}\,D), & \text{if}\quad i=0; \\[2mm]
\Omega_W^{j-r}({\rm log}\,D\vert_W), & \text{if}\quad i=r-1; \\[2mm]
0, & \text{otherwise}.
\end{array}\right.
\]
\end{lemma}

\begin{proof}
This is also well known, but we include an argument for the lack of a good reference. Arguing locally, we may assume that 
there is a smooth divisor $H$ containing $W$, such that  $G=D+H$ has SNC, and also has SNC with $W$; for example, if we have local algebraic coordinates
$x_1,\ldots,x_n$ on $X$ such that $W$ is defined by the ideal $(x_1,\ldots,x_r)$ and $D$ is defined by $x_{r+1}\cdots x_{r+s}$, then we may take $H$
to be defined by $x_1$. Let $G'=D'+\widetilde{H}$, where $\widetilde{H}$ is the strict transform of $H$. 
Note that if $r=2$, then the induced map $\varphi\colon \widetilde{H}\to H$ is an isomorphism and the inverse image of $W$ is $F\vert_{\widetilde{H}}$.
On the other hand, if $r\geq 3$, then $\varphi$ is the blow-up of $H$ along $W$, with exceptional divisor $F\vert_{\widetilde{H}}$ and we can apply induction for the divisor $D\vert_H$.

On $X'$ we have the short exact sequence
\begin{equation}\label{eq1_lem1_n-1}
0\to\Omega_{X'}^j({\rm log}\,D')\to \Omega_{X'}^j({\rm log}\, G')\to \Omega^{j-1}_{\widetilde{H}}({\rm log}\,D'\vert_{\widetilde{H}})\to 0.
\end{equation}
Note that by Lemma~\ref{lem1_n-1}, we have
\begin{equation}\label{eq2_lem1_n-1}
R^i\pi_*\Omega^j_{X'}({\rm log}\,G')\simeq \left\{
\begin{array}{cl}
\Omega_X^j({\rm log}\,G), & \text{if}\quad i=0; \\[2mm]
0, & \text{otherwise}.
\end{array}\right.
\end{equation}

We prove the assertion in the lemma by induction on $r\geq 2$. Suppose first that $r=2$. Pushing forward the exact sequence (\ref{eq1_lem1_n-1})
and using (\ref{eq2_lem1_n-1}), we obtain an exact sequence
$$0\to \pi_*\Omega_{X'}^j({\rm log}\,D')\to \Omega^j_X({\rm log}\,G)\overset{\alpha}\longrightarrow \Omega_H^{j-1}\big({\rm log}\,(D\vert_H+ W)\big)
\to R^1\pi_*\Omega_{X'}^j({\rm log}\,D')\to 0.$$
Note that $\alpha$ is the composition
$$\Omega^j_X({\rm log}\,G)\to \Omega_H^{j-1}({\rm log}\,D\vert_H)\to \Omega_H^{j-1}\big({\rm log}\,(D\vert_H+W)\big),$$
where the first map is surjective and the second one is injective. We thus see that
$$\pi_*\Omega^j_{X'}({\rm log}\,D')\simeq \Omega^j_X({\rm log}\,D)\quad\text{and}\quad R^1\pi_*\Omega_{X'}^j({\rm log}\,D')\simeq
\Omega_W^{j-2}({\rm log}\,D\vert_W).$$
Since $\pi$ has fibers of dimension $\leq 1$, we deduce that $R^i\pi_*\Omega_{X'}^j({\rm log}\,D')=0$ for $i\geq 2$, which completes the proof for $r=2$.

Suppose now that $r\geq 3$, and that we know the assertion for $r-1$. Pushing forward the exact sequence (\ref{eq1_lem1_n-1})
and using (\ref{eq2_lem1_n-1}) together with the inductive assumption, we first see that for $i\geq 2$, we have
$$R^i\pi_*\Omega^j_{X'}({\rm log}\,D')\simeq R^{i-1}\varphi_*\Omega^{j-1}_{\widetilde{H}}({\rm log}\,D'\vert_{\widetilde{H}})
=\left\{
\begin{array}{cl}
\Omega_Z^{j-r}({\rm log}\,D\vert_Z), & \text{if}\quad i=r-1; \\[2mm]
0, & \text{otherwise}.
\end{array}\right.
$$
We also get an exact sequence
$$0\to \pi_*\Omega_{X'}^j({\rm log}\,D')\to \Omega_X^j({\rm log}\,G)\overset{\beta}\longrightarrow 
\Omega_Z^{j-1}({\rm log}\,D\vert_H)\to R^1\pi_*\Omega_{X'}^j({\rm log}\,D')\to 0.$$
We thus conclude that
$$ \pi_*\Omega_{X'}^j({\rm log}\,D')\simeq {\rm ker}(\beta)\simeq \Omega_X^j({\rm log}\,D)\quad\text{and}\quad
R^1\pi_*\Omega_{X'}^j({\rm log}\,D')\simeq {\rm coker}(\beta)=0.$$
This completes the proof of the lemma.
\end{proof}

\begin{proof}[Proof of Theorem~\ref{n-1}]
We may and will assume that $X$ is affine, and fix a linear system of divisors on $X$ obtained by restricting 
a complete very ample linear system on a projective completion of $X$. 
Suppose first that $r<n$ and that $H$ is a general element in our linear system such that $H\cap\overline{\{\xi\}}$
is nonempty. If we knew the assertions in the theorem for $g$ and for the irreducible components of $H\cap\overline{\{\xi\}}$
(which have codimension $r$ in $H$), then we would get using Lemma~\ref{dim_reduction} that $R^{r-1}f_*\Omega_Y^j(\log E)\cdot \shO_H$
vanishes at some point of $H\cap\overline{\{\xi\}}$ for all $j$ (we assume $j\leq r-1$ if $\overline{\{\xi\}}$ is an irreducible component of $Z$). By Nakayama's lemma, we then deduce that $R^{r-1}f_*\Omega_Y^j(\log E)$ vanishes at some point of $\overline{\{\xi\}}$ and thus its stalk at $\xi$ is $0$.

After repeating this several times, we reduce to the case when $\xi$ is a closed point, hence $r=n$. Since the case $j=n$ follows from
Corollary~\ref{cor_HL}, we only need to prove
\begin{equation}\label{eq_case_closed_point}
R^{n-1}f_*\Omega_Y^j(\log E)=0\quad\text{for all}\quad j<n.
\end{equation}

Suppose that we have a sequence of morphisms: 
$$X_N\overset{g_N}\longrightarrow X_{N-1}\overset{g_{N-1}}\longrightarrow\cdots\overset{g_2}\longrightarrow X_1\overset{g_1}\longrightarrow X_0=X$$
with the following properties:
\begin{enumerate}
\item[i)] For every $i$ with $0\leq i\leq N-1$, we have a smooth irreducible subvariety $Z_i$ of $X_i$ of codimension $\geq 2$ such that $g_{i+1}$ is the blow-up of $X_i$ along $Z_i$,
with exceptional divisor $E_{i+1}$;
in particular, $X_1,\ldots,X_N$ are smooth.
\item[ii)] For every $i$ with $1\leq i\leq N$, if 
$$D_i=\big(g_i^*(D_{i-1})+E_i\big)_{\rm red}$$ 
(with the convention $D_0=0$), then $D_i$ has SNC, and also has SNC
with $Z_i$ for $i\leq N-1$.
\item[iii)] All $Z_i$ lie inside  the inverse image of $Z$. 
\end{enumerate}
For every $i$, let us denote by $f_i$ the composition $X_i\to X$.
In order to prove (\ref{eq_case_closed_point}), it is enough to show by induction on $N\geq 0$ that under these assumptions 
$$R^{n-1}(f_{N})_*\Omega_{X_N}^j({\rm log}\,D_N)=0 \,\,\,\,\,\,{\rm for} \,\,\,\, 0\leq j\leq n-1.$$ 
Indeed, we can find a sequence of morphisms as above
such that $f_N$ gives a log resolution of $(X, Z)$ as in the statement of the theorem. Since the sheaf $R^{n-1}f_*\Omega_Y^j({\rm log}\,E)$ is independent of the resolution
(see Remark~\ref{rmk_independence}), we obtain the desired vanishing. 

The case $N=0$, when $f_N$ is the identity, is trivial, hence we only need to prove the induction step. 
There are two cases to consider. If $Z_{N-1}\subseteq {\rm Supp}(D_{N-1})$, then it follows from
Lemma~\ref{lem1_n-1} that
$$R^{n-1}(f_{N})_*\Omega_{X_N}^j({\rm log}\,D_N)\simeq R^{n-1}(f_{N-1})_*\Omega_{X_{N-1}}^j({\rm log}\,D_{N-1}),$$
hence we are done by induction.

Suppose now that $Z_{N-1}\not\subseteq {\rm Supp}(D_{N-1})$ and consider the Leray spectral sequence
$$E_2^{p,q}=R^p(f_{N-1})_*R^q(g_N)_*\Omega^j_{X_N}({\rm log}\,D_N)\Rightarrow R^{p+q}(f_N)_*\Omega^j_{X_N}({\rm log}\,D_N).$$
It follows from Lemma~\ref{lem1_n-1} that the only possible nonzero terms $E_2^{p,q}$, with $p+q=n-1$, are
$$E_2^{n-1,0}=R^{n-1}(f_{N-1})_*\Omega^j_{X_{N-1}}({\rm log}\,D_{N-1})\quad\text{and}$$
$$E_2^{n-r,r-1}=R^{n-r}(f_{N-1})_*\Omega_{Z_{N-1}}^{j-r}({\rm log}\,D_{N-1}\vert_{Z_{N-1}}),$$
where $r={\rm codim}_{X_{N-1}}(Z_{N-1})$.
We see that $E_2^{n-1,0}=0$ by induction. Since $Z_{N-1}\not\subseteq {\rm Supp}(D_{N-1})$, the induced morphism
$Z_{N-1}\to f_{N-1}(Z_{N-1})$ is birational. If $\dim(Z_{N-1})=n-r$ is positive, then this morphism has fibers of dimension $\leq n-r-1$,
hence $E_2^{n-r,r-1}=0$. On the other hand, if $r=n$, then $j-r<0$, hence again $E_2^{n-r,r-1}=0$. We thus conclude from the spectral sequence
that $R^{n-1}(f_{N})_*\Omega_{X_N}^j({\rm log}\,D_N)=0$, completing the proof of the theorem.
\end{proof}

\noindent
{\bf Cohomological dimension $n-3$.}
The pattern of the previous two paragraphs continues for one more step, by means of the following result 
of Dao-Takagi \cite[Corollary 2.8]{DT}, conjectured and proved in a more restrictive setting by Varbaro \cite{Varbaro}.

\begin{theorem}[Dao-Takagi-Varbaro Theorem]\label{thm:DTV}
If $\xi\in Z$ is a point with ${\rm codim}_X(\xi)=r$ and ${\rm depth}(\shO_{Z,\xi}) \ge 3$, then ${\rm lcd}_{\xi}(X, Z) \le r-3$.
\end{theorem}

We will point out below how one can use Theorem \ref{thm:lcd} in order to give another proof of this theorem at a closed isolated singular point, as part of studying a more general related question (see Conjecture~\ref{depth-vanishing-conj}  and Theorem~\ref{depth-vanishing}  below) that will also be useful in studying the Du Bois complex of $Z$.

\noindent
{\bf Optimal bounds.}
It is known that the ideal pattern 
$${\rm depth}(\shO_Z) \ge k  \implies {\rm lcd}(X, Z) \le n-k$$
stops in fact at $k =3$ with the result above; in characteristic $0$ there are examples of determinantal subschemes $Z \subseteq \AAA^n$ for which ${\rm depth}(\shO_Z) \ge 4$ but ${\rm lcd}(X, Z) \le n-3$, see e.g. \cite[Example 2.11]{DT}. Note that things are better in positive characteristic, where Peskine and Szpiro \cite{PS} showed that indeed for any point $\xi\in Z$, we have  ${\rm lcd}_{\xi}(X, Z) \leq n-{\rm depth}(\shO_{Z,\xi})$.

As a consequence of Theorem \ref{thm:lcd} or the surrounding circle of ideas, we do however obtain the ideal result for two classes of subschemes: those with quotient singularities, and those given by monomial ideals. We state the next result here, as it fits well with the current discussion, but for its proof it is useful to first read \S\ref{scn:DBLC}.

\begin{corollary}[Quotient singularities]\label{cor:quotient}
Let $Z$ be a closed, irreducible subscheme of codimension $r$, with quotient singularities, in the smooth, irreducible $n$-dimensional variety $X$.
Then 
$${\rm lcd}(X, Z) = n - {\rm depth}(\shO_Z) = r.$$
\end{corollary}

\begin{proof}
First, since quotient singularities are rational, hence Cohen-Macaulay, the fact that ${\rm lcd}(X, Z) \geq n - {\rm depth}(\shO_Z) = r$
follows from Remark~\ref{rmk_zero-gen}. Second, the graded pieces of the
 Du Bois complex of quotient singularities have a very simple form. Concretely, by \cite[Section 5]{DuBois} we have quasi-isomorphisms 
$$ \underline{\Omega}^i_{Z} \simeq \Omega_Z^{[i]}$$
for each $i \ge 0$, where the right hand side is the reflexive hull of $\Omega_Z^i$. 
It follows from Corollary~\ref{cor:lcd} below
that in order to show that ${\rm lcd}(X,Z)\leq r$, 
it suffices to show
\begin{equation}\label{eq_cor:quotient}
\shE xt^{j+i +1}_{\shO_X} \big(\Omega_Z^{[i]}, \omega_X\big) = 0\quad\text{for all}\quad j \ge r, i \ge 0.
\end{equation} 
But for quotient singularities we have 
$${\rm depth} \big(\Omega_Z^{[i]} \big) = \dim Z = n-r,$$
which indeed implies (\ref{eq_cor:quotient}).
In order to see this, arguing \'{e}tale locally, we may assume  that we have a morphism $\pi\colon Y \to Z = Y/G$, where $Y$ is smooth and $G$ is a group acting without quasi-reflections. In this case we have 
$$\Omega_Z^{[i]}  \simeq (\pi_* \Omega_Y^i)^G$$
by \cite[Lemma 1.8]{Steenbrink2}. This in turn is a direct summand of $\pi_* \Omega_Y^i$, since we are in characteristic zero, 
and therefore has maximal depth $n-r$. 
\end{proof}

Corollary~\ref{cor:quotient} says that the only non-trivial local cohomology sheaf is $\cH^r_Z (\shO_X)$, 
even though $Z$ is not necessarily a local complete intersection. Due to results of Ogus, this in turn implies that subvarieties in $\PP^n$ with quotient singularities share other nice features of local complete intersections (cf. \cite[Corollary 4.8]{Ogus}) regarding the global cohomological dimension ${\rm cd}(\cdot)$ of the complement, or of Barth-Lefschetz-type.

\begin{corollary}\label{like-LCI}
Let $Z\subset \PP^n$ be a closed subscheme of codimension $r$, with quotient singularities. Then 
\begin{enumerate}
\item ${\rm cd}(\PP^n \smallsetminus Z) < 2r -1$.
\item The restriction maps $H^i (\PP^n, \CC) \to H^i (Z, \CC)$ are isomorphisms for $i \le n - 2r$.
\end{enumerate}
\end{corollary}
\begin{proof}
This holds more generally for any $Z$ which is Cohen-Macaulay and satisfies ${\rm lcd}(\PP^n, Z) = r$, as an immediate consequence of \cite[Theorem 4.7]{Ogus}. The second part is stated for algebraic de Rham cohomology in \emph{loc. cit.}; over $\CC$ this is identified with singular cohomology by \cite[Theorem IV.1.1]{Hartshorne2}.
\end{proof}

\begin{remark}[Toroidal singularities]
If $Z$ has toroidal singularities, it is also the case that $ \underline{\Omega}^i_{Z} \simeq \Omega_Z^{[i]}$
for all $i$; see \cite[Chapter~V.4]{GNPP}. However it is known that the depth of $\Omega_Z^{[i]}$
is not always maximal, and therefore the argument above does not go through.\footnote{It is known that the depth is maximal for all $i$ in the case of simplicial toric varieties, but these are exactly the toric varieties that have quotient singularities.} It would be interesting to have a complete answer for the local cohomological dimension of such singularities.
\end{remark}

We next consider the case of monomial ideals, for which we recover a result of Lyubeznik (see \cite[Teorem~1(iv)]{Lyubeznik0}).

\begin{corollary}[Monomial ideals]\label{cor:monomial}
If $I \subseteq A = \CC [X_1, \ldots, X_n]$ is a radical monomial ideal (i.e. $A/I$ is a Stanley-Reisner
ring), then 
$${\rm lcd}(A, I) = n - {\rm depth}(A/I) = {\rm pd} (A/I).$$
\end{corollary}

\begin{proof}
According to Example \ref{ex:monomial}, the Hodge filtration on each $H^q_I (A)$ is generated at level $0$. 
It suffices then to apply Corollary \ref{zero-gen} and Remark~\ref{rmk_zero-gen}. 
Note that in this case $A/I$ has Du Bois singularities by \cite[Theorem~6.1]{Schwede2}, 
as Stanley-Reisner rings have $F$-injective (even $F$-pure) type. 
\end{proof}

\begin{remark}
Unlike in the case of quotient singularities, Stanley-Reisner rings are not necessarily Cohen-Macaulay, and therefore 
$ {\rm pd} (A/I)$ does not coincide in general with ${\rm codim}(I)$.
\end{remark}

\section{The Du Bois complex and reflexive differentials}\label{ch:DuBois}

Building on Deligne's work constructing mixed Hodge structures on the cohomology of algebraic varieties, Du Bois introduced 
in \cite{DuBois} a version of the de Rham complex for singular varieties. Given a (reduced) complex algebraic variety $Z$,
the \emph{Du Bois complex} of $Z$ is an object $\underline{\Omega}_Z^{\bullet}$ in the derived category of filtered complexes on $Z$. 
Its (shifted) graded pieces $\underline{\Omega}_Z^p:={\rm Gr}^p_F\underline{\Omega}_Z^{\bullet}[p]$ are objects in the derived category of coherent sheaves on $Z$.
There are canonical morphisms $\Omega_Z^p\to \underline{\Omega}_Z^p$ that are isomorphisms over the smooth locus of $Z$.  The condition that this is an isomorphism everywhere on $Z$ for $p =0$ defines Du Bois singularities, an important class of singularities (see e.g. \cite{KS2} for a survey) that 
has already made an appearance in this paper. The complexes $\underline{\Omega}_X^p$ enjoy, in the proper setting, several important properties of the De Rham complex of smooth varieties.
For an introduction to this circle of ideas, see \cite[Chapter~7.3]{PetersSteenbrink} and \cite{Steenbrink}.

\subsection{The Du Bois complex and local cohomological dimension}\label{scn:DBLC}
This is a preliminary section, in which we record some basic facts involving the Du Bois complex of $Z$, and in particular 
reinterpret the vanishing conditions in Theorem \ref{thm:lcd} in these terms. A key ingredient is a description, due to Steenbrink, for the graded pieces of the Du Bois complex of $Z$ via a log resolution of an ambient variety. Suppose that $X$ is an irreducible complex algebraic variety and
$Z$ is a reduced closed subscheme of $X$ such that $X\smallsetminus Z$ is smooth (we will shortly assume that $X$ is smooth, but we want to state the first result in this slightly more general setting).
Let $f \colon Y \to X$ be a proper map that is an isomorphism over $X \smallsetminus Z$, with $Y$ smooth, and such that  
$E = f^{-1} (Z)_{{\rm red}}$ is an SNC divisor.

\begin{theorem}[{\cite[Proposition 3.3]{Steenbrink}}]\label{Steenbrink}
With the above notation, for each $p$ we have an isomorphism in ${\bf D}^b \big({\rm Coh} (X)\big)$:
$$\derR f_* \big(\Omega^{n-p}_Y (\log E) (-E)\big) \simeq \underline{\Omega}^{n-p}_{X, Z}.$$
\end{theorem}

Here $\underline{\Omega}^{n-p}_{X, Z}$ is the $(n-p)$-th du Bois complex of the pair $(X, Z)$,  which sits in an exact triangle 
\begin{equation}\label{eqn2}
\underline{\Omega}^{n-p}_{X, Z} \to \underline{\Omega}^{n-p}_{X} \to \underline{\Omega}^{n-p}_{Z} \overset{+1}{\longrightarrow}
\end{equation}
where the other two terms are the usual du Bois complexes.

From now on, we return to our usual assumption that the ambient variety $X$ is smooth. 
In this case we of course have
$\underline{\Omega}^{n-p}_{X} \simeq \Omega^{n-p}_{X}$.

We relate $R^q f_* \Omega^p_Y (\log E)$ to the Du Bois complex as follows: by Grothendieck  duality
\begin{equation}\label{GD}
\derR f_* \Omega^p_Y (\log E) \simeq \derR f_*\big(\derR \shH om (\Omega^{n-p}_Y (\log E) (-E), \omega_Y)\big) \simeq
\end{equation}
$$\simeq \derR \shH om \big(\derR f_* \big(\Omega^{n-p}_Y (\log E) (-E)\big), \omega_X\big),$$
and therefore for each $q$ we have
\begin{equation}\label{eqn1}
R^q f_* \Omega^p_Y (\log E) \simeq \shE xt^q_{\shO_X} \big(\derR f_* \big(\Omega^{n-p}_Y (\log E) (-E)\big), \omega_X\big).
\end{equation}
As a consequence, we have:

\begin{lemma}\label{forms-DuBois}
For every $q \ge 1$ we have 
$$R^q f_* \Omega^p_Y (\log E) \simeq \shE xt^{q+1}_{\shO_X} (\underline{\Omega}^{n-p}_{Z}, \omega_X).$$
\end{lemma}
\begin{proof}
By ($\ref{eqn1}$), Theorem \ref{Steenbrink} and ($\ref{eqn2}$), both sides are isomorphic to 
the sheaf $\shE xt^q_{\shO_X} (\underline{\Omega}^{n-p}_{X, Z}, \omega_X)$.
\end{proof}

The above lemma leads to a useful equivalent formulation of Theorem \ref{thm:lcd}:

\begin{corollary}\label{cor:lcd}
For every positive integer $c$, the following are equivalent:
\begin{enumerate}
\item ${\rm lcd}(X, Z) \le c$.
\item $\shE xt^{j+i +1}_{\shO_X}(\underline{\Omega}^i_{Z}, \omega_X) = 0$ for all  $j \ge c$ and $i \ge 0$.
\end{enumerate}
\end{corollary}

\smallskip

We end this section with the discussion of the example pointed out in Remark~\ref{rmk_Ma}.

\begin{example}\label{eg_Ma}
Consider the surjection 
$$R=\CC[x_1,\ldots,x_4]\to S=\CC[s^4,s^3t,st^3,t^4],$$
with kernel $I$,
and the corresponding embedding $Z={\rm Spec}(S)\hookrightarrow X={\rm Spec}(R)$. 
Note that $Z$ is an integral variety, with a unique singular point $0$, defined by the ideal ${\mathfrak m}_0$.
The normalization morphism
$\pi\colon \widetilde{Z}\to Z$ corresponds to the inclusion $S\hookrightarrow \widetilde{S}=\CC[s^4,s^3t,s^2t^2,st^3,t^4]$. Note that 
$\widetilde{Z}$ is a toric variety, hence it is Du Bois. The morphism 
$\pi$ is an isomorphism over $Z\smallsetminus\{0\}$ and $\pi^{-1}(0)$ consists of one point. It thus follows from \cite[Proposition~3.9]{DuBois}
that $\underline{\Omega}_Z^p\simeq\pi_*\underline{\Omega}_{\widetilde{Z}}^p$ for all $p$. In particular, we have
$\underline{\Omega}_Z^0\simeq\pi_*\shO_{\widetilde{Z}}$, hence $Z$ is not Du Bois. 

On the other hand, we have $F_0H^q_I(R)=E_0H^q_I(R)$ for all $q$. This is clear for $q\neq 2$ since in this case
$H^q_I(R)=0$ (for $q=3$ one can use \cite[Corollary 2.11]{Ogus}). Using Lemma~\ref{forms-DuBois} and graded local duality,
we can identify the graded Matlis dual of the morphism
\begin{equation}\label{eq_eg_Ma}
F_0H_I^2(R)\hookrightarrow E_0H^2_I(R)
\end{equation}
with the morphism
$$H^2_{{\mathfrak m}_0}(S)\to H^2_{{\mathfrak m}_0}(\widetilde{S}).$$
This is an isomorphism, as can be seen from the long exact sequence of local cohomology associated to
$$0\to S\to \widetilde{S}\to {\mathbf C}\to 0.$$
Therefore (\ref{eq_eg_Ma}) is an isomorphism.
\end{example}

\subsection{Higher Du Bois singularities for local complete intersections}
The Hodge filtration on local cohomology allows us to state and prove a generalization of the vanishing theorem for cohomologies of the Du Bois complex of a hypersurface \cite[Theorem 1.1]{MOPW}, and of its converse \cite[Theorem 1]{Saito_et_al}, to the case of local complete intersections of arbitrary codimension.

Before proving the main result, Theorem \ref{thm-DB-main}, we make some preparations. 
We begin with a description of the graded pieces of the Du Bois complex in terms of the de Rham complex 
of the local cohomology Hodge module (cf. \cite[Lemma~2.1]{MOPW} for the case of hypersurfaces). This is all we use here, but a stronger result, at the level of the Du Bois complex, was proved by Saito in \cite[Theorem~0.2]{Saito-HC}.

\begin{proposition}\label{lem_DB_DR}
If $X$ is a smooth, irreducible $n$-dimensional complex algebraic variety and $i\colon Z\hookrightarrow X$ is the inclusion map of a closed reduced subscheme $Z$,
then for every integer $p$ we have an isomorphism in ${\bf D}^b \big({\rm Coh} (X)\big)$
$$\underline{\Omega}_Z^p\simeq \derR \shH om_{\shO_X}\big({\rm Gr}^F_{p-n}{\rm DR}_Xi_*i^!\QQ^H_X[n],\omega_X\big)[p].$$
In particular, if $Z$ is a local complete intersection of pure codimension $r$, then 
$$\underline{\Omega}_Z^p\simeq \derR \shH om_{\shO_X}\big({\rm Gr}^F_{p-n}{\rm DR}_X\shH_Z^r(\shO_X),\omega_X\big)[p+r].$$
\end{proposition}

\begin{proof}
Consider a log resolution
$f\colon Y\to X$ of $(X,Z)$ as in Section~\ref{scn:birational}.
Lemma~\ref{lem_grDR} gives an isomorphism
$${\rm Gr}^F_{p-n}{\rm DR}_X\big(j_*\QQ_U^H[n]\big)\simeq 
\derR f_*\Omega_Y^{n-p}({\rm log}\,E)[p].$$
We similarly have an isomorphism 
$${\rm Gr}^F_{p-n}{\rm DR}_X\big(\QQ^H_X[n]\big)\simeq\Omega_X^{n-p}[p].$$
Applying $ \derR \shH om_{\shO_X}(\cdot,\omega_X)$ and the isomorphisms (\ref{GD}), we obtain the commutative diagram
$$
\begin{tikzcd}
\derR \shH om_{\shO_X}\big({\rm Gr}^F_{p-n}{\rm DR}_X(j_*\QQ^H_U[n]), \omega_X\big)\dar \rar & \derR \shH om_{\shO_X}\big({\rm Gr}^F_{p-n}{\rm DR}_X\QQ_X^H[n],
\omega_X\big)\dar\\
\derR f_*\big(\Omega_Y^p({\rm log}\,E)(-E)\big)[-p] \rar & \Omega_X^p[-p],
\end{tikzcd}
$$
in which the vertical maps are isomorphisms. On the other hand, the exact triangle (\ref{eq5_loc_coh}) in the derived category
of Hodge modules induces the exact triangle
$${\rm Gr}^F_{p-n}{\rm DR}_X\big(i_*i^!\QQ_X^H[n]\big)\longrightarrow {\rm Gr}^F_{p-n}{\rm DR}_X\QQ_X^H[n]\longrightarrow 
{\rm Gr}^F_{p-n}{\rm DR}_X\big(j_*\QQ_U^H[n]\big)\overset{+1}\longrightarrow.$$
Applying $\derR\shH om_{\shO_X}(\, \cdot \,, \omega_X)$ and using the above commutative diagram, as well as Theorem~\ref{Steenbrink}, we obtain
the first assertion in the proposition. The second one is an immediate consequence.
\end{proof}

Assume now that $Z$ is a local complete intersection of pure codimension $r$. We recall from the beginning of \S\ref{section_lci} that we have
$$F_k \shH^r_Z(\shO_X) \subseteq E_k \shH^r_Z(\shO_X) \,\,\,\,\,\,{\rm  for~all}\,\,\,\, k.$$ 
It follows that for every $p\geq 0$, if
we denote by ${\rm Gr}^E_{p-n}{\rm DR}_X\shH^r_Z(\shO_X)$
the corresponding graded piece of the de Rham complex of $\shH^r_Z(\shO_X)$ with respect to the filtration $E_{\bullet}$, we have a morphism
of complexes
$$\varphi_p\colon {\rm Gr}^F_{p-n}{\rm DR}_X\shH^r_Z(\shO_X)\to {\rm Gr}^E_{p-n}{\rm DR}_X\shH^r_Z(\shO_X).$$
By definition, if $p(Z)\geq p$, then $\varphi_p$ is an isomorphism. If we only know that $p(Z)\geq p-1$, then
$\varphi_p$ is an injective morphism of complexes, whose cokernel is concentrated in cohomological degree $0$. 

Applying $\derR \shH om_{\shO_X}(\,\cdot\, ,\omega_X)$ to $\varphi_p$ and taking $\cH^{i+p+r}(-)$, we obtain
$$\psi^i_p\colon {\mathcal Ext}^{i+p+r}_{\shO_X}\big({\rm Gr}^E_{p-n}{\rm DR}_X\shH^r_Z(\shO_X),\omega_X\big)\to 
 {\mathcal Ext}^{i+p+r}_{\shO_X}\big({\rm Gr}^F_{p-n}{\rm DR}_X\shH^r_Z(\shO_X),\omega_X\big)\simeq \cH^i(\underline{\Omega}^p_Z),$$
 where the isomorphism on the right is provided by Proposition~\ref{lem_DB_DR}.

Using the description of ${\rm Gr}^E_{\bullet}\cH^r_Z(\shO_X)$ in Lemma~\ref{descriptionE}, it is easy to describe the domain of
$\psi^i_p$. Indeed, note first that $\shE^{\bullet}:={\rm Gr}^E_{p-n}{\rm DR}_X\shH^r_Z(\shO_X)$ is the complex
$$
0\to \Omega_X^{n-p}\otimes\omega_Z\otimes\omega_X^{-1}\to \Omega_X^{n-p+1}\otimes\shN_{Z/X}\otimes\omega_Z\otimes\omega_X^{-1}\to
\cdots\to \Omega_X^n\otimes {\rm Sym}^p(\shN_{Z/X})\otimes\omega_Z\otimes\omega_X^{-1}\to 0
$$
placed in cohomological degrees $-p,\ldots,0$. Recall now that if $\shF$ is a locally free $\shO_Z$-module, then 
\begin{equation}\label{eq_Ext}
{\mathcal Ext}_{\shO_X}^\ell(\shF,\omega_X)
=0\quad\text{for}\quad \ell\neq r\quad \text{and}\quad
{\mathcal Ext}^r_{\shO_X}(\shF,\omega_X)\simeq \shF^{\vee}\otimes_{\shO_Z}\omega_Z.
\end{equation}

Let us consider the hypercohomology spectral sequence
$$E_1^{k,\ell}={\mathcal Ext}_{\shO_X}^\ell(\shE^{-k},\omega_X)\Rightarrow {\mathcal Ext}^{k + \ell}_{\shO_X}(\shE^{\bullet},\omega_X).$$
Note that by (\ref{eq_Ext}), we have $E_1^{k,\ell}=0$ unless $\ell=r$ and $0\leq k\leq p$; moreover, if $0\leq k\leq p$, then
$$E_1^{k,r}\simeq \Omega_X^k\otimes_{\shO_X}{\rm Sym}^{p-k}(\shN_{Z/X})^{\vee}.$$
The spectral sequence then implies that for every integer $i$ we have
\begin{equation}\label{eq_Ext2}
{\mathcal Ext}_{\shO_X}^{i+p+r}(\shE^{\bullet},\omega_X)=E_2^{i+p,r}=\cH^i({\mathcal C}_p^{\bullet}),
\end{equation}
where ${\mathcal C}_p^{\bullet}$ is the complex 
$$0\to {\rm Sym}^p(\shN_{Z/X})^{\vee}\to \Omega_X^{1}\otimes_{\shO_X}{\rm Sym}^{p-1}(\shN_{Z/X})^{\vee}\to\cdots\to 
\Omega_X^{p-1}\otimes_{\shO_X}\shN_{Z/X}^{\vee}\to \Omega_X^p\otimes_{\shO_X}\shO_Z\to 0,$$
placed in cohomological degrees $-p,\ldots,0$. Inspection of the maps in this complex shows that, if $p\leq n-r$, then ${\mathcal C}_p^{\bullet}$
is obtained by truncating the generalized Eagon-Northcott complex ${\mathcal D}_{n-r-p}$ associated to the canonical morphism 
$$g\colon {\mathcal T}_X\vert_Z\to \shN_{Z/X},$$ 
keeping the first $p+1$ terms (and suitably translating). For basic facts about 
generalized Eagon-Northcott complexes, we refer to \cite[Chapter~2.C]{Bruns}.

The exact sequence
$$\shN_{Z/X}^{\vee}\to \Omega^1_X\vert_Z\to \Omega^1_Z\to 0$$
implies that we have a canonical isomorphism
\begin{equation}\label{eq_Ext3}
\cH^0({\mathcal C}_p^{\bullet})\simeq \Omega_Z^p.
\end{equation}
Via the isomorphisms (\ref{eq_Ext2}) and (\ref{eq_Ext3}), the morphism $\psi^0_p$ gets identified with the morphism
$\Omega_Z^p\to\cH^0(\underline{\Omega}_Z^p)$ induced by the canonical morphism $\Omega_Z^p\to 
\underline{\Omega}_Z^p$ (this follows from the fact that this holds on the smooth locus of $Z$, which is straightforward to check, and the fact that the sheaf
$\cH^0(\underline{\Omega}_Z^p)$ is torsion-free, which follows for example from the description in \cite[Theorem~7.12]{HJ}).

Finally, we recall the fact that generalized Eagon-Northcott complexes exhibit depth-sensitivity. This implies that if 
${\rm depth}(I_r(g)\big)\geq k$, for some $1\leq k\leq p$, then $\cH^i({\mathcal C}^{\bullet})=0$ for $-p\leq i\leq -p+k-1$.
Note that in our case, the ideal $I_r(g)$ defines the singular locus $Z_{\rm sing}$ of $Z$ and since $Z$ is Cohen-Macaulay,
the condition is equivalent to ${\rm codim}_Z(Z_{\rm sing})\geq k$. The depth sensitivity of generalized Eagon-Northcott complexes
is a consequence of the behavior in the case of generic matrices (see \cite[Theorem~2.16]{Bruns}) and of general properties
of perfect modules (see \cite[Proposition~2.11.2]{Kustin}). In fact, the depth sensitivity of a more general class of complexes is proved in 
\cite[Theorem~8.4]{Kustin}.

After this preparation, we can now prove our result relating the singularity level of the Hodge filtration $p(Z)$ and higher $p$-Du Bois singularities. 

\begin{proof}[Proof of Theorem~\ref{thm-DB-main}]
We follow the approach in \cite{MOPW} and \cite{Saito_et_al}, which treat the two implications in the theorem for hypersurfaces. 
We may and will assume that $Z$ is singular, since otherwise the equivalence is trivial. In this case, it follows from
Theorem~\ref{thm_upper_bound} that $p(Z)\leq\frac{\dim(Z)-1}{2}$, and it is easy to see that we may assume that $p\leq\dim(Z)=n-r$.

Suppose first that $p(Z)\geq p$. 
In this case the morphism $\varphi_p$ is an isomorphism, hence so are the morphisms $\psi_p^i$. Since the complex ${\mathcal C}_p^{\bullet}$ is supported
in nonpositive cohomological degrees, we have $\cH^i({\mathcal C}_p^{\bullet})=0$ for $i>0$. It thus follows from
(\ref{eq_Ext2}) that
$$\cH^i(\underline{\Omega}_Z^p)\simeq {\mathcal Ext}^{i+p+r}_{\shO_X}({\mathcal E}^{\bullet},\omega_X)\simeq \cH^i({\mathcal C}_p^{\bullet})=0
\quad\text{for}\quad i>0.$$
We also see that $\psi_p^0$ induces an isomorphism
$$\Omega_Z^p\simeq \cH^0({\mathcal C}^{\bullet})\simeq {\mathcal Ext}^{p+r}_{\shO_X}({\mathcal E}^{\bullet},\omega_X)\to 
\cH^0(\underline{\Omega}_Z^p).$$
This implies that the canonical morphism $\Omega_Z^p\to \underline{\Omega}_Z^p$ is an isomorphism. Since the same argument applies if we replace
$p$ by any $k$, with $0\leq k\leq p$, we conclude that $Z$ has at most higher $p$-Du Bois singularities. 

Conversely, suppose now that $Z$ has at most higher $p$-Du Bois singularities. Arguing by induction on $p$, it follows that we may assume that 
$p(Z)\geq p-1$. In this case the morphism $\varphi_p$ is injective and its cokernel is 
$$M=\omega_X\otimes_{\shO_X}\big(E_p\cH_Z^r(\shO_X)/F_p\cH_Z^r(\shO_X)\big),$$
placed in cohomological degree $0$. We need to show that $M=0$.
Since the functor $\derR \shH om_{\shO_X}\big(-,\omega_X[n]\big)$ is a duality, it is enough to show that 
$\derR \shH om_{\shO_X}\big(M,\omega_X)=0$, that is, we have
$${\mathcal Ext}^{i}_{\shO_X}(M,\omega_X)=0\quad\text{for all}\quad i\in\ZZ.$$

Let us put
$$\shE^{\bullet}:={\rm Gr}^E_{p-n}{\rm DR}_X\shH^r_Z(\shO_X)\quad \text{and}\quad 
\shF^{\bullet}:={\rm Gr}^F_{p-n}{\rm DR}_X\shH^r_Z(\shO_X).$$
The exact triangle 
$$\derR \shH om_{\shO_X}(M,\omega_X)\longrightarrow \derR \shH om_{\shO_X}(\shE^{\bullet},\omega_X)\longrightarrow  \derR \shH om_{\shO_X}(\shF^{\bullet},\omega_X)\overset{+1}\longrightarrow$$
induces a long exact cohomology sequence
$$\ldots\to {\mathcal Ext}^{i-1}_{\shO_X}(\shF^{\bullet},\omega_X)\to {\mathcal Ext}^{i}_{\shO_X}(M,\omega_X)\to {\mathcal Ext}^{i}_{\shO_X}(\shE^{\bullet},\omega_X)
\to {\mathcal Ext}^{i}_{\shO_X}(\shF^{\bullet},\omega_X)\to\ldots.$$
As we have seen, the fact that the morphism $\Omega_Z^p\to \underline{\Omega}_Z^p$ is an isomorphism translates as saying
that ${\mathcal Ext}^{i}_{\shO_X}(\shF^{\bullet},\omega_X)=0$ for all $i\neq p+r$ and the morphism
$${\mathcal Ext}^{p+r}_{\shO_X}(\shE^{\bullet},\omega_X)\to {\mathcal Ext}^{p+r}_{\shO_X}(\shF^{\bullet},\omega_X)$$
is an isomorphism. It follows from the long exact sequence that
${\mathcal Ext}_{\shO_X}^{p+r}(M,\omega_X)=0$ and
$${\mathcal Ext}_{\shO_X}^{i}(M,\omega_X)\simeq {\mathcal Ext}_{\shO_X}^{i}(\shE^{\bullet},\omega_X)\quad\text{for all}\quad i\neq p+r.$$
Therefore thanks to (\ref{eq_Ext2}) we are done if we show that $\cH^i({\mathcal C}_p^{\bullet})=0$ for $i\neq 0$.
This is trivial if $p=0$. On the other hand, if $p\geq 1$, recalling that we are assuming $p(Z)\geq p-1$ by induction,
Corollary~\ref{cor_upper_bound} gives
$${\rm codim}_Z(Z_{\rm sing})\geq 2p(Z)+1\geq 2p-1\geq p.$$
The depth-sensitivity of the generalized Eagon-Northcott complexes thus implies that $\cH^i({\mathcal C}_p^{\bullet})=0$ for $i\neq 0$,
which completes the proof of the theorem.
\end{proof}

We also note the following corollary of the proof of Theorem \ref{thm-DB-main}:

\begin{corollary}
Suppose  that $Z$ has higher $p$-du Bois singularities, with $p \ge 1$. Then $Z$ is normal and $\Omega_Z^p$ is reflexive, i.e. $\Omega_Z^p\simeq (\Omega_Z^p)^{\vee\vee}$.
\end{corollary}

\begin{proof}
Since ${\rm codim}_Z(Z_{\rm sing})\geq 2p+1\geq 3$
by Corollary~\ref{cor_upper_bound} and $Z$ is Cohen-Macaulay, it follows that $Z$ is normal. 
In order to see that $\Omega_Z^p$ is reflexive, it is enough then to show that ${\rm depth}(J,\Omega_Z^p)\geq 2$,
where $J$ is the ideal defining $Z_{\rm sing}$ in $Z$. Using again the fact that ${\rm codim}_Z(Z_{\rm sing})\geq 2p+1\geq p$,
it follows that the complex ${\mathcal C}_p^\bullet$ used in the proof of Theorem \ref{thm-DB-main} gives a locally free resolution of $\Omega_Z^p$. The well-known behavior of depth 
in short exact sequences together with the fact that ${\rm depth}(J,{\mathcal C}_p^j)={\rm depth}(J,\shO_Z)\geq 2p+1$ for all $j$ implies
$${\rm depth}(J,\Omega_Z^p)\geq (2p+1)-p=p+1\geq 2.$$
\end{proof}

We next give a refinement of the vanishing statement for the higher cohomology of the graded pieces of the Du Bois complex in Theorem~\ref{thm-DB-main},
in terms of the dimension of the locus where $p(Z)$ is small.

\begin{theorem}\label{thm_van_DB}
If for some $p  \ge 0$ we have $F_p \shH^r_Z \shO_X = E_p \shH^r_Z \shO_X$ away from a closed subset $W\subseteq Z$ of dimension $s$ (with the convention that $s = -\infty$ if $W$ is empty), then 
$$\shH^i(\underline{\Omega}_Z^p) = 0 \,\,\,\,\,\,{\rm for ~all} \,\,\,\,\,\,  0< i < \dim Z - p - s  -1 .$$
\end{theorem}

\begin{proof}
We keep the setup used in the proof of Theorem~\ref{thm-DB-main}. Note that by assumption, the morphism 
$$\varphi_p\colon \shF^{\bullet}={\rm Gr}^F_{p-n}{\rm DR}_X\shH^r_Z(\shO_X)\to \shE^{\bullet}={\rm Gr}^E_{p-n}{\rm DR}_X\shH^r_Z(\shO_X)$$
is an isomorphism away from $W$.
Recall that by (\ref{eq_Ext2}), we have
\begin{equation}\label{eq1_thm_van_DB}
{\mathcal Ext}^{j+p+r}_{\shO_X}(\shE^{\bullet},\omega_X)=0\quad\text{for all}\quad j>0.
\end{equation}

On the other hand, if $\shG^{\bullet}$ is a complex on $X$ concentrated in degrees $\leq 0$ such that ${\rm Supp}(\shG^q)\subseteq W$ for all $q$, then 
$\shE xt_{\shO_X}^m(\shG^{\bullet},\omega_X)=0$ if $m<n-s$. This can be proved inductively by considering the long exact sequences of 
$\shE xt$ sheaves associated to the short exact sequences of complexes
$$0\to \sigma^{\leq j-1}(\shG^{\bullet})\to \sigma^{\leq j}(\shG^{\bullet})\to \shG^j[-j]\to 0$$
and the fact that $\shExt_{\shO_X}^m(\shG^j,\omega_X)=0$ for $m<n-s$ (see Lemma~\ref{lem_BS}). Here we denote by
$ \sigma^{\leq j}(\shG^{\bullet})$ the ``stupid" truncation of $\shG^{\bullet}$ consisting of the terms placed in cohomological degrees $\leq j$. 

This applies in particular to the complexes ${\rm ker}(\varphi)$ and ${\rm coker}(\varphi)$. Since $i+p+r+1<n-s$, we conclude that
\begin{equation}\label{eq2_thm_van_DB}
\shE xt^{i+p+r+1}_{\shO_X}\big({\rm coker}(\varphi), \omega_X\big)=0=\shE xt^{i+p+r}_{\shO_X}\big({\rm ker}(\varphi), \omega_X\big).
\end{equation}
The exact sequences
$$\shE xt^{i+p+r}_{\shO_X}(\shE^{\bullet},\omega_X)\to \shE xt^{i+p+r}_{\shO_X}\big({\rm im}(\varphi),\omega_X\big)\to
 \shE xt^{i+p+r+1}_{\shO_X}\big({\rm coker}(\varphi),\omega_X\big)=0$$
 and 
 $$ \shE xt^{i+p+r}_{\shO_X}\big({\rm im}(\varphi),\omega_X\big)\to
 \shE xt_{\shO_X}^{i+p+r}(\shF^{\bullet},\omega_X)\to 
 \shE xt^{i+p+r}_{\shO_X}\big({\rm ker}(\varphi), \omega_X)=0,$$
 together with the vanishing in (\ref{eq1_thm_van_DB}) imply
 $$\cH^i(\underline{\Omega}_Z^p)\simeq \shE xt_{\shO_X}^{i+p+r}(\shF^{\bullet},\omega_X\big)=0,$$
 where the first isomorphism follows from Proposition~\ref{lem_DB_DR}.
 This completes the proof of the theorem.
\end{proof}

We conclude that there is a range of automatic vanishing in terms of the dimension of the singular locus of $Z$; when $Z$
is a hypersurface, this is \cite[Corollary 3.5]{MOPW}.

\begin{corollary}
If the singular locus of the local complete intersection $Z$ has dimension $s$, then for all $p \ge 0$ we have
$$\cH^i(\underline{\Omega}_Z^p) =0 \,\,\,\,\,\,{\rm for} \,\,\,\,\,\,1\leq i < \dim Z - s - p -1.$$
\end{corollary}

\begin{question}
\emph{Does this result continue to hold when $Z$ is an arbitrary (or at least Cohen-Macaulay) closed subscheme whose singular locus has dimension $s$}?
\end{question}

\subsection{Depth and local vanishing}
We return to the general setting of a reduced closed subscheme $Z$ of a smooth variety $X$.
We have seen in Theorem~\ref{n-1} that if $n\geq 2$, then 
$R^{n-1} f_* \Omega^j_Y (\log E) = 0$ for all $j<n$, and related this to the condition ${\rm lcd}(X, Z) \le n-2$. For lower values of  ${\rm lcd}(X, Z)$ it is also important to understand the
vanishing of $R^{n-2}f_*\Omega_Y^j(\log E)$.  We make the following conjecture based on the depth of $\shO_Z$
at its closed points.

\begin{conjecture}\label{depth-vanishing-conj}
If ${\rm depth} (\shO_Z) \ge k + 2$, then $R^{n-2} f_* \Omega_Y^{n-k} (\log E) = 0$.
\end{conjecture}

\begin{remark}\label{rmk:k1}
Conjecture \ref{depth-vanishing-conj} holds for $k = 0, 1$.
Indeed, by Corollary \ref{KS-injection} we have an inclusion 
$$R^{n-2} f_* \omega_Y (E) \hookrightarrow \shE xt^{n-1}_{\shO_X} (\shO_Z, \omega_X),$$
which gives the assertion for $k =0$. For $k =1$, it is a consequence of Theorem \ref{thm:DTV} and Theorem \ref{thm:lcd}.
\end{remark}

The main result of this section is a proof of Conjecture~\ref{depth-vanishing-conj}
when $Z$ has isolated singularities. Reversing the use of Theorem \ref{thm:lcd} in the above remark, 
this gives an alternative proof of Theorem \ref{thm:DTV} at closed isolated singular points.  Further consequences are explained in \S\ref{scn:hdiff}.

\begin{theorem}\label{depth-vanishing}
If ${\rm depth} (\shO_Z) \ge k + 2$ and $Z$ has isolated singularities, then 
$$R^{n-2} f_* \Omega_Y^{n-k} (\log E) = 0.$$
\end{theorem}

We first  need some preparations, starting with a lemma that will also be used in the next section. We recall that the theory of depth admits an extension to 
complexes. If $(R,{\mathfrak m})$ is the local ring of $X$ at a closed point $x\in X$ and
$M$ is an element of the bounded derived category of $R$-modules, we put
$${\rm depth}(M):=n-\max\{i\mid {\rm Ext}_R^i(M,R)\neq 0\big\}=\min\big\{i\mid H_{\mathfrak m}^i(M)\neq 0\}$$
(with the convention that this is $\infty$ if $M=0$).
For a proof of the above equality, as well as for other properties of the depth of complexes, see \cite{Iyengar}. 
If $M$ is an element of the bounded derived category of coherent sheaves on $X$, we put
$${\rm depth}(M)=\min_{x\in X}{\rm depth}(M_x),$$
where the minimum is over all closed points of $X$.

\begin{lemma}\label{lem:depth2}
If $Z$ is reduced and $\dim Z \ge 2$, then for any $k \ge 0$ we have an equivalence
$$R^{n-2} f_* \Omega_Y^{n-k} (\log E) = 0 \iff \depth (\underline{\Omega}^k_{Z}) \ge 2,$$
and either condition implies
$$\depth(\cH^0 \underline{\Omega}^k_{Z}) \ge 2.$$ 
\end{lemma}

\begin{proof}
Since $n$ must be at least $3$, by Lemma~\ref{forms-DuBois} we have 
$$R^{n-2} f_* \Omega_Y^{n-k} (\log E)  \simeq \shE xt^{n-1} (\underline{\Omega}^k_{Z}, \omega_X),$$
and the vanishing of the latter is equivalent to  $\depth(\underline{\Omega}^k_{Z}) \ge 2$.
For the last statement, we consider the spectral sequence
$$E^{i,j}_2 = \shE xt^i _{\shO_X}(\shH^{-j} \underline{\Omega}^k_{Z}, \omega_X) \implies \shE xt_{\shO_X}^{i + j} (\underline{\Omega}^k_{Z}, \omega_X).$$
Note that $E^{i,j}_2=0$ if $i>n$ or $j>0$; in the latter case, this follows from the
 well known fact that $\shH^\ell \underline{\Omega}^k_{Z} = 0$ for $\ell < 0$. The term
$E^{n-1,0}_2 = \shE xt_{\shO_X}^{n-1} (\shH^ 0 \underline{\Omega}^k_{Z}, \omega_X)$ contributes to computing 
$\shE xt_{\shO_X}^{n-1} (\underline{\Omega}^k_{Z}, \omega_X)$, and since the differentials at each level are
$$E^{n-1-r , r-1}_r \to  E^{n-1,0}_r \to E^{n-1 + r, 1-r}_r,$$
it follows that $E^{n-1,0}_2 = E^{n-1,0}_{\infty}$. Therefore 
$$ \shE xt^{n-1}_{\shO_X} (\shH^ 0 \underline{\Omega}^k_{Z}, \omega_X) =0,$$
which is equivalent to $ \depth(\cH^0 \underline{\Omega}^k_{Z}) \ge 2$.
\end{proof}

We now set up some notation under the hypothesis of Theorem~\ref{depth-vanishing}. Since $Z$
has isolated singularities and ${\rm depth}(\shO_Z)\geq 2$, it follows that $Z$ is normal. Without loss of generality, we may assume that $Z$ is irreducible
and $P\in Z$ is a point such that $Z\smallsetminus\{P\}$ is smooth. Let $g\colon \widetilde{Z}\to Z$ be a projective morphism that is an isomorphism over
$Z\smallsetminus\{P\}$, with $\widetilde{Z}$ smooth, and such that $D:=g^{-1}(P)_{\rm red}$ is an SNC divisor.

According to the first statement the Lemma \ref{lem:depth2}, in order to prove Theorem \ref{depth-vanishing}, it suffices to show that under its hypotheses we have 
$$\depth (\underline{\Omega}^k_{Z}) \ge 2.$$
We already know that this depth is $\geq 1$, hence it is enough to show that 
$H^1_{{\mathfrak m}_P}(\underline{\Omega}^k_{Z})=0$, where ${\mathfrak m}_P$ is the maximal ideal defining $P$. 
We may further assume that $Z$ is affine. By Remark~\ref{rmk:k1}, we may assume that $k\geq 1$.
 Using Theorem~\ref{Steenbrink} and the exact triangle (\ref{eqn2}), we have an isomorphism
$$\underline{\Omega}_Z^k\simeq \derR g_*\Omega_{\widetilde{Z}}^k({\rm log}\,D)(-D),$$
hence in order to prove the theorem it suffices to show that 
\begin{equation}\label{eqn:reduction}
H^1_D\big(\widetilde{Z}, \Omega_{\widetilde{Z}}^k({\rm log}\,D)(-D)\big) = 0.
\end{equation}

\medskip

The rest of the section is devoted to proving this statement. The proof is independent from the rest of the paper, and involves some basic mixed Hodge theory.

The condition on ${\rm depth}(\shO_Z)$ will be used via the following lemma. 

\begin{lemma}\label{conseq_depth}
If $Z$ is a variety with isolated singularities and ${\rm depth}(\shO_Z)\geq k+2$, then
$$R^qg_*\shO_{\widetilde{Z}}=0\quad\text{for}\quad 1\leq q\leq k.$$
\end{lemma}

\begin{proof}
We follow the approach of \cite[Proposition~4.3]{ABW}, but include the proof since in \emph{loc. cit.} the authors make the stronger assumption that $Z$ is Cohen-Macaulay.\footnote{Note that in the Cohen-Macaulay case an even stronger statement appears in \cite{Kovacs1}.} Since $Z$ has isolated singularities, it is easy to see that we may assume that $Z$ is projective. In this case, if $\shL$ is an ample
line bundle on $Z$, then Kawamata-Viehweg vanishing gives
\begin{equation}\label{eq1_conseq_depth}
H^i\big(\widetilde{Z},g^*\shL^{-m}\big)=0\quad\text{for all}\quad i<d=\dim(Z),~ m>0.
\end{equation}
For a given $m>0$, consider the Leray spectral sequence
\begin{equation}\label{eqn:SS}
E_2^{p,q}=H^p\big(Z,R^qg_*\shO_{\widetilde{Z}} \otimes\shL^{-m}\big)\implies H^{p+q}\big(\widetilde{Z},g^*\shL^{-m}\big).
\end{equation}
If $q>0$, then $R^qg_*(\shO_{\widetilde{Z}})$ has $0$-dimensional support, hence $E_2^{p,q}=0$ for $p>0$.
On the other hand, since our assumptions imply that $Z$ is normal, we have $g_* \shO_{\widetilde{Z}}=\shO_Z$, and since 
${\rm depth}(\shO_Z)\geq k+2$, it follows that if $m\gg 0$, then for all $p\leq k+1$ we have
$$E_2^{p,0}=H^p\big(Z,\shL^{-m})\simeq H^{d-p}\big(Z, \shL^m \otimes \omega_Z^\bullet [-d]\big)^{\vee}=0.$$
Indeed, for the vanishing on the right, note that the depth hypothesis implies that the object $A^\bullet = \omega_Z^\bullet [-d]$ can have cohomologies only in degrees $0, \ldots, d - k -2$. On the other hand, the hypercohomology group $H^{d-p}\big(Z, \shL^m \otimes \omega_Z^\bullet [-d]\big)$ is computed by a spectral sequence whose contributing $E_2$-terms are 
$H^i (Z, \shL^m \otimes\cH^j A^\bullet)$, with $i + j = d - p$. Since $m \gg 0$, the only term that could contribute corresponds to 
$i =0$, but in this case $\cH^{d - p} A^\bullet = 0$ since $p \le k+1$.

Going back to the spectral sequence in ($\ref{eqn:SS}$), we keep the assumption that $m\gg 0$.
Note that if $r\geq 2$, then $d_r\colon E_r^{0,q}\to E_r^{r,q-r+1}$ 
is $0$ unless $r=q+1>k+1$, while 
$d_r\colon E_r^{-r,q+r-1}\to E_r^{0,q}$ is always $0$. We thus conclude that if $q\leq k$, then
$E_{\infty}^{0,q}=E_2^{0,q}$ and this vanishes because of (\ref{eq1_conseq_depth}).
If we assume in addition that $q\geq 1$, then we conclude that $R^qg_*\shO_{\widetilde{Z}}=0$, since this sheaf
has $0$-dimensional support.
\end{proof}

\begin{corollary}\label{fiber_cohomology}
If $Z$ is a variety with isolated singularities and ${\rm depth}(\shO_Z)\geq k+2$, then
$$H^q(D,\shO_D)=0\quad\text{for}\quad 1\leq q\leq k.$$
\end{corollary}

\begin{proof}
It is shown in the proof of \cite[Lemma~1.2]{Namikawa} that, for each $q$, the vanishing of $R^q g_* \shO_{\widetilde{Z}}$ implies
$H^q (D,\shO_D)=0$ (the statement in \emph{loc. cit.} assumes that $Z$ has rational singularities, but the proof gives in fact this implication for any variety with isolated singularities).
\end{proof}

We can now prove the main result; part of the proof follows an argument in \cite{SvS}.

\begin{proof}[Proof of ${\rm (}\ref{eqn:reduction}{\rm )}$]
We write $D=\sum_{i=1}^ND_i$, and define 
$$\Omega^k_{D^{(q)}}:=\bigoplus_{i_1<\cdots<i_q}\Omega^k_{D_{i_1}\cap\ldots\cap D_{i_q}}.$$
We then have 
an exact complex on $\widetilde{Z}$:
$$C_k^{\bullet}:\,\,\,\,0\to \Omega_{\widetilde{Z}}^k({\rm log}\,D)(-D)\to\Omega_{\widetilde{Z}}^k\to \Omega^k_{D^{(1)}}\to \Omega^k_{D^{(2)}}\to\cdots$$
(for example, the case $k=1$ is treated in \cite[Lemma~4.1]{MOP}, but the proof therein extends to arbitrary $k$). 
Let 
$$\Mmod_k :={\rm ker}\big(\Omega^k_{D^{(1)}}\to \Omega^k_{D^{(2)}}\big).$$

\noindent 
\emph{Claim:} We have
 $$H^1_D\big(\widetilde{Z},\Omega_{\widetilde{Z}}^k({\rm log}\,D)(-D)\big)=0 \iff  H^0(\widetilde{Z}, \Mmod_k)=0.$$

In order to prove this, note first that the short exact sequence
$$0\to \Omega_{\widetilde{Z}}^k({\rm log}\,D)(-D)\to\Omega_{\widetilde{Z}}^k\to\Mmod_k\to 0$$
gives a long exact sequence of local cohomology
$$0=H^0_D(\widetilde{Z},\Omega_{\widetilde{Z}}^k)\to H^0_D(\widetilde{Z},\Mmod_k)\to H^1_D\big(\widetilde{Z}, \Omega_{\widetilde{Z}}^k({\rm log}\,D)(-D)\big)
\to H_D^1(\widetilde{Z},\Omega_{\widetilde{Z}}^k)\to H_D^1(\widetilde{Z},\Mmod_k).$$
Of course, $\Mmod_k$ is supported on $D$, hence
$$H^0_D(\widetilde{Z},\Mmod_k)=H^0(\widetilde{Z},\Mmod_k).$$
The claim thus follows if we show that the map
\begin{equation}\label{eq1_prop_equiv1}
H_D^1(\widetilde{Z},\Omega_{\widetilde{Z}}^k)\to H_D^1(\widetilde{Z},\Mmod_k)
\end{equation}
is injective. This is essentially shown in the proof of \cite[Theorem~1.3]{SvS}, but since it is not stated there in this form, we recall the main steps in the argument.

First, Steenbrink's vanishing theorem (see \cite[Theorem~2]{Steenbrink}) gives
$$R^{d-1}g_* \big( \Omega_{\widetilde{Z}}^{d-k}({\rm log}\,D)(-D)\big)=0,$$
where $d=\dim(Z)$ (note that $k\leq d-2$). Using the Local Duality theorem and relative duality, we deduce from this that
\begin{equation}\label{eq1_prop_equiv2}
H_D^1\big(\widetilde{Z},\Omega^k_{\widetilde{Z}}({\rm log}\,D)\big)=H_{{\mathfrak m}_P}^1\big(\derR g_*\Omega_{\widetilde{Z}}^k({\rm log}\,D)\big)=0.
\end{equation}
Second, it is shown in \cite[p.~99]{SvS} via an argument using mixed Hodge structures that if
$$\Nmod_k=\Omega^k_{\widetilde{Z}}({\rm log}\,D)/\Omega^k_{\widetilde{Z}}({\rm log}\,D)(-D),$$
then the inclusion $\Mmod_k\hookrightarrow \Nmod_k$ induces an isomorphism $H^0(\widetilde{Z},\Mmod_k)\simeq
H^0(\widetilde{Z},\Nmod_k)$. Note that we have a commutative diagram with exact rows
$$
\begin{tikzcd}
0 \rar & \Omega_{\widetilde{Z}}^k\rar \dar& \Omega_{\widetilde{Z}}^k({\rm log}\,D)\rar\dar & {\mathcal Q}_k\dar{\rm Id}\rar & 0\\
0 \rar & \Mmod_k \rar & \Nmod_k \rar & {\mathcal Q}_k \rar & 0.
\end{tikzcd}
$$
By considering the connecting homomorphisms in the long exact sequences of cohomology and local cohomology, we obtain
the commutative diagram
$$
\begin{tikzcd}
H_D^0(\widetilde{Z},{\mathcal Q}_k) \dar{\alpha}\rar{\beta} & H^1_D(\widetilde{Z}, \Omega_{\widetilde{Z}}^k)\dar{\gamma} \\
H^0(\widetilde{Z},{\mathcal Q}_k)\dar{\rm Id} \rar & H^1(\widetilde{Z}, \Omega_{\widetilde{Z}}^k)\dar{\delta}\\
H^0(\widetilde{Z},{\mathcal Q}_k)\rar{\rho} & H^1(\widetilde{Z},\Mmod_k).
\end{tikzcd}
$$
Note that $\alpha$ is an isomorphism since ${\mathcal Q}_k$ is supported on $D$ and $\beta$ is an isomorphism by (\ref{eq1_prop_equiv2}). Since $H^0(\widetilde{Z},\Mmod_k)\to
H^0(\widetilde{Z},\Nmod_k)$ is an isomorphism, it follows that $\rho$ is injective. We thus conclude from the above diagram that the composition
$\delta\circ\gamma$ is injective and this gets identified with the map (\ref{eq1_prop_equiv1}) since ${\mathcal M}_k$ is supported on $D$.
This completes the proof of the claim.

Our goal is therefore to show that $H^0(\widetilde{Z}, \Mmod_k)=0$.
If we apply $H^q(-)$  to the complex $C_p^{\bullet}$ and ignore the first two terms, we obtain the following complex
$$0\longrightarrow H^q\big(\widetilde{Z}, \Omega^p_{D^{(1)}}\big)\overset{d^{p,q}_1}\longrightarrow H^q\big(\widetilde{Z}, \Omega^p_{D^{(2
)}}\big)\overset{d^{p,q}_2}\longrightarrow\cdots.$$
With this notation, for every $p$ and $q$, the $(p,q)$ component of the pure Hodge structure on the graded piece 
${\rm Gr}^W_{p+q}H^{p+q+i}(D,{\mathbf C})$  is given by 
$${\rm Gr}^W_{p+q}H^{p+q+i}(D,{\mathbf C})^{p,q}={\rm ker}(d^{p,q}_{i+1})/{\rm im}(d^{p,q}_i)$$
(with the convention that $d^{p,q}_i=0$ for $i\leq 0$); 
see \cite[Part II, 1]{ElZein} for a detailed description of the mixed Hodge structure
on the cohomology of $D$. 

We thus see that in order to complete the proof, it is enough to show that 
$${\rm Gr}^W_kH^k(D,{\mathbf C})^{k,0}=0.$$
By Hodge symmetry, this is equivalent to
\begin{equation}\label{eq_depth_vanishing}
{\rm Gr}^W_kH^k(D,{\mathbf C})^{0,k}=0.
\end{equation}
On the other hand, we have
$${\rm Gr}_F^0H^k(D,{\mathbf C})\simeq H^k(D,\shO_D)=0,$$
where the vanishing follows from Corollary~\ref{fiber_cohomology}.
Since $D$ is compact, we have the identification $W_kH^k(D,\CC)=H^k(D,\CC)$, 
so that we have a surjective morphism
$${\rm Gr}_F^0H^k(D,{\mathbf C})\to {\rm Gr}_F^0{\rm Gr}^W_kH^k(D,{\mathbf C})={\rm Gr}^W_kH^k(D,{\mathbf C})^{0,k}.$$
We thus obtain the vanishing in (\ref{eq_depth_vanishing}) and this completes the proof.
\end{proof}

\subsection{$h$-differentials and reflexive differentials}\label{scn:hdiff}
The characterization of local cohomological dimension and the local vanishing results in the previous sections have consequences regarding $h$-differentials on singular spaces.  Recall that a recent result of Kebekus-Schnell \cite[Corollary 1.12]{KeS}
states that if $Z$ is a variety with rational singularities, then for all $p$ the $h$-differentials ${\Omega_h^p}|_{Z}$ of \cite{HJ}  coincide with the reflexive differentials 
$\Omega_Z^{[p]} := (\Omega_Z^p)^{\vee \vee}$. 

Here we show, using Theorem \ref{depth-vanishing}, that when $Z$ has isolated singularities this holds under a weaker hypothesis, at least for forms of low degree. 

\begin{proof}[Proof of Theorem \ref{thm:hisolated}]
The sheaves of $h$-differentials ${\Omega_h^k}|_{Z}$ are identified in \cite[Theorem 7.12]{HJ} with $\cH^0 \underline{\Omega}^k_{Z}$, hence equivalently we will show that the natural morphism
$$\cH^0 \underline{\Omega}^k_{Z} \to \Omega_Z^{[k]}$$
in an isomorphism under the hypothesis of the theorem. Since the two sheaves are isomorphic on the smooth locus of 
$Z$, it suffices in turn to show that the sheaf $\cH^0 \underline{\Omega}^k_{Z}$ satisfies the $S_2$ property. Equivalently, 
if $\I$ is the ideal sheaf of the singular locus of $Z$, we need to show that 
\begin{equation}\label{desire}
\depth (\I, \cH^0 \underline{\Omega}^k_{Z}) \ge 2.
\end{equation}

Since $Z$ has isolated singularities, using Theorem~\ref{depth-vanishing} we deduce that the hypothesis implies
$R^{n-2} f_* \Omega_Y^{n-k} (\log E) = 0$. Lemma~\ref{lem:depth2} then gives
$$
\depth (\cH^0 \underline{\Omega}^k_{Z}) \ge 2.
$$

By Auslander-Buchsbaum it follows that for each 
$x\in Z$, in some open neighborhood of $x$ we can find a resolution with locally free $\shO_X$-modules of finite rank
$$0 \to \shF_{n-2} \to \cdots \to \shF_0 \to \cH^0 \underline{\Omega}^k_{Z} \to 0.$$
Since the singular locus of $Z$ is $0$-dimensional, for each $i$ we have $\depth (\I, \shF_i) = n$.
If we denote 
$$\shG_{i-1} : = {\rm Coker} ( \shF_i \to \shF_{i-1}),$$
then it follows using basic properties of depth that $\depth (\I, \shG_{i-1}) \ge i +1$ for $0\leq i\leq n-1$. In particular we obtain that 
$\cH^0 \underline{\Omega}^k_{Z} = \shG_0$ satisfies ($\ref{desire}$).
This completes the proof of the theorem.
\end{proof}

\begin{remark}\label{rmk:extension}
It is known that if $Z$ is normal, then the isomorphism  ${\Omega_h^k}|_{Z} \simeq \Omega_Z^{[k]}$ implies that 
the $k$-forms on $Z_{\rm reg}$ extend to $k$-forms on any resolution of singularities $\pi \colon \widetilde{Z} \to Z$. 
This consequence of Theorem \ref{thm:hisolated} is known in greater generality: van Straten 
and Steenbrink \cite[Theorem 1.3]{SvS} have shown that if $Z$ is any variety with isolated singularities, then $k$-forms extend for $k \le \dim Z - 2$.
\end{remark}

\begin{remark}
When $k = 0$, we have in fact that if ${\rm depth} (\shO_Z) \ge 2$ and $Z$ is Du Bois away from a finite set of points, then the canonical morphism 
$$\shO_Z \to \cH^0 \underline{\Omega}^0_{Z} \simeq  {\Omega_h^0}|_{Z}$$
is an isomorphism, or in other words $Z$ is weakly normal. Indeed, in this case we have a short exact sequence
$$0 \to \shO_Z \to \cH^0 \underline{\Omega}^0_{Z} \to \tau \to 0,$$
where $\tau$ is supported in dimension $0$, and hence has depth $0$ if it is nonzero. On the other hand, if ${\rm depth} (\shO_Z) \ge 2$, then $R^{n-2} f_* \omega_Y (E) = 0$ (see Remark \ref{rmk:k1}), hence again by Lemma~\ref{lem:depth2} we have ${\rm depth}(\cH^0 \underline{\Omega}^0_{Z}) \ge 2$. The only way this can happen is to have $\tau = 0$.
\end{remark}

\medskip

Another source for the type of vanishing needed in the proof of Theorem \ref{thm:hisolated}  is Corollary \ref{explicit-vanishing}. We record the special case $p = n-k$ and $q = n-2$ needed here:

\begin{lemma}\label{lem:c}
Assuming that $Z$ has codimension $r$, we have 
$$R^{n-2} f_* \Omega_Y^{n-k} (\log E) = 0 \,\,\,\,\,\,{\rm for}\,\,\,\, k \le \left[ \frac{n-1}{r}\right] - 2.$$
Moreover, if $Z$ is normal, the same holds for 
$$k \le \left[ \frac{n}{r+1}\right]  +  \left[ \frac{n-1}{r+1}\right] - 2.$$
\end{lemma}

In general this only tells us about the depth of ${\Omega_h^k}|_{Z}$ at closed points, which is only a step towards the stronger $S_2$-property.  When $Z$ has isolated singularities however, the stronger statement is true, as in 
Theorem \ref{thm:hisolated}.

\begin{corollary}\label{cor:DT}
With the notation above we have the following:

\noindent
(i) If ${\rm depth} (\shO_Z) \ge 3$, then  $\depth({\Omega_h^1}|_{Z}) \ge 2$.

\noindent
(ii) If $Z$ has codimension $r$ and $k \le \left[ \frac{n-1}{r}\right] - 2$, then $\depth({\Omega_h^k}|_{Z}) \ge 2$. The same conclusion holds if $Z$ is in addition normal, and $k \le \left[ \frac{n}{r+1}\right]  +  \left[ \frac{n-1}{r+1}\right] - 2$.

\noindent
(iii) Under the assumptions of $(ii)$, if $Z$ has isolated singularities, then ${\Omega_h^k}|_{Z} \simeq \Omega_Z^{[k]}$.
\end{corollary}
\begin{proof}
This follows from Lemma \ref{lem:depth2}, Remark \ref{rmk:k1} and Lemma \ref{lem:c}. For (iii) we use in addition the same argument as in the proof of Theorem \ref{thm:hisolated}.
\end{proof}


\end{document}